\documentclass[12pt]{amsart}
\usepackage[left=1in,top=1in,right=1in,bottom=1in,letterpaper]{geometry}
\usepackage[all]{xy}
\usepackage{amsmath}
\usepackage{amssymb}
\usepackage{amscd}
\usepackage{bm}

\usepackage{url}
\usepackage{verbatim} 
\usepackage{xcolor}
\usepackage{amsthm}
\usepackage{pifont}
\usepackage{DotArrow}
\usepackage{graphicx}
\usepackage{hyperref}

\title[Covers of infinite translation surfaces]{Indiscriminate covers of infinite translation surfaces are innocent, not devious}

\author[]{W. Patrick Hooper}
\address{The City College of New York, New York, NY, USA 10031}
\address{CUNY Graduate Center, New York, NY, USA 10016}
\email{whooper@ccny.cuny.edu}
\author[]{Rodrigo Trevi\~no}
\address{Department of Mathematics\\
         University of Maryland, College Park}
\email{rodrigo@math.umd.edu}

\ifdefined\theorem
\else
\newtheorem{theorem}{Theorem}
\newtheorem{proposition}[theorem]{Proposition}
\newtheorem{lemma}[theorem]{Lemma}
\newtheorem{remark}[theorem]{Remark}
\newtheorem{corollary}[theorem]{Corollary}

\theoremstyle{definition}
\newtheorem{definition}[theorem]{Definition}
\newtheorem{notation}[theorem]{Notation}

\fi
\ifdefined\openquestion
\else
\newtheorem{openquestion}[theorem]{Open Question}
\fi

\ifdefined\question
\else
\newtheorem{question}[theorem]{Question}
\fi

\newlength{\savearraycolsep}
	{\setlength{\savearraycolsep}{\arraycolsep}%
	\setlength{\arraycolsep}{#1}%
	\begin{array}{#2}}%
	{\end{array}\setlength{\arraycolsep}{\savearraycolsep}}

\newcommand{\red}[1]{{\textcolor{red}{#1}}}

\def\C{\mathbb{C}}%
\def\N{\mathbb{N}}%
\def\P{\mathbb{P}}%
\def\R{\mathbb{R}}%
\def\Z{\mathbb{Z}}%
\def\H{\mathbb{H}} 
\def\Isom{{\mathit{Isom}}} 

\def\GL{\textit{GL}}
\def\SL{\textit{SL}}
\def\SO{\textit{SO}}

\def\0{{\mathbf{0}}}
\def\1{{\mathbf{1}}}

\def\g{{\mathbf{g}}}

\def\v{{\mathbf{v}}} %

\def\sA{{\mathcal{A}}}
\def\sB{{\mathcal{B}}}

\def\sD{{\mathcal{D}}}
\def\sE{{\mathcal{E}}}

\def\sG{{\mathcal{G}}}
\def\sH{{\mathcal{H}}}

\def\sO{{\mathcal{O}}}

\def\sR{{\mathcal{R}}}

\def\sU{{\mathcal{U}}}

\def\half{{\frac{1}{2}}}

\def\bs{{\backslash}}
\def\ker{\textit{ker}}%
\newcommand{\nullset}{\emptyset}
\def\dist{\mathit{dist}} 

\makeatletter
\def\imod#1{\allowbreak\mkern10mu({\operator@font mod}\,\,#1)}
\makeatother

\def\Aff{\mathit{Aff}} 

\def\and{{\quad \textrm{and} \quad}}

\newif\ifdraft\drafttrue
\draftfalse

\newcommand{\Hom}[0]{{\mathrm{Hom}}}
\newcommand{\Cyl}[0]{{\mathcal{C}}}
\newcommand{\Cov}[0]{{\mathrm{Cov}}}

\newcommand{\Covt}[0]{\mathrm{TopCov}}

\newcommand{\ladder}[0]{L}

\newcommand{\gmap}{\mathfrak{g}}%

\newcommand{\compat}[1]{{\ifdraft{\textcolor{orange}{\textrm{{\bf Pat:} #1}}}\else\ignorespaces\fi}}

\begin{document}
\clearpage
\begin{abstract}
We consider the interaction between passing to finite covers
and ergodic properties of the straight-line flow on finite area translation surfaces with infinite topological type. Infinite type provides for a rich family of degree $d$ covers for any integer $d>1$. We give examples which demonstrate that passing to a finite cover can destroy ergodicity, but we also provide evidence that this phenomenon is rare. We define a natural notion of a random degree $d$ cover and show that, in many cases, ergodicity and unique ergodicity are preserved under passing to random covers. This work provides a new context for exploring the relationship between recurrence of the Teichm\"uller flow and ergodic properties of the straight-line flow.
\end{abstract}
\dedicatory{
\begin{flushleft}The private lives of surfaces \\
are innocent, not devious. \\
\hspace{1in} - Kay Ryan 
\end{flushleft}}
\maketitle
\thispagestyle{empty}

\section{Introduction}

A {\em translation surface} is a pair $(S, \alpha)$, where $S$ is a Riemann surface and 
$\alpha$ is a holomorphic 1-form on $S$. Let $Z \subset S$ denote the zeros of $\alpha$.
The $1$-form $\alpha$ endows $S \smallsetminus Z$ with local coordinates to the plane: for any 
$p$ we have the locally defined {\em coordinate chart} to $\C$ given by the local homeomorphism
$q \mapsto \int_p^q \alpha$. These coordinate charts differ locally only by translation. 
A translation surface inherits a metric by pulling back the Euclidean metric on the plane along the coordinate charts.
Points in $S \smallsetminus Z$ are locally isometric to the plane, while points in $Z$ are cone singularities with cone angle $2(k+1) \pi$ 
where $k \geq 1$ is the degree of the zero of $\alpha$.

We'll say a translation surface is {\em classical} if $S$ is a closed surface.
Here, there is a well known interplay between two types of dynamical systems:
{\em (1)} the dynamics of the (horizontal) {\em translation flow} on a translation surface $(S,\alpha)$ given in local coordinates
by 
$$F^t:S \to S; \quad (x,y) \mapsto (x+t,y),$$
and {\em (2)} the Teichm\"uller deformation $g^t$ on the moduli space of translation surfaces, where
$g^t(S,\alpha)$ is obtained from $(S,\alpha)$ by postcomposing the coordinate charts with the affine coordinate change
$g^t(x,y)=(e^{-t} x, e^t y)$. Namely, the Teichm\"uller deformation renormalizes the translation flow. A famous consequence of this relationship is given by Masur's Criterion: if the forward orbit of the Teichm\"uller deformation, $\{g^t(S,\alpha)\}_{t \geq 0}$, has a convergent subsequence $g^{t_n}(S,\alpha)$ with $t_n \to \infty$ (i.e., the orbit is {\em non-divergent}), then the translation flow is uniquely ergodic \cite{M92}. 

A surface has {\em infinite topological type} if its fundamental group is not finitely generated. This includes infinite genus surfaces.
Recently, many results similar in spirit to Masur's criterion have been proven in special cases for translation surfaces of infinite 
topological type \cite{Hinf,HW13, rodrigo:erg, trevino:bratMasur}.

Reasoning about infinite type translation surfaces introduces several difficulties:
\begin{enumerate}
\item Translation flow trajectories on a infinite type surface may not be defined for all time. In fact, there are interesting examples where trajectories are defined nowhere for some fixed time, for example the Icicled surface of Randecker \cite[Example 5.5]{Randecker}. In order to consider ergodicity, we require the translation flow to be {\em defined for all time almost everywhere} (i.e., for all $t$, the time $t$ map $F^t$ is defined a.e.).
\item While the translation flow is continuous, the domain is not compact.
\item There is no well established moduli space of infinite-type translation surfaces. (In this direction, the preprints  \cite{HooperImmersions1} and \cite{HooperImmersions2} propose a topology for the space of translation surfaces with basepoint.)
\end{enumerate}
In the classical case (1) and (2) can be easily resolved (via standard coding arguments) and moduli space is a finite dimension orbifold.
The translation flow can often be easily shown to be defined for all time almost everywhere, and we will in fact be assuming stronger hypotheses. Difficulty (2) simply must be worked around. A standard approach to work around (3) involves defining some topological space of surfaces where $\SL(2,\R)$ acts and prove that non-divergence of an orbit $t \mapsto g^t(S, \alpha)$ within this space entails consequences for the translation flow on $(S,\alpha)$. To date, the primary mechanism for  
building such a topological space of surfaces uses affine symmetries of the translation surface $(S,\alpha)$, as we will now explain.

Two translation surfaces
$(S,\alpha)$ and $(S', \alpha')$ are {\em translation equivalent} if there is a homeomorphism $h:S \to S'$ which is a translation in local coordinates.
Let $\SL_\pm(2,\R)$ denote the group of $2 \times 2$ real matrices with determinant $\pm 1$. 
The group $\SL_\pm(2,\R)$ acts on the collection of all translation surfaces by simultaneous postcomposition with all local charts;
see \S \ref{sect:background}.
We define $\sO(S,\alpha)$ to be the $\SL_\pm(2,\R)$ orbit:
\begin{equation}
\label{eq:sO}
\sO(S,\alpha)=\{A(S,\alpha)~:~A \in \SL_\pm(2,\R)\}/\text{translation equivalence}.
\end{equation}
The orbit $\sO(S,\alpha)$ is parameterized by a choice of $A \in \SL_\pm(2,\R)$
and thus inherits a topology as a topological quotient space. This space can be described concretely
as the quotient of $\SL_\pm(2,\R)$ by a subgroup, namely the surface's {\em Veech group}, 
$$V(S,\alpha) = \{A \in \SL_\pm(2,\R)~:~\text{$A(S,\alpha)$ is translation equivalent to $(S,\alpha)$}\}.$$
Observe that two surfaces $A_1(S,\alpha)$ and $A_2(S,\alpha)$ with $A_1,A_2 \in \SL_\pm(2,\R)$ are
translation equivalent if and only if $A_1$ and $A_2$ determine the same left coset (or equivalently, if $A_2^{-1} A_1 \in V(S,\alpha)$). Thus, there is a natural identification between $\sO(S,\alpha)$ and the left coset space
$\SL_\pm(2,\R)/V(S,\alpha)$.
In particular, this structure allows us to say
that the Teichm\"uller trajectory $g^t(S,\alpha)$ is 
{\em non-divergent in $\sO(S,\alpha)$} if there is a sequence $t_n \to +\infty$ so that 
$g^{t_n}(S,\alpha)$ converges in $\sO(S,\alpha)$.

Work of the second author \cite[Theorem 2]{rodrigo:erg} (in this paper as Corollary \ref{cor:rodrigo ergodicity from Veech group})
shows that if a finite area translation surface of possibly infinite genus $(S,\alpha)$ has a non-divergent Teichm\"uller trajectory in $\sO(S,\alpha)$, then the translation flow is ergodic on $(S,\alpha)$. 
It seems reasonable to conjecture that in fact unique ergodicity holds under these hypotheses: Separate works of the two co-authors each prove results to this effect in special cases; see Example F.2 and Appendix H of \cite{Hinf} and the main result of \cite{trevino:bratMasur}.


Suppose $(S,\alpha)$ is a finite area translation surface with infinite topological type which has a translation flow which is defined for all time almost everywhere and is ergodic.
Any cover of $S$ inherits a translation structure by pulling back the $1$-form $\alpha$ under the covering map.
When endowed with this structure, the cover's translation flow will also be defined for all time almost everywhere.
In this paper we study the ergodicity of the translation flow on finite unbranched covers $(\tilde S, \tilde \alpha)$
of $(S,\alpha)$. We introduce some spaces of covers and following the above paradigm study how the behavior of Teichm\"uller trajectories through this space influences the ergodicity of the translation flow on a cover.

In order to say something constructive, we pick some integer $d \geq 2$, and restrict attention to covers of degree $d$. Choose an arbitrary non-singular basepoint on the surface $S$. Let $(\tilde S, \tilde \alpha)$ be a cover of $(S,\alpha)$, where the flat structure on $(\tilde S, \tilde \alpha)$ is lifted from the one on $(S,\alpha)$. The fiber of the basepoint can be identified in an arbitrary way with the set $\{1,2, \ldots, d\}$. The {\em monodromy action} is the natural right action of the fundamental group of $S$ on the fiber of the basepoint. Our identification of the fiber determines
a {\em monodromy representation} $\pi_1(S) \to \Pi_d$, where $\Pi_d$ is the permutation group of $\{1, \ldots, d\}$. We note that a cover can be reconstructed from its monodromy representation, and two covers are 
isomorphic (in the sense of covering theory) if and only if these monodromy representations differ by conjugation by an element of $\Pi_d$. These ideas are reviewed in \S \ref{sec:spaces}.
Now fix a subgroup $G \subset \Pi_d$. We say a cover $(\tilde S, \tilde \alpha)$
has {\em monodromy in $G$} if it can be realized by a monodromy representation to $G$. From the above remarks, we note that the space of covers of $(S,\alpha)$ with monodromy in $G$ up to cover isomorphism is identified with
$$\Pi_d \bs \Hom\big(\pi_1(S),G\big),$$
where $\Pi_d$ acts on $\Hom\big(\pi_1(S),G\big)$ by conjugation.
We endow $\Hom\big(\pi_1(S),G\big)$ with the product topology by viewing it as a subset of $G^{\pi_1(S)}$,
where the finite set $G$ is given the discrete topology. In particular, $\Hom\big(\pi_1(S),G\big)$ is homeomorphic to a Cantor set, since $\pi_1(S)$ is a free group with countably many generators; see \S \ref{sec:spaces}.

We define
$\text{Cov}_G(S,\alpha)$ to be the collection of all covers of $(S,\alpha)$ with monodromy in $G$ up to translation equivalence. 
Translation equivalence is a coarser notion of equivalence than covering isomorphism, and we give $\text{Cov}_G(S,\alpha)$ the quotient topology by viewing this space of covers as a 
topological quotient of $\Pi_d \bs \Hom\big(\pi_1(S),G).$
The choice of $h \in \Hom\big(\pi_1(S),G)$ determines a cover $(\tilde S_h, \tilde \alpha_h)$ with $h$ describing the monodromy of the cover. We note that the cover need not be connected. In fact $(\tilde S_h, \tilde \alpha_h)$ is connected
if and only if the image of the monodromy representation, $h\big(\pi_1(S)\big)$ acts transitively on $\{1, \ldots, d\}$.

The main object of interest to this paper is the union of $\SL_\pm(2, \R)$ orbits of
covers of $(S,\alpha)$ with monodromy in $G$:
\begin{equation}
\label{eq:covers cocycle}
\tilde \sO_G(S,\alpha)=\{A (\tilde S, \tilde \alpha):~
\text{$A \in \SL_\pm(2,\R)$ and $(\tilde S, \tilde \alpha) \in \Cov_G(S,\alpha)$}\}
/\sim,
\end{equation}
where $\sim$ denotes translation equivalence.
We call the action of the diagonal subgroup $g^t$ of $\SL_\pm(2, \R)$ on $\tilde \sO_G(S,\alpha)$ the {\em cover cocycle}. Again, $\tilde \sO_G(S,\alpha)$ inherits a topology because it can be considered a topological quotient
of $\SL_\pm(2,\R) \times \Cov_G(S,\alpha)$. This is a natural space in which to study Teichm\"uller deformations
of covers of $(S, \alpha)$. Like $\Cov_G(S,\alpha)$, this space contains both connected and disconnected surfaces
provided $G$ acts transitively on $\{1,\ldots,d\}$. We prove the following:

\begin{theorem}[Connected accumulation point implies ergodicity]
\label{thm:1}
Let $d \geq 2$ be an integer and let $G \subset \Pi_d$ be a subgroup which acts transitively on $\{1,\ldots, d\}$.
Let $(S,\alpha)$ be a finite area translation surface
with infinite topological type, and let $(\tilde S, \tilde \alpha)$ be a cover with monodromy in $G$.
Then, the translation flow on $(\tilde S, \tilde \alpha)$ is defined for all time almost everywhere and is ergodic if 
the Teichm\"uller trajectory $g^t(\tilde S, \tilde \alpha)$ has an
$\omega$-limit point in $\tilde \sO_G(S,\alpha)$ representing a connected surface.
\end{theorem}

Theorem \ref{thm:1} is proved in \S  \ref{sec:erg}.
We note that ergodicity of the translation flow on a finite cover has consequences for classifying the invariant measures.

\begin{proposition}[Ergodicity and unique lifts of measures]
\label{prop:lifting measures}
Suppose that $(\tilde S, \tilde \alpha)$ is a finite cover of the finite area translation surface $(S, \alpha)$ with infinite topological type. Also suppose that both surfaces have translation flows which are defined for all time almost everywhere and are ergodic.
Then, Lebesgue measure on $(\tilde S,\tilde \alpha)$ is the unique translation flow-invariant measure which projects to Lebesgue measure on $(S, \alpha)$ under the covering map.
\end{proposition}

See \S \ref{sect:proofs} for the brief proof.

\begin{corollary}[Lifting unique ergodicity]
\label{cor:lifting}
Suppose the conditions of Proposition \ref{prop:lifting measures} are satisfied
and additionally the translation flow on $(S, \alpha)$ is uniquely ergodic. 
Then, the translation flow on the cover $(\tilde S,\tilde \alpha)$ is uniquely ergodic.
\end{corollary}
\begin{proof}
Suppose $\mu$ is a translation flow-invariant probability measure on $(\tilde S,\tilde \alpha)$.
Then its push forward under the covering map, $p_\ast \mu$ is a translation-flow invariant measure on $(S, \alpha)$, and by uniqueness is necessarily Lebesgue measure. The proposition guarantees that $\mu$
is also Lebesgue measure.
\end{proof}

The existence of a Veech group for the surface $(S,\alpha)$ gives a mechanism for generating the type of
accumulation we need for applying Theorem \ref{thm:1}:

\begin{proposition}
\label{prop:accumulation points}
Suppose $(S,\alpha)$ is a finite area translation surface with infinite topological type, and let $(\tilde S, \tilde \alpha)$ be a cover with monodromy in $G$.
If $(S,\alpha)$ has a non-divergent Teichm\"uller trajectory in $\sO(S,\alpha)$, then the Teichm\"uller trajectory of the cover $g^t(\tilde S, \tilde \alpha)$ has an $\omega$-limit point in $\tilde \sO_G(S,\alpha)$.
\end{proposition}

Again, see \S \ref{sect:proofs} for the proof. In order to use the conclusions of Theorem \ref{thm:1}, 
we would need to know that there is a connected $\omega$-limit point. We expect that it is 
difficult in general to determine precisely when this is true. However, we will introduce
a notion under which most covers have connected accumulation points.

Let $G \subset \Pi_d$ be as above, and let $(S,\alpha)$ be a finite area translation surface with infinite topological type. We explain in \S \ref{sect:measures} that there is a natural Borel probability measure $m_G$ on the space $\text{Cov}_G(S,\alpha)$ of covers with monodromy in $G$ which is invariant under automorphisms of the base surface $S$. This gives us a notion of a ``random cover.''
Informally, the measure $m_G$ corresponds to the notion of random cover
obtained by flipping a fair coin to determine the images
of the generators of the fundamental group under the monodromy representation.

\begin{theorem}[Random covers accumulate on connected covers]
\label{thm:2}
Let $(S,\alpha)$ be a finite area translation surface
with infinite topological type, and suppose that it has a non-divergent Teichm\"uller trajectory in $\sO(S,\alpha)$.
Suppose $G$ is a transitive subgroup of the permutation group $\Pi_d$. 
Then $m_G$-almost every cover $(\tilde S, \tilde \alpha) \in \Cov_G(S, \alpha)$ has
a Teichm\"uller trajectory with a connected accumulation point.
\end{theorem}

The proof lies in \S \ref{sect:proofs}. As a consequence of Theorem \ref{thm:1}, we see:

\begin{corollary}[Ergodicity of random covers]
\label{cor:2}
With the hypotheses of Theorem \ref{thm:2}, $m_G$-almost every cover of $(S, \alpha)$ with monodromy in $G$
has ergodic translation flow. 
\end{corollary}

Corollary \ref{cor:2} has applications to certain types of finite skew product extensions of $n$-adic odometers as follows. 
For any integer $n \geq 2$, let $X_n = \{0,\dots, n-1\}^\mathbb{N}$ be given the product topology.
The \emph{$n$-adic odometer} is the map $\Omega_n: X_n \rightarrow X_n$ defined as addition by $1$ with infinite carry to the right:
\begin{equation}
\label{eq:odometer}
\Omega_n : x = (x_1,x_2, \dots) \mapsto (0, 0, \dots, x_{\kappa(x)} + 1, x_{\kappa(x) + 1}, x_{\kappa(x) + 2}, \dots ),
\end{equation}
where $\kappa(x) \in \N \cup  \{+\infty\}$ is the smallest index $k \in \N$ so that $x_k \neq n-1$ or $+\infty$ if no such index exists.

Denote by $\Gamma^+$ the free group on the countably infinite set of generators $\{\gamma_i:~i \in \N\}$.
Let $n \geq 2$ be an integer and $G$ a subgroup of $\Pi_d$. For any $\psi_+ \in \mathrm{Hom}(\Gamma^+,G)$, we define the skew product over the $n$-adic odometer $E_{\psi_+}: X_n \times \{1,\dots, d\} \rightarrow X_n \times \{1,\dots, d\}$ as
\begin{equation}
\label{eqn:skew}
E_{\psi_+}(x,m) = \big(\Omega_n(x), \psi_+(\gamma_{\kappa(x)})\,(m)\big),
\end{equation}
where $\kappa(x)$ is as above. The image $E_{\psi_+}(x,m)$ is well defined unless $\kappa(x)=+\infty$, but we ignore this issue because points have well defined orbits off a countable set.

The space $\mathrm{Hom}(\Gamma^+,G)$ parameterizes these skew products and comes with a natural product measure as in Definition \ref{def:random cover} in \S \ref{sect:measures}. We will denote this measure by $\mu_+$.
\begin{theorem}
\label{thm:skew}
Let $n \geq 2$ be an integer and $G$ be a subgroup of $\Pi_d$ that acts transitively on $\{1,\ldots, d\}$. Then for $\mu_+$-almost every $\psi_+ \in \mathrm{Hom}(\Gamma^+,G)$, the skew product $E_{\psi_+}$ is uniquely ergodic.
\end{theorem}
The connection between translation flows and $n$-adic odometers comes from the fact that the suspension flow over an $n$-adic odometer is measurably isomorphic to the translation flow on the infinite genus surface first studied by Chamanara. The skew-products described here have a suspension given by translation flow on a cover of Chamanara's surface. As such, the theorem follows as a consequence of Corollaries \ref{cor:lifting} and \ref{cor:2}. We prove this result at the beginning of \S \ref{sect:Chamanara}.

\subsection*{Devious covers} 
The above discussion has left open a natural question: What happens if $g^t(S,\alpha)$ is non-divergent in $\sO(S,\alpha)$, but the Teichm\"uller orbit of the connected cover $(\tilde S, \tilde \alpha)$ only accumulates
on disconnected covers in $\tilde \sO_G(S,\alpha)$. We call such covers {\em devious}. 

Apparently there is no {\em a priori} reason why the translation flow on a devious cover should be ergodic or non-ergodic. We illustrate this by example in \S \ref{sect:evil covers}.

In \S \ref{sect:Chamanara} after proving Theorem \ref{thm:skew}, we 
concentrate on devious covers of the surface $(S_2,\alpha_2)$ 
of Chamanara which is related to the $2$-adic odometer. Devious $G$-covers with non-ergodic translation flow
occur for all $G$ for fairly trivial reasons in this case. (The surface $(S_2,\alpha_2)$ has horizontal saddle connections which can connect the surface but are not seen by the translation flow.)
However we also exhibit devious covers of $(S_2,\alpha_2)$ whose translation flow is uniquely ergodic. Our construction of such covers works as long as the subgroup $G \subset \Pi_d$ is large enough to contain two subgroups $H_1$ and $H_2$ which fail to act transitively on $\{1,\ldots, d\}$ but so that $H_1 \cup H_2$ generates a subgroup of $G$ which does act transitively. 

The translation flow on Chamanara's surface is special for many reasons, including that the Teichm\"uller trajectory of this surface is periodic in $\sO(S_2, \alpha_2)$, and so we consider a different surface in \S \ref{sect:ladder surface}. We consider double covers of affine images of a translation surface we call the ladder surface, $(S_\ladder,\alpha_\ladder)$. (See Figure \ref{fig:ladder_surface} on page \pageref{fig:ladder_surface}.)
The ladder surface has a non-elementary Veech group containing a parabolic and a reflection.
We produce several possible behaviors of the translation flow on double covers of $A(S_\ladder,\alpha_\ladder)$:
\begin{enumerate}
\item[(L1)] For some $A \in \SL_\pm(2,\R)$, there are no devious double covers of $A(S_\ladder,\alpha_\ladder)$, i.e. 
for every connected double cover $(\tilde S,\tilde \alpha)$, the forward orbit $g^t A(\tilde S,\tilde \alpha)$ accumulates on a connected cover in $\tilde \sO_{\Z_2}(S_L,\alpha_L)$ so that $A(\tilde S,\tilde \alpha)$ always has ergodic translation flow.
\item[(L2)] For some $A \in \SL_\pm(2,\R)$, there are devious double covers of $A(S_\ladder,\alpha_\ladder)$ but nonetheless every connected double cover has ergodic translation flow.
\item[(L3)] For some $A \in \SL_\pm(2,\R)$, there are devious double covers of $A(S_\ladder,\alpha_\ladder)$ and every devious double cover has non-ergodic translation flow.
\end{enumerate}
Let us briefly summarize how to distinguish the cases above. We use $V' \subset \SL_\pm(2,\R)$ to denote the
known Veech group of $(S_L,\alpha_L)$. In all the cases we assume that $g^t A V'$ is non-divergent in $\SL_\pm(2,\R) / V'$, which implies that $g^t A(S_L, \alpha_L)$ is non-divergent in $\sO(S_L,\alpha_L)$. 
We find an infinite index subgroup $\tilde V' \subset V'$ with the property that $\tilde V' \subset V(\tilde S, \tilde \alpha)$ for every double cover $(\tilde S, \tilde \alpha)$  of $(S_L, \alpha_L)$. 
We are in case (L1) if $g^t A \tilde V'$ is non-divergent in $\SL_\pm(2,\R) / \tilde V'$. (Statement (L1) follows from Proposition \ref{prop:non-divergence covers} of \S \ref{sect:ladder surface}.) Now assume $g^t A \tilde V'$ is divergent.
We describe in Theorem \ref{thm:main ladder} a combinatorial measurement $v(A)$ of the divergence rate
of the trajectory $g^t A \tilde V'$ in $\SL_\pm(2,\R)/\tilde V'$. As $v(A)$ decreases the trajectory diverges quicker. We show that if $v(A)>\varphi^2$ where $\varphi$ is the golden mean then we are in case (L2). On the other hand if $v(A)<\varphi^2$, then we are in case (L3). (Statements (L2) and (L3) appear again in Theorem \ref{thm:main ladder}.) The sharpness of this transition demonstrates the power of the methods used to distinguish ergodic and non-ergodic behavior. We wonder:
\begin{enumerate}
\item[(L2$\frac{\text{1}}{\text{2}}$)] Could there be an $A  \in \SL_\pm(2,\R)$ (with $v(A)=\varphi^2$) so that some devious double covers
of $(S_\ladder,\alpha_\ladder)$ exhibit ergodic translation flow while others exhibit non-ergodic translation flow?
\end{enumerate}

\subsection*{Acknowledgments}
The authors would like thank the anonymous referee for numerous useful comments which greatly improved the exposition. This project began at ICERM which provided a stimulating environment. 
W. P. H. was supported by N.S.F. Grants DMS-1101233 and DMS-1101233 as well as by a PSC-CUNY Award (funded by The Professional Staff Congress and The City University of New York). R. T. was supported by BSF grant 2010428, ERC starting grant DLGAPS 279893, NSF Postdoctoral Fellowship DMS-1204008, and an AMS-Simons Travel Grant. The quote from on the first page is taken from \emph{Surfaces} in \cite{RyanPoems}.

\section{Context on finite area flat surfaces of infinite topological type}

There has been an increased interest in the study of the dynamics and geometry of flat surfaces of infinite genus.
Unlike classical flat surfaces (which are compact flat surfaces of finite type), there is no natural space for 
parametrizing flat metrics for all surfaces of a given topological type. This gives the first obstacle to utilizing tools from the theory of compact surfaces. Different techniques have been developed to overcome this fundamental shortcoming, which prevents us from developing a theory to answer one of the most basic questions: whether the translation flow on a given flat surface of infinite type and finite area is ergodic or not.

Most studies concentrate on the case of the surface having finite or infinite area (a notably exception being \cite{Hinf}, where the methods work for surfaces of finite and infinite area). Such choice has great implications to the tools used, the results obtained, and the method of construction used to produce examples or to define some ``spaces of surfaces''. A common tool in both contexts is the use of Veech groups, which are a sort of symmetry groups of the surface. For compact flat surfaces, these are always discrete subgroups of $SL(2,\mathbb{R})$. Flat surfaces of finite type with non-trivial Veech groups are part of a very deep theory
which has grown out of the foundational contributions of Thurston \cite{T88} and Veech \cite{V}, so the fact that they can be used in the infinite type setting is encouraging. 

We will concentrate here on the development of the theory of flat surfaces of infinite type and  finite area, though there is also a rapidly developing theory of flat surfaces of infinite type and infinite area (where often interest centers on infinite abelian covers of finite type flat surfaces).

To our knowledge, the first papers on dynamics on flat surfaces of infinite type are those which come out of the infinite step polygonal billiards introduced in \cite{troubetzkoy:infinite, infinite-step} through the unfolding procedure. All such surfaces considered were of finite area and came from ``rational'' polygons, i.e., the angles of the billiard from which the surfaces were constructed satisfied some rationality conditions. Ergodic properties as well as topological results were obtained for a large class of these types of surfaces. The approach of these articles is very different from the approach here, because arguments do not make use of Veech groups. 

The seminal paper of Chamanara \cite{Chamanara04} introduced a 1-parameter family of flat surfaces of infinite type and of finite area with a non-trivial Veech group. The main results of that paper discussed the Veech groups that appear. Most importantly, even though the surfaces constructed possess many symmetries, the Veech group is never a lattice for any surface arising in this construction. We review Chamanara's construction in section \ref{sect:evil covers} and apply some of our results to spaces of covers thereof.

Another study of flat surfaces of infinite type and finite area has been the construction of Bowman which extends the Arnoux-Yoccoz family of flat surfaces to include a surface of infinite genus and finite area \cite{Bowman13}. This surface of infinite genus and finite area admits a affine automorphism with hyperbolic derivative. Moreover, it was found that the Veech group of this surface is isomorphic to $\mathbb{Z}\times \mathbb{Z}_2$ and that the directions preserved by the hyperbolic automorphism correspond to uniquely ergodic translation flows.

In \cite{Hinf}, a construction of Thurston is modified to produce infinite genus flat surfaces with non-elementary Veech groups. The construction sometimes produces finite area surfaces,
and in this case the translation flow can be shown to be uniquely ergodic on various affine images of the surfaces.

The work \cite{rodrigo:erg} addresses the question of ergodicity of translation flows for surfaces of finite area. The hypotheses of the main results are 
independent of topological type and therefore can be used to determine when the translation flow of a flat surface of infinite genus and finite area is ergodic or uniquely ergodic. We use this criterion in section \ref{sec:erg} to study ergodic properties of translation flows on covers of infinite translation surfaces.

Finally, in \cite{LT:models} a way of constructing flat surfaces of finite area and infinite genus is developed through a connection to adic and cutting and stacking transformations,
generalizing constructions of Bufetov in the classical case \cite{Bufetov13}.
Among other things, it is shown there that translation flows on surfaces of infinite genus and finite area can exhibit behavior which does not occur for translation flows on compact flat surfaces. 
\compat{I'm commenting out this. I think it is not true unless we place more hypotheses on the compact surface. I think it is true if we assumed the base compact surface recurs modulo its Veech group. ---- In the present article, we also obtain results which are also not present for translation flows on compact surfaces, namely the existence of devious (unbranched) covers.}


In general it is not well understood what Veech groups can arise for a flat surface of infinite type and finite area. In particular, it is unknown whether there exists a flat surface of infinite type and finite area whose Veech group is a lattice in $\SL_\pm(2,\mathbb{R})$. 
(In contrast, it is known that for practically any subgroup $G$ of $\SL_\pm(2,\mathbb{R})$, there is a flat surface of infinite type and infinite area whose Veech group is $G$ \cite{PSV11}.) 

As it can be seen, Veech groups have played a significant role in the study of translation flows. The properties of translation flows so far obtained for surfaces of infinite genus and finite area are mostly similar to those of compact flat surfaces.

In this paper, we broaden the scope of known phenomena of translation flows on flat surfaces by utilizing the structure provided by coverings to produce spaces of flat surfaces of infinite type and finite area which are invariant under Teichm\"uller deformations, and provide sufficiently interesting Teichm\"uller dynamics so that varied
phenomena appear for corresponding translation flows. This construction allows us to determine when the lift of a translation flow to a cover retains ergodic properties which the base surface possesses, such as unique ergodicity. This approach thus 
overcomes some of the problems created by not having a natural parameter space
for infinite type flat surfaces.

\section{Background on translation surfaces}
\label{sect:background}

Let $(S, \alpha)$ be a translation surface and let $Z \subset S$ denote the zeros of $\alpha$. The {\em straight line flow} on $(S,\alpha)$ in direction $\theta$ is the flow 
$F^t_\theta:S \to S$ defined for $t \in \R$
given in local coordinates by $F^t_\theta(x,y)=(x+t \cos \theta, y +t \sin \theta)$ away from the zeros of $\alpha$.
We reserve the name {\em translation flow} for the straight line flow $F^t_\theta$ where
$\theta=0$. This is the flow on $S$ determined by the rightward unit vector field $X$. We will use $Y$ denote the upward unit vector field. 

A more global definition of the straight line flows can be done as follows. Since $\alpha$ is holomorphic, the 1-forms $\Re(\alpha)$ and $\Im(\alpha)$ are harmonic, and thus closed. Therefore, the distributions $\mathrm{ker}\, \Im(\alpha)$ and $\mathrm{ker}\, \Re(\alpha)$ define a pair of foliations away from $Z$, the \emph{horizontal and vertical} foliations. The generators of the distributions in the unit tangent bundle of $S-Z$ are the vector fields $X_\alpha$ and $Y_\alpha$ which generate, respectively, the \emph{horizontal and vertical} flows. The translation flow, as defined above, then corresponds to the horizontal flow. From this point of view, we will denote by $\varphi_t^{X_\alpha}$ and $\varphi_t^{Y_\alpha}$, respectively, the horizontal and vertical flows.

There is a group of deformations of the flat metric on $(S,\alpha)$ which is parametrized by the group $GL(2,\mathbb{R})$. We will mostly be interested in the action of the area preserving subgroup $\SL_\pm(2,\R) \subset \GL(2,\R)$, and this is the $\SL_\pm(2,\R)$ action mentioned in the introduction.
Fix a matrix $A \in \GL(2,\R)$. We get new (non-conformal) local coordinate charts to the plane by postcomposing the charts on $(S,\alpha)$ to $\C$ by the real-linear map 
$$A:\C \to \C; \quad x+iy \mapsto a_{1,1} x + a_{1,2}y + i (a_{2,1} x+a_{2,2} y).$$
Then, we get a new Riemann surface structure $S'$ on $S$ by pulling back the complex structure using these deformed
charts, and the charts determine a new holomorphic $1$-form $\alpha'$ on $S'$. We define $A(S,\alpha)=(S', \alpha')$. The action of the rotation subgroup $\SO(2,\mathbb{R})$ parametrizes the directional flows on a given surface $(S,\alpha)$: for $\theta \in \SO(2,\mathbb{R}) \simeq S^1$, the horizontal flow on $\theta(S,\alpha)$ corresponds to the straight line flow on $(S,\alpha)$ in direction $\theta$.

\compat{Please read this paragraph for correctness and clarity.}
Let $(S,\alpha)$ and $(S',\alpha')$ be translation surfaces. We say they are {\em translation equivalent} if there is a homeomorphism $h:S \to S'$ which is a locally a translation in local coordinate charts provided by the $1$-forms. 
The {\em Veech group} of $(S,\alpha)$ is the subgroup $V(S,\alpha) \subset \SL_\pm(2,\R)$ of elements $A \in \SL_\pm(2,\R)$ so that
$A(S,\alpha)$ and $(S,\alpha)$ are translation equivalent. An {\em affine homeomorphism} from a translation surface $(S,\alpha)$ to another 
translation surface $(S',\alpha')$ is a diffeomorphism $\phi:S \to S'$ whose derivative $D(\phi)$ is constant when measured using local coordinates provided by the $1$-forms. By considering the derivative $D(\phi)$ to be a real linear map we have $D(\phi) \in \GL(2, \R)$.
We can describe this condition in terms of the horizontal and vertical vectors on $(S,\alpha)$ and $(S,\alpha')$ using the equation
$$\left( \begin{array}{r} 
\phi_\ast X_\alpha \\
\phi_\ast Y_\alpha
\end{array}\right)=A \left( \begin{array}{r} 
X_{\alpha'} \\
Y_{\alpha'}
\end{array}\right),
$$
where $\phi_\ast$ denotes the push forward action on vector fields and
$A=D(\phi)$ is a $2 \times 2$ matrix (the {\em derivative}) .
Observe that the statement that $A(S,\alpha)$ and $(S',\alpha')$ are translation equivalent is equivalent to the statement that there is a affine homeomorphism $\phi:(S,\alpha) \to (S', \alpha')$ with derivative $A$. 
The {\em affine automorphism group of $(S,\alpha)$}, $\Aff(S,\alpha)$ is the group of all affine homeomorphisms
from $(S,\alpha)$ to itself. By the prior observation, we have 
$$D\big(\Aff(S,\alpha)\big)=V(S,\alpha).$$

\section{Finite covers of infinite surfaces}
\label{sec:covers}

In this section, we work out the theory of spaces of finite degree covers of a translation surface $(S,\alpha)$ of infinite topological type. We describe the topology of the space of covers $\text{Cov}_G(S,\alpha)$, mentioned in the introduction, in subsection \ref{sec:spaces}. In subsection \ref{sect:measures}, we place a natural Borel measure on this space.
Subsection \ref{sect:disconnected} discusses why disconnected covers should be considered rare.

\subsection{Spaces of covers}
\label{sec:spaces}
Let $S$ be a topological surface without boundary having infinite topological type (i.e., with infinitely generated fundamental group) which is realizable as a Riemann surface. By Rad\'o's Theorem $S$ is second-countable 
\cite{Rado1925} \cite[\S 1.3]{HubbardTeich1}. 
Choose a basepoint $s_0 \in S$. 
It follows from Richards' classification of surfaces \cite[Theorem 3]{Richards}
that the fundamental group $\pi_1(S,s_0)$ is isomorphic to the free group with countably many generators. 
\compat{The paper ON FENCHEL-NIELSEN COORDINATES ON TEICHM ULLER
SPACES OF SURFACES OF INFINITE TYPE
by
D. ALESSANDRINI, L. LIU, A. PAPADOPOULOS, W. SU, AND Z. SUN
takes a similar point of view of this fact. See the first few paragraphs of \S 4.}

A {\em covering} of $S$ is a pair $(p, \tilde S)$, where $\tilde S$ is a topological surface and $p: \tilde S \to S$ is a topological covering. Two covers $(p_1, \tilde S_1)$ and $(p_2, \tilde S_2)$ are {\em isomorphic} if there is a homeomorphism $h:\tilde S_1 \to \tilde S_2$ so that $p_1=p_2 \circ h$.

We recall idea of the monodromy action from covering space theory. Let $p:\tilde S \to S$ be a covering map.
The {\em monodromy action} is the right action on the fiber over the basepoint, $p^{-1}(s_0)$,
defined by
$$p^{-1}(s_0) \times \pi_1(S,s_0) \to p^{-1}(s_0); \quad \tilde s \cdot \gamma := \tilde \beta(1),$$
where $\beta:[0,1] \to S$ is a loop in the class of $\gamma \in \pi_1(S,s_0)$, and $\tilde \beta:[0,1] \to \tilde S$
is a lift (i.e., $p \circ \tilde \beta=\beta$) so that $\tilde \beta(0)=\tilde s$. It should be noted that the definition of $\tilde s \cdot \gamma$
is independent of the choice of $\beta$. (Once $\beta$ is chosen,
its lift $\tilde \beta$ is determined based on the condition that
$\tilde \beta(0)=\tilde s$, and $\tilde \beta(1)$ depends only on $\gamma$.)

It is a basic observation from covering space theory that the monodromy action determines the cover up to isomorphism. This includes
disconnected covers of the connected surface $S$. We will briefly show how to build a cover from an action. Concretely, given any right action of $\pi_1(S,s_0)$
on a discrete set $J$, we can build a cover of $S$ with this action as the monodromy action. To see this, fix such an
action. For each $j \in J$, let $\text{Stab}(j) \subset \pi_1(S,s_0)$ be the stabilizer of $j$. 
We can then build a cover $\tilde S_j$ as the quotient of the universal cover of $S$ by $\text{Stab}(j)$. In the cover $\tilde S_j$, the lifts of the basepoint are then naturally identified with elements of the orbit $[j]$ of $j$ under the action of $\pi_1(S,s_0)$. 
It can be observed that if $k \in [j]$, then there is a cover isomorphism from $\tilde S_j$ to $\tilde S_k$ which respects the identification between $[j]$ and the lifts of the basepoint on each surface. It may be observed that by taking the disjoint union of such surfaces over all orbits under the action, we obtain a cover,  
\begin{equation}
\label{eq:built cover}
\tilde S=\bigsqcup_{[j] \subset J} \tilde S_j
\end{equation}
of $S$ with the desired monodromy action. (Here we are picking a representative $j$ from each orbit [j].) Furthermore,
actions on two discrete sets $J$ and $J'$ determine isomorphic covers if and only if the actions are conjugate up to bijection,
i.e., if there is a bijection $f:J \to J'$
so that $f(j \cdot \gamma)=f(j) \cdot \gamma$ for all $\gamma \in \pi_1(S, s_0)$ and $j \in J$.

We will now specialize this discussion to finite covers of $S$.
Suppose $p:\tilde S \to S$ is a covering map of degree $d$. 
Let $\Pi_d$ be the symmetric group acting by permutations of 
$\{1, \ldots, d\}$, and let 
$$\ell:\{1, \ldots, d\} \to p^{-1}(s_0)$$ 
be a labeling (a bijection) to the fiber. The associated 
{\em monodromy representation} (which depends on the labeling) is the group homomorphism
$M_\ell: \pi_1(S,s_0) \to \Pi_d$ defined so that 
$$M_\ell(\gamma)(i)=\ell^{-1} \big(\ell(i) \cdot \gamma^{-1}\big) \quad
\text{for all $\gamma \in \pi_1(S,s_0)$ and all $i \in \{1, \ldots, d\}$,}
$$
where we have multiplied on the right by $\gamma^{-1}$ to make this map into a homomorphism (as opposed to an anti-homomorphism).
Conversely, such a representation determines an action on $\{1,\ldots, d\}$ and so, from the above discussion, 
a choice of a monodromy representation $\pi_1(S,s_0) \to \Pi_d$ determines
a $d-$fold cover of $S$. Given two such representations, the covers are isomorphic if and only if they differ by conjugation by an element of $\Pi_d$, which has the effect of changing the labeling function. Thus, the space of $d$-fold covers of $S$ up to isomorphism is canonically identified with
$$\Pi_d \bs \Hom\big(\pi_1(S,s_0),\Pi_d\big),
\quad \text{where $\Pi_d$ is acting by conjugation.}$$ 
We endow $\Hom\big(\pi_1(S,s_0),\Pi_d\big)$ with the topology of pointwise convergence (or equivalently, the subspace topology coming from the inclusion
of $\Hom\big(\pi_1(S,s_0),\Pi_d\big)$ into the product space $\Pi_d^{\pi_1(S,s_0)}$), and this space of covers gets the quotient topology.



As in the introduction if $G$ is a subgroup of $\Pi_d$
for an integer $d \geq 2$, we say that a cover $\tilde S$ has {\em monodromy in $G$} if there is a representation $\pi_1(S,s_0) \to G$ which determines a cover isomorphic to $\tilde S$. 
Note that this concept is independent of the basepoint. We define $\Covt_G(S)$ to be the collection of all covers of $S$ with monodromy in $G$
up to isomorphism of covers. Such covers are thus determined by elements of
$\Hom\big(\pi_1(S,s_0),G\big) \subset \Hom\big(\pi_1(S,s_0),\Pi_d\big)$. 
So, the space $\Covt_G(S)$ is in bijective correspondence with 
\begin{equation}
\label{eq:space of covers}
\Pi_d \bs \Hom\big(\pi_1(S,s_0),G\big).
\end{equation}
We use the identification with this quotient space to topologize $\Covt_G(S)$.

\begin{proposition}
\label{prop:topological covers}
The space $\Covt_G(S)$ is homeomorphic to a Cantor set.
\end{proposition}
\begin{proof}
We use the characterization of Cantor sets as non-empty, perfect, compact, totally disconnected, and metrizable.
We already know that $\Hom\big(\pi_1(S,s_0),G\big)$ is a Cantor set and thus satisfies these properties.
Because $\Covt_G(S)$ is a quotient of the Cantor set $\Hom\big(\pi_1(S,s_0),G\big)$, we see it is non-empty
and compact. Because the equivalence classes have finite size, we see that because $\Hom\big(\pi_1(S,s_0),G\big)$ is perfect, so is $\Covt_G(S)$.
Also because of this, we can use a metric on $\Hom\big(\pi_1(S,s_0),G\big)$ and the use the Hausdorff metric restricted to equivalence classes to put a metric
on $\Covt_G(S)$. To see the space is totally disconnected, suppose $h_1, h_2 \in \Hom\big(\pi_1(S,s_0),G\big)$ have distinct images
in $\Covt_G(S)$. Then for any permutation $p\in \Pi_d$, there is an element $\gamma_{p} \in \pi_1(S,s_0)$
so that $h_1(\gamma_{p}) \neq p \cdot h_2(\gamma_{p}) \cdot p^{-1}$. 
For each permutation $q \in \Pi_d$, consider the following sets $U_q \subset \Hom\big(\pi_1(S,s_0),G\big)$:
$$U_q=\{h~:~\text{$h_1(\gamma_{p})=q \cdot h(\gamma_p) \cdot q^{-1}$ for all $p \in \Pi_d$}\}.$$
The set $U_q$ is clopen because it is a finite union of cylinder sets. (To see this, we write each $\gamma_p$ as a product of generators. Then there are finitely many values $h$ can take on the generators used so that each $h(\gamma_p)$ equals $q^{-1} h_1(\gamma_p) q$, i.e., so that $h \in U_q$.) Also observe $h_1 \in U_e$, where
$e \in \Pi_d$ is the identity element, and $h_2 \not \in U_q$ for any $q$ from the remarks above.
Thus the two complimentary sets
$$U=\bigcup_{q \in \Pi_d} U_q \quad \text{and} \quad V=\bigcap_{q \in \Pi_d} \Hom\big(\pi_1(S,s_0),G\big) \smallsetminus U_q$$
are both clopen while $h_1 \in U$ and $h_2 \in V$. Finally, we observe they are invariant under conjugation since
$$U=\bigcup_{q \in U_q} q^{-1} U_e q.$$
Thus $U$ and $V$ descend to clopen sets in $\Pi_d \bs \Hom\big(\pi_1(S,s_0),G\big)$ which separate
$[h_1]$ and $[h_2]$, proving that this quotient is totally disconnected.
\end{proof}

\compat{Please read this paragraph and the remark below. The referee asked for clarification here.}
Now suppose that we give the topological surface $S$ a translation structure,
$(S,\alpha)$. 
As in the introduction, we use $\Cov_G(S,\alpha)$ to denote
the space of covers of $(S,\alpha)$ with monodromy in $G$ up to translation equivalence. There is a natural map 
\begin{equation}
\label{eq:top cover to cover}
\Covt_G(S) \to \Cov_G(S,\alpha)
\end{equation}
which sends a topological covering $p:\tilde S \to S$ to $(\tilde S, \tilde \alpha)$
where $\tilde \alpha$ is obtained by pulling back $\alpha$ under $p$.
This map is clearly surjective, and we endow $\Cov_G(S,\alpha)$ with the finest topology so that this natural map is continuous. In other words, translation equivalence induces an equivalence relation on $\Covt_G(S)$, and $\Cov_G(S,\alpha)$
as a topological space is naturally identified with the resulting quotient.

\begin{remark}
\label{rem:primitivity}
Note that our definition of translation automorphism does no consider basepoints.
So, it is possible for two covers $(\tilde S_1, \tilde \alpha_1)$ and $(\tilde S_2, \tilde \alpha_2)$ of $(S,\alpha)$ to be translation equivalent without arising from an isomorphism of covering maps. However such pairs of can be ruled out if the surface $(S,\alpha)$ is geometrically primitive in the sense the deck group of the universal cover of $S$ is the same as the translation automorphism group of the universal cover (when endowed with the pullback translation structure).
\end{remark}

Observe that a homeomorphism $\phi:S \to S'$ between topological surfaces induces a homeomorphism
between their spaces of covers with monodromy in $G$ up to cover isomorphism. This is easily seen through the fundamental group.
In order to consider their fundamental groups, we choose basepoints $s_0$ and $s_0'$. To identify the fundamental groups,
we make a choice of a curve 
$\beta:[0,1] \to S'$ so that $\beta(0)=\phi(s_0)$ and $\beta(1)=s'_0$.
The choice of $\beta$ gives rise to a group homomorphism
\begin{equation}
\label{eq:curve needed}
\phi_\beta:\pi_1(S,s_0) \to \pi_1(S',s'_0); [\gamma] \mapsto [\beta^{-1} \bullet (\phi \circ \gamma) \bullet \beta],
\end{equation}
where $\bullet$ denotes path concatenation (with the path on the left being traversed first). This action allows $\phi$ to act on
the spaces of topological covers with monodromy in $G$. By identifying these spaces of covers as
in equation \ref{eq:space of covers}, we see that this action is given by
\begin{equation}
\label{eq:homeomorphism acting on covers}
\Covt_G(S) \to \Covt_G(S'); \quad
[h] \mapsto [h \circ \phi_\beta^{-1}],
\end{equation}
where $[h]$ is the $\Pi_d$-conjugacy class of a homomorphism 
$h:\pi_1(S,s_0) \to G$.
Note that while the action of $\phi$ on the fundamental group depended on the choice of the curve $\beta$,
the action on this space of covers does not, since the map $h \mapsto h \circ \phi_\beta^{-1}$
only changes by post-conjugation by a permutation when $\beta$ is changed.

A homeomorphism $(S,\alpha) \to (S',\alpha')$ between translation surfaces does not necessarily induce a homeomorphism 
from $\Cov_G(S,\alpha)$ to $\Cov_G(S',\alpha')$. (This fails for instance, if $(S,\alpha)$ admits translation automorphisms and $(S',\alpha')$ does not.) However, an affine homeomorphism does induce such a homeomorphism.
As above, this homeomorphism is induced by the map $h \mapsto h \circ \phi_\beta^{-1}$. We summarize this observation below.

\begin{proposition}
\label{prop:action on covers}
Let $\phi:(S,\alpha) \to (S',\alpha')$ be an affine homeomorphism with derivative $A \in \GL(2,\R)$,
and let $\phi_\beta:\pi_1(S,s_0) \to \pi_1(S,s'_0)$ be the group homomorphism as defined
above for some choice of curve $\beta$. If $(\tilde S, \tilde \alpha)$ is a cover
of $(S,\alpha)$ with monodromy homomorphism $h:\pi_1(S,s_0) \to G$, then 
$A(\tilde S, \tilde \alpha)$ is translation equivalent
to the cover of $(S',\alpha')$ with monodromy homomorphism $h \circ \phi_\beta^{-1}$.
Thus, map $h \mapsto h \circ \phi_\beta^{-1}$ induces a homeomorphism
$A_\ast:\Cov_G(S,\alpha) \to \Cov_G(S',\alpha')$, which depends only on the derivative
$A$ of the affine homeomorphism $\phi$.
\end{proposition}

As a consequence of this proposition, we observe that the Veech group of $(S,\alpha)$ acts on $\Cov_G(S,\alpha)$:

\begin{corollary}[The Veech group acts on covers]
Let $A \in \SL_\pm(2,\R)$ be an element of the Veech group of $(S,\alpha)$, and let $(\tilde S, \tilde \alpha) \in \Cov_G(S,\alpha)$. Then $A (\tilde S, \tilde \alpha)$ is also in $\Cov_G(S,\alpha)$.
\end{corollary}

\subsection{Measures on spaces of covers}
\label{sect:measures}
In this subsection, we will construct some natural
measures on our spaces of covers.
We begin by describing an abstract construction. Later in the subsection, we will specialize the discussion to
our setting of translation surfaces.

Let $\Gamma^+$ be the non-abelian free group with a countable generating set $\{\gamma_i~:~i \in \N\}$, and let $G \subset \Pi_d$ as above. We endow the space $\Hom(\Gamma^+,G)$ with its natural product topology,
which makes the $\Hom(\Gamma^+,G)$ homeomorphic to a Cantor set.
This is the coarsest topology so that 
for each $\eta \in \Gamma^+$ and each $\sigma \in G$, the set of the form 
$$\{h~:~\Hom(\Gamma^+,G)~:~h(\eta)=\sigma\}.$$
is open. Two ordered $k$-tuples of distinct elements $(e_1,\ldots,e_k) \in \N^k$ and $(\sigma_1 ,\ldots, \sigma_k) \in G^k$, determine a
{\em cylinder set} in $\Hom(\Gamma^+,G)$,
\begin{equation}
\label{eqn:cylinders}
\Cyl(e_1,\ldots,e_k; \sigma_1,\ldots,\sigma_k)=\{h~:~\Hom(\Gamma^+,G)~:~
\text{$h(\gamma_{e_i})=\sigma_i$ for $i=1,\ldots,k$}\}.
\end{equation}
Each cylinder set is both closed and open in the product topology, and the collection of cylinder sets generate the topology.

To characterize a Borel measure on $\Hom(\Gamma^+,G)$, it suffices to describe the measures of 
the cylinder sets. 
\begin{definition}
\label{def:random cover}
The {\em product measure} $\mu$ on $\Hom(\Gamma^+,G)$
is defined so that for every cylinder set we have
$$\mu\big(\Cyl(e_1,\ldots,e_k; \sigma_1,\ldots,\sigma_k)\big)=\frac{1}{|G|^k}.$$
This is the product measure induced on $\Hom(\Gamma^+,G)$
by the counting measure on $G$.
\end{definition}

We remark that this measure $\mu$ is interesting even in the case when $\Gamma^+$ is a finitely generated free group, and related questions remain open \cite{Puder13}.

Automorphisms of $\Gamma^+$ act on $\Hom(\Gamma^+,G)$. Concretely, if $\phi: \Gamma^+ \to \Gamma^+$ is an automorphism, then we can define
\begin{equation}
\label{eqn:homAction}
\phi_\ast:\Hom(\Gamma^+, G) \to \Hom(\Gamma^+, G); \quad h \mapsto h \circ \phi^{-1}.
\end{equation}

\begin{lemma}
\label{lem:measPres}
The action of any automorphism of $\Gamma^+$ preserves
the product measure $\mu$ on $\Hom(\Gamma^+,G)$. 
In particular, the measure $\mu$ is independent
of our choice of generating set.
\end{lemma}

\begin{remark}[Proof in abelian case]
In the case when $G$ is abelian, $\Hom(\Gamma^+,G)$
can be identified with the topological group $G^\N$, and 
$\mu$ is Haar measure. In this case, the proposition follows from the naturality of Haar measure.
\end{remark}

\begin{proof}
Let $\phi: \Gamma^+ \to \Gamma^+$ be an automorphism, and let $\phi_\ast$ be its action
on $\Hom(\Gamma^+, G)$:
$$\phi_\ast:\Hom(\Gamma^+, G) \to \Hom(\Gamma^+, G); \quad
h \mapsto h \circ \phi^{-1}.$$
We will prove that $\phi_\ast$ preserves the product measure $\mu$ on $\Hom(\Gamma^+, G)$.
It suffices to prove that the measures of cylinder sets are preserved.
Let $\Cyl=\Cyl(e_1,\ldots,e_k; \sigma_1,\ldots,\sigma_k)$
be a cylinder set. We will prove that $\mu \circ \phi_\ast(\Cyl)=1/|G|^k$. 

Let
$X=\langle \gamma_{e_1},\ldots, \gamma_{e_k} \rangle \subset \Gamma^+$. Since $\Gamma^+=\langle \gamma_e~:~e \in \N\rangle$, there is a finite set $\{e_1',\ldots, e_m'\} \subset \N$ so that
$$\phi^{-1}(X) \subset \langle \gamma_{e_1'}, \ldots, \gamma_{e'_m} \rangle.$$
We'll call the subgroup on the right hand side of the equation $Y$. By viewing $\mu$ as the product of counting measures, we see
\begin{equation}
\label{eq:count}
\mu \circ \phi^{\ast}(\Cyl)=\frac{\#\{h \in \Hom(Y,G)~:~h \circ \phi^{-1}(\gamma_{e_i})=\sigma_i \quad \text{for $1 \leq i \leq k$}\}}{|G|^m}.
\end{equation}
So it suffices to show that the number of homomorphisms in the numerator is $|G|^{m-k}$. 

As above, we can find a finite set $\{e_1'', \ldots, e_n''\} \subset \N$ so that 
$$\phi(Y) \subset \langle \gamma_{e_1''},\ldots, \gamma_{e_n''}\rangle.$$
Call the set on the right hand side $Z$. Note that $X \subset Z$. 

We recall some basic definitions from the theory of free groups. 
A {\em basis} of a free group $F$ is a set $x_1,\ldots,x_k$ so that $F=\langle x_1 \rangle  \ast \ldots \ast \langle x_k \rangle$. 
A subgroup $H$ of a free group $F$ is a {\em free factor} if every (equivalently, some) basis of $H$ can be extended to a basis of $F$. 

Consider $X$, $Y$, and $Z$ as above.
Observe that $X$ is a free factor in $Z$. 
So, $X$ is a free factor in $\phi(Y)$ \cite[Claim 2.5]{Puder13}. 
That is,
we can extend $\{\gamma_{e_1},\ldots, \gamma_{e_k}\}$ to a free generating set of $\phi(Y)$.
Using $\phi^{-1}$, we can pull this back to a generating set of $Y$. So, we have $k \leq n$,
and there is a free generating set of $Y$ given by $\beta_1,\ldots, \beta_m$ so that
$$\beta_i=\phi^{-1}(\gamma_{e_i}) \quad \text{for $1 \leq i \leq k$.}$$
Since this set generates $Y$, we see that $\mathrm{Hom}(Y,G)$ is in bijective correspondence
with the possible images of $\{\beta_1, \ldots, \beta_m\}$. The last $m-k$ elements in this 
basis are irrelevant to the values of $\phi^{-1}(\gamma_{e_i})$, so we see
that there are exactly $|G|^{m-k}$ possible values which give homomorphisms
in the numerator of equation \ref{eq:count}.
\end{proof}

Recall that whenever $S$ is a topological surface of infinite topological type,
then its fundamental group $\pi_1(S,s_0)$ is isomorphic to a countably generated free group. It is a simple observation that conjugation
by a permutation preserves the measure $\mu_G$ constructed above. By identifying 
$\Hom(\Gamma^+,G)$ with $\Hom\big(\pi_1(S,s_0),G\big)$ and 
$\Covt_G(S)$ with $\Pi_d \bs \Hom\big(\pi_1(S,s_0),G\big)$, we obtain a measure
$\nu_G$ on $\Covt_G(S)$, which we call the {\em product measure} on $\Covt_G(S)$.

\begin{corollary}
\label{cor:action on topological covers}
Let $\phi:S \to S'$ be a homeomorphism between two topological surfaces of infinite topological type. Let $G \subset \Pi_d$ for $d \geq 2$, and let $\nu_G$
and $\nu_G'$ be the product measures on $\Covt_G(S)$
and $\Covt_G(S')$, respectively. Then, $\nu_G'$ is the pushforward
of the measure $\nu_G$ under the map $\Covt_G(S) \to \Covt_G(S')$
induced by $\phi$ as in equation \ref{eq:homeomorphism acting on covers}.
\end{corollary}
\begin{proof}
We can identify each space
of covers with $\Hom(\Gamma^+,G)$
using isomorphisms to the fundamental groups.
As noted in equation \ref{eq:homeomorphism acting on covers},
the action of a homeomorphism on the space of covers is induced by a group isomorphism between the fundamental groups and via our identifications, an automorphism of $\Gamma^+$.
Lemma \ref{lem:measPres} 
tells us that the measure $\mu_G$ is invariant under such automorphisms.
Also the automorphism commutes with the (partially defined) $\Pi_d$-action.
Thus, our measures $\nu_G$
and $\nu_G'$ are the same in view of the identification of each space
of covers with $\Pi_d \bs \Hom(\Gamma^+,G)$.
\end{proof}

Now let $(S,\alpha)$ be a translation surface of infinite topological type.
Since $\Cov_G(S, \alpha)$ is a quotient of $\Covt_G(S)$, we obtain an
measure $m_G$ on $\Cov_G(S, \alpha)$ as the pushforward of $\nu_G$.
We call $m_G$ the {\em product measure} on  $\Cov_G(S, \alpha)$.

Recall that Proposition \ref{prop:action on covers} says that when $A \in \SL(2,\R)$ and $A(S,\alpha)$ is translation equivalent to $(S',\alpha')$, there is an induced homeomorphism $A_\ast$
from $\Cov_G(S,\alpha)$ to $\Cov_G(S',\alpha')$. This homeomorphism respects the product measures
on these spaces:

\begin{corollary}[Affine naturality of measures]
Let $(S,\alpha)$ be a translation surface of infinite topological type.
Let $A \in \SL_\pm(2,\R)$ and let $(S',\alpha')=A(S,\alpha)$. Then
$m_G'=m_G \circ A_\ast^{-1}$ where $m_G$ and $m_G'$ are the product measures
on $\Cov_G(S,\alpha)$ and $\Cov_G(S',\alpha')$.
\end{corollary}
\begin{proof}
Because $A(S,\alpha)=(S', \alpha')$, there must be an affine homeomorphism 
$\phi:(S,\alpha) \to (S',\alpha')$ with derivative $A$.
The homeomorphism $A_\ast:\Cov_G(S,\alpha) \to \Cov_G(S',\alpha')$ lifts to a homeomorphism
$\Phi:\Covt_G(S,\alpha) \to \Covt_G(S',\alpha')$ by Proposition \ref{prop:action on covers}.
Since $\Phi_\ast(\nu_G)=\nu_G'$ by Corollary \ref{cor:action on topological covers}
and $m_G$ and $m_G'$ are obtained as images of $\nu_G$ and $\nu_G'$, we see
$m_G'=m_G \circ A_\ast^{-1}$.
\end{proof}

\subsection{Disconnected covers}
\label{sect:disconnected}
Let $G \subset \Pi_d$ and let $h \in \Hom\big(\pi_1(S,s_0),G\big)$. By interpreting $h$ as the monodromy action of the fundamental group of $S$ on the fibers of the basepoint, we obtain a cover $\tilde S$ of $S$ as in \S \ref{sec:spaces}.
This cover is explicitly described by equation (\ref{eq:built cover}), and we can see the following:

\begin{proposition}
\label{prop:connectivity}
The cover associated to $h \in \Hom\big(\pi_1(S,s_0),G\big)$ is connected if and only if the image
$h\big(\pi_1(S,s_0)\big)$ acts transitively on $\{1,2, \ldots, d\}$. 
\end{proposition}

In particular, in order to have connected covers of $(S,\alpha)$ with monodromy in $G$, the subgroup $G \subset \Pi_d$ must act transitively on $\{1,2, \ldots, d\}$. The goal of this subsection is to 
formulate the following precise version of the statement that the collection of all disconnected covers is small.

\begin{proposition}
\label{prop:measure zero}
Let $G \subset \Pi_d$ be a subgroup which acts transitively on $\{1,2, \ldots, d\}$. Let $S$ be a topological surface of infinite topological type. Then, $\nu_G$-almost every cover
in $\Covt_G(S)$ is connected.
\end{proposition}
\begin{proof}
Let $\sH$ denote the collection of all subgroups of $H \subset G$ so that $H$ does not act transitively on $\{1,\ldots, d\}$. Note that
$\sH$ is a finite set. Consider the set $\sD \subset \text{Cov}_G(S,\alpha)$ of disconnected covers with monodromy in $G$.
Recall that $\text{Cov}_G(S,\alpha)$ is a quotient of 
$\Hom\big(\pi_1(S,s_0),G\big)$.
Let $\tilde \sD \subset \Hom\big(\pi_1(S,s_0),G\big)$ be the lift of $\sD$. 
By definition of $\nu_G$, we have $\nu_G(\sD)=\mu_G(\tilde \sD)$.
Proposition \ref{prop:connectivity} tells us that
$$\tilde \sD=\bigcup_{H \in \sH} \Hom\big(\pi_1(S,s_0),H\big).$$
Let $\tilde \sD_H=\Hom\big(\pi_1(S,s_0),H\big)$. By subadditivity of measures, it suffices to prove that 
$\mu_G(\tilde \sD_H)=0$ for all $H \in \sH$.

Fix $H \in \sH$. Observe that $H$ is a proper subgroup of $G$, since $H$ does not act transitively while $G$ does.
Fix some $\epsilon>0$. We will show that $\mu_G(\tilde \sD_H)<\epsilon$. Since $H$ is a proper subset of $G$, we can find a $k$ so that $(\frac{|H|}{|G|})^k<\epsilon$. Observe that $\tilde \sD_H$ is contained in the union of cylinder sets
$$\bigcup_{(h_1, \ldots, h_k) \in H^k} 
\Cyl(1,\ldots,k; h_1,\ldots,h_k),$$
where we are using notation from equation \ref{eqn:cylinders}.
Observe that by monotonicity and by Definition \ref{def:random cover} of $\mu_G$, we have
$$\mu_G(\tilde \sD_H) \leq \sum_{(h_1, \ldots, h_k) \in H^k} \mu_G\big(\Cyl(1,\ldots,k; h_1,\ldots,h_k)\big)=
\frac{|H|^k}{|G|^k}<\epsilon.$$
This proves that $\mu_G(\tilde \sD_H)=0$, and thus $\nu_G(\sD)=0$ by the remarks in the previous paragraph.
\end{proof}

\section{Ergodicity}
\label{sec:erg}
Let $(S,\alpha)$ be a flat surface and let $G$ be a subgroup of the permutation group $\Pi_d$ for some integer $d \geq 2$. The group $\SL_\pm(2,\R)$ acts on the 
the space of affine deformations of covers with monodromy in $G$, $\tilde \sO_G(S,\alpha)$,
and the action of the diagonal subgroup, $g^t$, is the cover cocycle.
(See (\ref{eq:covers cocycle}).)

In this section, we prove Theorem \ref{thm:1}, which pertains to a
connected cover $(\tilde S, \tilde \alpha) \in \Cov_G(S,\alpha)$: If the Teichm\"uller trajectory (covers cocycle orbit) $g^t(\tilde S, \tilde \alpha)$ has an accumulation point in $\tilde \sO_G(S,\alpha)$ representing a connected surface, then the translation flow on $(\tilde S, \tilde \alpha)$ is defined for all time almost everywhere and is ergodic.

\compat{Please check this paragraph.}
We will see that Theorem \ref{thm:1} is a consequence of the following result,
which gives a criterion for ergodicity in terms of the geometries realized under the Teichm\"uller deformation. Before stating the Theorem we establish some notation. Let $\Sigma\subset \bar{S}$ be the subset of the metric completion of $S$ defined as the union of the zeros of $\alpha$ and the points in $\bar S \smallsetminus S$.
For $t \in \R$ we define $\mbox{dist}_t$ to be the metric on $S$ obtained by pulling back the flat metric on $g^t(S,\alpha)$ under the affine homeomorphism $(S,\alpha) \to g^t(S,\alpha)$ with derivative $g_t$.

\begin{theorem}[{\cite[Theorem 3]{rodrigo:erg}}]
\label{thm:integrability}
Let $(S,\alpha)$ be a flat surface of finite area. Suppose that for any $\eta>0$ there exist a function $t\mapsto \varepsilon(t)>0$, a one-parameter family of subsets
$$S_{t} = \bigsqcup_{i=1}^{C_t}S_t^i$$ 
of $S$ made up of $C_t < \infty$ path-connected components, each homeomorphic to a closed orientable surface with boundary, and functions $t\mapsto \mathcal{D}_t^i>0$, for $1\leq i \leq C_t$, such that for 
$$\Gamma_t^{i,j} = \{\mbox{paths connecting }\partial S_t^i \mbox{ to }\partial S_t^j\}$$
and
\begin{equation}
\label{eqn:systole}
\delta_t = \min_{i\neq j} \sup_{\gamma\in\Gamma_t^{i,j} }\mbox{dist}_t(\gamma,\Sigma)
\end{equation}
the following hold:
\begin{enumerate}
\item  $\mathrm{Area}(S\backslash S_{t}) < \eta\,\mathrm{Area}(S)$ for all $t>0$,
\item $\mbox{dist}_t(\partial S_{t},\Sigma) > \varepsilon(t)$ for all $t>0$, 
\item the diameter of each $S_t^i$, measured using $\dist_t$, is bounded above by $\mathcal{D}_t^i$ and
\begin{equation}
\label{eqn:integrability2}
\int_0^\infty \left( \varepsilon(t)^{-2}\sum_{i=1}^{C_t}\mathcal{D}_t^i + \frac{C_t-1}{\delta_t}\right)^{-2}\, dt = +\infty.
\end{equation}
\end{enumerate}
Then the translation flow is defined for all time almost everywhere and is ergodic.
\end{theorem}

The theorem above is a geometric criterion for ergodicity. The spirit of the theorem is that if, as ones deforms a flat surface $(S,\alpha)$ using the Teichm\"{u}ller deformation $g^t$, the geometry of the surface does not deteriorate too quickly (as measured by the diameter of big components, among other things), the translation flow is ergodic. 

In \cite{rodrigo:erg}, this theorem was proved with the additional hypothesis that the set of points whose trajectories leave every compact subset of $S$ has zero measure. This is equivalent to the statement that the translation flow is defined for all time almost everywhere (which we state as a conclusion above). To get the version above we need to prove that the hypothesis is unnecessary:

\begin{proof}[Proof that the translation flow is defined for all time almost everywhere]
We will show that the geometric conditions listed above force the translation flow to be defined for all time almost everywhere.
Suppose $(S,\alpha)$ is a finite area translation surface and satisfies the list of geometric conditions given in the Theorem. Assume to the contrary that the translation flow $F^s$ is not defined for all time almost everywhere. Then there is a time $s_0$ and a measurable subset $X \subset S$ of Lebesgue measure $m>0$ so that $F^{s_0}(x)$ is undefined for all $x \in X$. There is a geometric consequence to lying in $X$: For any $x \in X$ there is a $s_x\in \R$ with $0<s_x\leq s_0$ so that $F^s(x)$ is defined for $s \in [0,s_x)$ and 
$$\lim_{s \to {s_x}^-} F^s(x) \in \Sigma.$$
In particular, the distance from $x$ to $\Sigma$ measured with $\dist_t$ is no larger than $s_x e^{-t} \leq s_0 e^{-t}$. 

Select an $\eta$ so that $m>\eta\,\mathrm{Area}(S)$. Then any subsurface $S_t$ satisfying condition (1) of the theorem must satisfy $S_t \cap X \neq \nullset$. As a consequence of condition (2) and remarks above, we must have $\varepsilon(t)<s_0 e^{-t}$.

We will draw a contradiction to the integral \eqref{eqn:integrability2} being infinite. To do this it suffices to get some control of the sum of the diameters. We partition the set $\{t \in \R:~t \geq 0\}$ into two pieces ${\mathcal S}$ and ${\mathcal L}$ (for ``small'' and ``large''). We declare $t$ to lie in ${\mathcal S}$ if all components $S_t^i$ have diameter less than or equal to $\varepsilon(t)$,
and declare $t$ to lie in ${\mathcal L}$ otherwise. We will control the sum $\sum_i {\mathcal D}_t^i$ by separate arguments on the two sets.

Suppose $t \in {\mathcal S}$. Then all components of $S_t$ have diameter less than $\varepsilon(t)$. Let $x_i \in S_t^i$ be a point in one of the components. Since $\dist_t(x_i,\Sigma) > \varepsilon(t)$ by hypothesis, we can embed a Euclidean $B_i$ ball of radius 
${\mathcal D}_t^i<\varepsilon(t)$ about $x_i$ within $S$ using the distance $\dist_t$. 
Then we have $S_t^i \subset B_i$, so 
$$\mathrm{Area}(S_t^i) \leq \mathrm{Area}(B_i) = \pi ({\mathcal D}_t^i)^2.$$
Observe this holds for all $i$ and so it follows that the sum $\sum_i {\mathcal D}_t^i$ is bounded from below by a constant:
\begin{equation}
\label{eq:diameter bound 1}
\Big(\sum_i {\mathcal D}_t^i\Big)^2 \geq \sum_i ({\mathcal D}_t^i)^2 \geq \frac{4}{\pi} \sum_i \mathrm{Area}(S_{t}^i)=
\frac{1}{\pi} \mathrm{Area}(S_{t}) >\frac{1}{\pi} (1-\eta) \mathrm{Area}(S).
\end{equation}

Now suppose $t \in {\mathcal L}$. Then $S_t$ has a component of  diameter at least $\varepsilon(t)$. 
In this case we can use the very na\"ive bound
\begin{equation}
\label{eq:diameter bound 2}
\sum_i {\mathcal D}_t^i \geq \varepsilon(t).
\end{equation}

Combining \eqref{eq:diameter bound 1} and \eqref{eq:diameter bound 2}, we see that $\sum_i {\mathcal D}_t^i \geq D(t)$
where $D(t)$ is defined by
$$D(t)=\min~\left\{\varepsilon(t), \sqrt{\frac{1}{\pi} (1-\eta) \mathrm{Area}(S)}\right\}.$$
The quantity on the right is a constant while $\varepsilon(t)$ tends to zero since $\varepsilon(t)<s_0 e^{-t}$.
Thus $D(t)=\varepsilon(t)$ for $t>t_\ast$ for some $t_\ast \geq 0$. For $t>t_\ast$, $\sum_i {\mathcal D}_t^i \geq \varepsilon(t)$
and so the quantity being integrated in \eqref{eqn:integrability2} satisfies 
$$\left( \varepsilon(t)^{-2}\sum_{i=1}^{C_t}\mathcal{D}_t^i + \frac{C_t-1}{\delta_t}\right)^{-2} \leq \left(\varepsilon(t)^{-2} \varepsilon(t)\right)^{-2}=\varepsilon(t)^2.$$
Recalling $\varepsilon(t)<s_0 e^{-t}$, we see the total integral \eqref{eqn:integrability2} is bounded by
$$\int_0^{t_\ast} \left( \varepsilon(t)^{-2}\sum_{i=1}^{C_t}\mathcal{D}_t^i + \frac{C_t-1}{\delta_t}\right)^{-2}\,dt+
s_0^2 \int_{t_\ast}^\infty e^{-2t}\,dt,$$
which is finite. This is our contradiction.
\end{proof}

A key consequence for us is the following which was mentioned in the introduction:

\begin{corollary}[{\cite[Theorem 2]{rodrigo:erg}}]
\label{cor:rodrigo ergodicity from Veech group}
Suppose $(S,\alpha)$ is a finite area translation surface of infinite topological type. Let $a \in \SL_\pm(2,\R)$. Then, if the trajectory $g^t a (S,\alpha)$ is non-divergent in $\sO(S,\alpha)$, then the translation flow on $a (S,\alpha)$ is defined for all time almost everywhere and is ergodic.
\end{corollary}

It is worth noting that in \cite{rodrigo:erg} this was proved independently from Theorem \ref{thm:integrability}.
We will prove this here using Theorem \ref{thm:integrability} because we will use some of the same ideas in the proof of Theorem \ref{thm:1}.

\begin{proof}
By replacing $(S,\alpha)$ with $a(S,\alpha)$, we can assume that $a$ is the identity.
We will use square brackets to denote the translation equivalence class of a translation surface in $\sO(S,\alpha)$.
Non-divergence guarantees that there is sequence of times $t_k$ tending to $+\infty$ so that the translation equivalence class 
$[g^{t_k} (S,\alpha)]$ converges in $\sO(S,\alpha)$ to some limit $[b(S, \alpha)]$. 
Since $\N \cup \{+\infty\}$ with its usual topology is compact and the map $\SL(2,\R) \to \sO(S,\alpha)$ given by $a \mapsto [a(S,\alpha)]$
is continuous, $K=\{[g^{t_k} (S,\alpha)]~:~k \in \N\} \cup \{[b(S, \alpha)]\}$ is compact in $\sO(S,\alpha)$.
Fix any $\eta>0$. Select a compact connected subsurface with boundary $L \subset S$ whose area is greater than $1- \eta$ times the area of $S$, and so that $L$ does not include any zeros of $\alpha$. Fix any $\epsilon>0$. For $t \in [t_k-\epsilon, t_k +\epsilon]$ select the subsurface $S_t=g^{-t_k} b_k(L)$ (or any such subsurface if $t$ belongs to multiple such intervals). 
Then since $g^{t} (S,\alpha)$ is translation equivalent to $g^{t-t_k} b_k(S, \alpha)$, there is an isometry from $S$ with the metric $\dist_t$
to the translation surface $g^{t-t_k} b_k(S,\alpha)$ which carries $S_t$ to the image of $L$ under $g^{t-t_k} b_k$.
So for all $t$ in any $[t_k-\epsilon, t_k +\epsilon]$, the subsurface we have selected is isometric to the image
of $m(L)$ under an $m \in g^{[-\epsilon,\epsilon]} K$ viewed as a subsurface of $m(S,\alpha)$.
Observe that the quantities in Theorem \ref{thm:integrability} vary continuously in $m$ as we deform the metric in this way.
Thus, the quantity being integrated over times $t\in \bigcup_k [t_k-\epsilon, t_k +\epsilon]$ is bounded uniformly from below by a uniform positive constant. Since $t_k$ is an infinite sequence tending to $+\infty$, the integral is infinite.
\end{proof}

\begin{proof}[Proof of Theorem \ref{thm:1}]
\compat{This proof was completely rewritten. Please read.}
Let $(S,\alpha)$ be a finite area translation surface with infinite topological type,
and let $(\tilde S, \tilde \alpha)$ be a cover with monodromy in $G \subset \Pi_d$.
We assume that $[g^t(\tilde S, \tilde \alpha)]$ has an $\omega$-limit point in $\tilde \sO_G(S,\alpha)$ which is the translation equivalence class of a connected surface. Let $t_k$ be a sequence of times tending to $+\infty$ for which 
$[g^{t_k}(\tilde S, \tilde \alpha)]$ approaches this $\omega$-limit. 
Then by definition of the topology on $\tilde \sO_G(S,\alpha)$, there is a sequence
$$\big(b_k, [(\tilde S_k, \tilde \alpha_k)]\big) \in \SL(2,\R) \times \Cov_G(S,\alpha)$$ 
so that for all $k$ the surface
$b_k (\tilde S_k, \tilde \alpha_k)$ is translation equivalent to $g^{t_k}(\tilde S, \tilde \alpha)$,
the sequence $b_k$ converges to some $b \in \SL(2,\R)$ and the sequence $[(\tilde S_k, \tilde \alpha_k)]$ converges to some connected $[(\tilde S_\infty, \tilde \alpha_\infty)] \in \Cov_G(S,\alpha)$. 

We will briefly discuss the convergence of $[(\tilde S_k, \tilde \alpha_k)]$
to the connected cover $[(\tilde S_\infty, \tilde \alpha_\infty)]$ within $\Cov_G(S,\alpha)$. Fix a basepoint $s_0 \in S$ which is not a zero. Choose a representative cover $(\tilde S_\infty, \tilde \alpha_\infty)$ from $[(\tilde S_\infty, \tilde \alpha_\infty)]$.
Since $[(\tilde S_\infty, \tilde \alpha_\infty)] \in \Cov_G(S,\alpha)$, we can 
denote the lifts of $s_0$ by $s_\infty^1,\ldots, s_\infty^d \in \tilde S_\infty$ 
obtaining a monodromy homomorphism 
$h_\infty:\pi_1(S,s_0) \to G \subset \Pi_d$. From the definition of the topology $\Cov_G(S,\alpha)$,
for all $k$ we can select covers $(\tilde S_k, \tilde \alpha_k)$ from the equivalence classes $[(\tilde S_k, \tilde \alpha_k)]$ and 
denote the lifts of the basepoint $s_0$ by $s_k^1,\ldots, s_k^d \in \tilde S_k$ in such a way so that
the corresponding monodromy homomorphisms $h_k:\pi_1(S,s_0) \to G$ converge
to $h_\infty$ within $\Hom\big(\pi_1(S,s_0), G\big)$. 

We will be using Theorem \ref{thm:integrability}.
Fix an $\epsilon>0$. As in the prior proof we will only bother to choose a subsurface when $t \in [t_k-\epsilon,t_k+\epsilon]$ for some $k$. 
We will now describe how we choose these subsurfaces. As in the prior proof, we can let $L \subset S$ be a compact connected subsurface with boundary whose area is more than $1-\eta$ times the area of $S$ so that $L$ contains no zeros of $\alpha$. A compact surface with boundary 
has finite genus, so by removing small open neighborhoods of a maximal collection of smooth disjoint arcs joining $\partial L$ to itself whose collective complement in $L$ is connected, we may assume that $L$ has all these properties and is a homeomorphic to a closed topological disk. We can also assume by possibly adding a bit to the subsurface that $s_0 \in L$. From above
$(\tilde S, \tilde \alpha)$ is translation equivalent to $g^{-t_k} b_k(\tilde S_k, \tilde \alpha_k)$ and so these surfaces cover $g^{-t_k} b_k(S, \alpha)$ by Proposition \ref{prop:action on covers}.
Let $L_k \subset (\tilde S_k, \tilde \alpha_k)$ be the collection of all lifts of $L$ under the covering map to $S$. 
The components of $L_k$ are naturally labeled $L_k^1, \ldots, L_k^d$ so that the lift of the basepoint $s^i_k$ lies in $L_k^i$ for all $i \in \{1,\ldots, d\}$.
For $t \in [t_k-\epsilon,t_k+\epsilon]$ and $i \in \{1,\ldots, d\}$
we define $S_t^i \subset \tilde S$ to be image of $g^{-t_k} b_k(L_k^i)$ under a translation isomorphism 
$g^{-t_k} b_k (\tilde S_k, \tilde \alpha_k) \to (\tilde S, \tilde \alpha)$. We define $S_t=\bigcup_{i=1}^d S_t^i$.
Observe that with this definition:
\begin{itemize}
\item The surface $S_t$ has a number of components $C_t$ equal to the degree $d$ of the covering maps.
\end{itemize}
Consider geometric quantities of $S_t \subset \tilde S$ measured with $\dist_t$. The surface $\tilde S$ with this metric is isometric
to $g^t(\tilde S, \tilde \alpha)$ which is translation equivalent to $g^{t-t_k} b_k (\tilde S_k, \tilde \alpha_k)$.
So by definition of $S_t$ these geometric quantities are the same as for the subsurface $g^{t-t_k} b_k(L_k)$
of $g^{t-t_k} b_k (\tilde S_k, \tilde \alpha_k)$. Observe that $g^{t-t_k} b_k$ lies in the compact set $g^{[-\epsilon,\epsilon]} K$ where $K=\{b_k:~k \in \N\} \cup \{b\}$ is compact as in the prior proof. Observe:
\begin{itemize}
\item The quantities $\epsilon(t)$ and $\mathcal{D}_t^i$ used to measure the components of $g^{t-t_k} b_k(L_k)$ as a subsurface
of $g^{t-t_k} b_k (\tilde S_k, \tilde \alpha_k)$ are exactly the same as the quantities for $g^{t-t_k} b_k(L)$ viewed as a subsurface of 
$g^{t-t_k} b_k (S, \tilde \alpha)$, because these quantities are covering map invariant.
\end{itemize}
In particular, this means that $\varepsilon(t)$ can be bounded uniformly away from zero when $t \in \bigcup_k [t_k-\epsilon,t_k+\epsilon]$
and $\mathcal{D}_t^i$ can be bounded uniformly away from $+\infty$ as in the proof of Corollary \ref{cor:rodrigo ergodicity from Veech group}.

It remains to control the quantity $\delta_t$. We must first construct the curves $\Gamma^{i,j}_t$. We utilize the convergence
of $h_k$ to $h_\infty$. 
Since $S_\infty$ is connected, we can select for all distinct $i,j \in \{1, \ldots, d\}$ a path $\Gamma_\infty^{i,j}$ in $\tilde S_\infty$ disjoint from the zeros joining $s_\infty^i$ to $s_\infty^j$. Let $\tilde \delta_0>0$ be the minimum over all pair $(i,j)$ of
the distance from $\Gamma_\infty^{i,j}$ to the set of points of the completion of $\tilde S$ which are either zeros or added in the completion. Let $\gamma^{i,j}$ be the loop in $S$ based at $s_0$ which is obtained as the image of $\Gamma_\infty^{i,j}$ under the covering map $\tilde S_\infty \to S$. Since $h_k$ tends to $h_\infty$, for $k$ sufficiently large $h_k(\gamma^{i,j})=h_\infty(\gamma^{i,j})$ for all $i$ and $j$. We will assume by dropping finitely many $k$ that this holds for all $k$. Then for each $k$, the lift $\tilde \gamma^{i,j}_k$ of $\gamma^{i,j}$ to $\tilde S_k$ which begins at $s_k^i$ ends at $s_k^j$. Observe that the minimal distance over all pairs $(i,j)$ of
the distance of $\gamma^{i,j}$ to zeros or points added in the completion is still $\tilde \delta_0$.
A subpath $\tilde p^{i,j}_k \subset \tilde \gamma^{i,j}_k$ joins $\partial L^i_k$ to $\partial L^j_k$ since the path joins $s_k^i \in L^i_k$ to $s_k^j \in L^j_k$.
For $t \in [t_k-\epsilon,t_k+\epsilon]$, we define $\Gamma_t^{i,j} \subset S$ to be the image of $g^{-t_k} b_k(\tilde p^{i,j}_k)$
under the translation isomorphism $g^{-t_k} b_k (\tilde S_k, \tilde \alpha_k) \to (\tilde S, \tilde \alpha)$ (as used above to define $S_t$). The quantity $\delta_t$ is the minimal distance of $\Gamma_t^{i,j}$ to $\Sigma \subset \bar S$ taken over pairs $(i,j)$
and measured with $\dist_t$.
As above $S$ with metric $\dist_t$ is isometric to $g^t(\tilde S, \tilde \alpha)$ which in turn is translation isomorphic to $g^{t-t_k} b_k (\tilde S_k, \tilde \alpha_k)$. Thus we get the same value of $\delta_t$ by looking at the path $g^{t-t_k} b_k(\tilde p^{i,j}_k)$ in the translation surface $g^{t-t_k} b_k (\tilde S_k, \tilde \alpha_k)$. Thus $\delta_t$ is bounded from below by
$\frac{1}{z} \tilde \delta_0$ when $t \in \bigcup_k [t_k-\epsilon,t_k+\epsilon]$ where $z$ is the maximal operator norm of $m^{-1}$ taken over $m$ taken from the compact set $g^{[-\epsilon,\epsilon]} K$ as above.

We have shown that the quantity being integrated in \eqref{eqn:integrability2} has a positive lower bound when $t \in \bigcup_k [t_k-\epsilon,t_k+\epsilon]$. Since $t_k \to +\infty$, this means that the integral is infinite, so Theorem \ref{thm:integrability}
guarantees that the translation flow is defined for all time almost everywhere and is ergodic.
\end{proof}

\section{Further proofs}
\label{sect:proofs}
In this section, we prove Proposition \ref{prop:lifting measures}, Proposition \ref{prop:accumulation points} and Theorem \ref{thm:2} from the introduction.

\begin{proof}[Proof of Proposition \ref{prop:lifting measures} (Ergodicity and unique lifts of measures)]
Let $(S,\alpha)$ be a translation surface, let $(\tilde S, \tilde \alpha)$ be a degree $d$ cover, and
let $p:\tilde S \to S$ denote the covering map. We assume that the translation flow is ergodic on both of these surfaces.
We will prove that Lebesgue measure is the unique translation flow invariant measure which projects to Lebesgue measure on $(S,\alpha)$ under the covering map. 

Fix a non-singular basepoint $s \in S$,
and let $h:\pi_1(S,s) \to \Pi_d$ be the monodromy representation.
We will consider the regular (or normal) cover of $S$ associated to the subgroup $\ker~h \subset \pi_1(S,s)$.
Let $(\hat S, \hat \alpha)$ denote this cover. Because the subgroup $\ker~h$ is normal, there is a covering group action
of $\Delta=\pi_1(S,s)/\ker~h$ on $\hat S$, and the quotient $\hat S/\Delta$ is naturally identified with $S$.
The covering $\hat S \to S$ factors through $\tilde S$. That is, there is a subgroup $\Gamma \subset \Delta$ so that
the $\tilde S$ is isomorphic as a cover to $\hat S/\Gamma$. We let $\hat p:\hat S \to \tilde S$ denote the covering
obtained by identifying $\tilde S$ with $\hat S/\Gamma$.

Now suppose that $\tilde \mu$ is a measure on $\tilde S$ which is invariant for the translation flow,
and satisfies $p_\ast(\tilde \mu)=\lambda$, where $\lambda$ denotes Lebesgue measure on $(S,\alpha)$.
Note that the measure $\tilde \mu$ lifts to a unique measure $\hat \mu$ on $(\hat S, \hat \alpha)$ so that $\hat p_\ast(\hat \mu)=\tilde \mu$
and so that
$\gamma_\ast(\hat \mu)=\hat \mu$ for all $\gamma \in \Gamma$. 
The key point in the proof is that because $\hat \mu$ projects through to Lebesgue measure on $S$, we know that if we average
the push forwards of $\hat \mu$ under the covering group $\Delta$, we get Lebesgue measure on $(\hat S, \hat \alpha)$, which we denote by $\hat \lambda$. That is,
$$\frac{1}{|\Delta|} \sum_{\delta \in \Delta} \delta_\ast(\hat \mu)=\hat \lambda.$$
Now consider the push-forward of these measures under the covering map $\hat p$. Since $p_\ast(\hat \mu)=\tilde \mu$, we see that
$$\frac{1}{|\Delta|} \left(\tilde \mu+\sum_{\delta \in \Delta \smallsetminus \{e\}} \hat p_\ast \circ \delta_\ast(\hat \mu)\right)=\tilde \lambda,$$
where $\tilde \lambda$ is the Lebesgue measure on $(\tilde S, \tilde \alpha)$. Finally, $\tilde \lambda$ is ergodic,
so each of the probability measures in the above convex combination must equal $\tilde \lambda$. In particular,
$\tilde \mu=\tilde \lambda$, which concludes the proof.
\end{proof}

\begin{proof}[Proof of Proposition \ref{prop:accumulation points}]
We assume that $(S,\alpha)$ is a finite area translation surface with infinite topological type
and that it has Teichm\"uller trajectory which is non-divergent in $\sO(S,\alpha)$. 
By non-divergence, there is a subsequence of times $t_n \to \infty$
so that $g^{t_n}(S,\alpha)$ tends to some $A_\infty (S,\alpha) \in \sO(S,\alpha)$,
where $A_\infty \in \SL_\pm(2,\R)$.
Because the topology on $\sO(S,\alpha)$ arises as a quotient of the topology on $\SL_\pm(2,\R)$,
we see that there is a sequence $A_n \in \SL(2,\R)$ tending to $A_\infty$ so that $g^{t_n}(S,\alpha)$ is translation equivalent to $A_n (S, \alpha)$. It then follows that there is a sequence
of elements $R_n$ of the Veech group of $(S,\alpha)$ so that
$g^{t_n}=A_n R_n.$

Now consider a cover $(\tilde S, \tilde \alpha)$
with monodromy in $G$. 
Proposition \ref{prop:action on covers} explains that the Veech group acts on the space of covers with monodromy in $G$. In particular, for each $n$, there is a cover $(\tilde S_n, \tilde \alpha_n) \in \Cov_G(S, \alpha)$ which is translation equivalent to $R_n (\tilde S, \tilde \alpha)$.
Then, in the space $\tilde \sO_G(S,\alpha)$, we have that 
$$g^{t_n}(\tilde S, \tilde \alpha)=A_n R_n (\tilde S, \tilde \alpha)=A_n (\tilde S_n, \tilde \alpha_n).$$
The key observation is that $\Cov_G(S, \alpha)$ is a quotient of a Cantor set and thus sequentially compact,
so there must be a limit point $(\tilde S_\infty, \tilde \alpha_\infty)$ for the sequence  $(\tilde S_n, \tilde \alpha_n) \in \Cov_G(S, \alpha)$. Since $A_n$ tends to $A_\infty$ in $\SL_\pm(2,\R)$,
we see that $g^{t_n}(\tilde S, \tilde \alpha)$ tends to $A_\infty (\tilde S_\infty, \tilde \alpha_\infty)$, which 
is our desired accumulation point.
\end{proof}

\begin{proof}[Proof of Theorem \ref{thm:2} (Random covers accumulate on connected covers)]
Let $(S,\alpha)$ be a finite area translation surface with infinite topological type
and a Teichm\"uller trajectory which is non-divergent in $\sO(S,\alpha)$.
As in the previous proof, this guarantees that there is a sequence $t_n \to \infty$ so that
$g^{t_n}=A_n R_n$ where $\{A_n \in \SL_\pm(2,\R)\}$ is a sequence tending to $A_\infty \in \SL_\pm(2,\R)$, and $R_n \in V(S,\alpha)$.

We will work with the space $\Covt_G(S)$ of topological covers of $(S,\alpha)$ which comes equipped with a measure $\nu_G$. In order to do this, observe that for each Veech group element, $R_n \in V(S,\alpha)$ we can find an affine homeomorphism $\phi_n:(S,\alpha) \to (S,\alpha)$ so that $D(\phi_n)=R_n$. 
Recall that we denote elements of $\Covt_G(S)$ by pairs $(p, \tilde S)$ where
$p:\tilde S \to S$ is a covering map.
With an additional choice of a curve for each $n$, we obtain from $\phi_n$ an action 
$$\Phi_n:\Covt_G(S) \to \Covt_G(S)$$
as in equation \ref{eq:homeomorphism acting on covers}.
Let ${\mathcal T}:\Covt_G(S) \to \Cov_G(S,\alpha)$ be the map defined in \eqref{eq:top cover to cover}. Then for each $(p, \tilde S) \in \Covt_G(S)$ we have that 
$$R_n(\tilde S,\tilde \alpha)=(\tilde S_n, \tilde \alpha_n) \quad \text{where
$(\tilde S,\tilde \alpha)={\mathcal T}(p,\tilde S)$ and 
$(\tilde S_n,\tilde \alpha_n)={\mathcal T}\circ \Phi_n(p,\tilde S)$.}$$
This follows from Proposition \ref{prop:action on covers}
and equation \ref{eq:homeomorphism acting on covers}.
In particular, with these hypotheses, we have that $g^{t_n}(\tilde S, \tilde \alpha)$ is translation equivalent to $A_n (\tilde S_n, \tilde \alpha_n)$. Since the sequence $\{A_n\}$ converges in $\SL_\pm(2,\R)$, it suffices to find a connected
accumulation point of the sequence $\{\Phi_n(p, \tilde S)\}$ for $\nu_G$-almost every $(p, \tilde S) \in \Covt_G(S)$.
(This is because our measure $m_G$ on the space of covers up to translation equivalence, $\Cov_G(S,\alpha)$, is the push forward of $\nu_G$ under the projection $\Covt_G(S) \to \Cov_G(S,\alpha)$.)

Let $\sD \subset \Covt_G(S)$ denote the collection of disconnected surfaces.
Note that these surfaces have $\nu_G$-measure zero by Proposition \ref{prop:measure zero}.
Let $\sE \subset \Covt_G(S)$ denote the collection of covers $(p, \tilde S) \in \Covt_G(S)$
so that every accumulation point of $\{\Phi_n(p, \tilde S)\}$ is disconnected. 
Then if $\sU \subset \Covt_G(S)$ is open and contains $\sD$, we have that
for every $(p, \tilde S) \in \sE$, 
there must be an $N$ so that $\Phi_n(p, \tilde S) \in \sU$ for all $n > N$. 
Indeed, if this were not true, then infinitely many $\Phi_n(p, \tilde S)$ lie in
the complement of $\sU$, which is sequentially compact since $\Cov_G(S,\alpha)$ is a Cantor set by Proposition \ref{prop:topological covers}. In other words, we have
$$\sE \subset \bigcup_N \bigcap_{n>N} \Phi_n^{-1}(\sU)= \liminf_{n \to \infty} \Phi_n^{-1}(\sU).$$

Now fix some $\epsilon>0$. We will show that $\nu_G(\sE)<\epsilon$.
Because $\nu_G$ is a Borel probability measure on a Cantor set,
$\nu_G$ is regular. Thus because $\nu_G(\sD)=0$, we can find an open set $\sU$ containing $\sD$
so that $\nu_G(\sU)<\epsilon$. Then from the above, we have
$$\nu_G(\sE) \leq \nu_G \left( \liminf_{n \to \infty} \Phi_n^{-1}(\sU) \right) \leq \liminf_{n \to \infty} \nu_G \circ \Phi_n^{-1}(\sU).$$
But, Corollary \ref{cor:action on topological covers} tells us that
$\nu_G$ is $\Phi_n$-invariant. Thus, the above sequence of inequalities
tells us that $\nu_G(\sE) \leq \nu_G(\sU) < \epsilon$. Since $\epsilon$
was arbitrary, we conclude that $\nu_G(\sE)=0$ as desired.
\end{proof}

\section{Examples of devious covers}
\label{sect:evil covers}

\subsection{Chamanara's surfaces}
\label{sect:Chamanara}
\compat{The surfaces studied here should all be named $(S_n,\alpha_n)$ with the primary case being $n=2$. Only a generic surface should be $(S,\alpha)$.}
We introduce a surface $(S_2,\alpha_2)$ first studied by Chamanara in \cite{Chamanara04}. (See also the related work \cite{CGL}.)
The surface is built from a closed $1 \times 1$ square with each of the edges subdivided into intervals of length $\frac{1}{2^k}$ for $k \in \N$ as indicated in left side of Figure \ref{fig:Chamanara}. The vertical intervals of equal length are then glued together by translation, and we do the same to the horizontal intervals. The intervals being identified have been labeled by the same integers in the figure. The endpoints of these intervals being glued 
and the corners of the square are discarded to give the space a translation structure. 

\begin{figure}
\begin{center}
\includegraphics[width=3in]{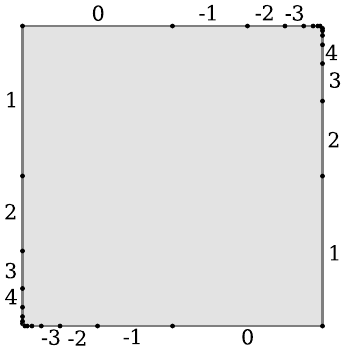}
\includegraphics[width=3in]{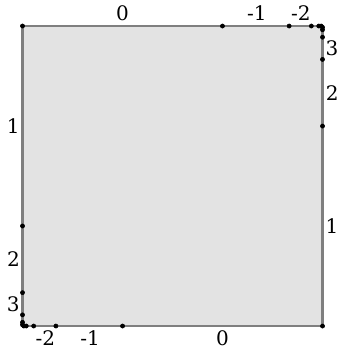}
\end{center}
\caption{Chamanara's surfaces $(S_2,\alpha_2)$ and $(S_3, \alpha_3)$.}
\label{fig:Chamanara}
\end{figure}

The surface $(S_2,\alpha_2)$ has an affine automorphism $\phi_2$ whose derivative is
$$D(\phi_2)=\left[\begin{array}{rr} \frac{1}{2} & 0 \\ 0 & 2 \end{array}\right] \in V(S_2,\alpha_2).$$
To see this observe that the image of the square under this linear map is a $\frac{1}{2} \times 2$ rectangle. If we push this rectangle into the surface allowing the rectangle only to pass through the edge labeled zero in the figure,
we see the identifications are respected and thus this map determines
an affine automorphism $\phi_2$ of the surface.

The vertical and horizontal flows on Chamanara's surface $(S_2,\alpha_2)$ can be seen as the suspension flow over the dyadic odometer
as defined in \eqref{eq:odometer}, except that our construction introduces countably many singularities through which the flows are not defined. To see this, it suffices to consider the dyadic odometer as an infinite interval exchange map on $[0,1]$ defined for $x\in[0,1]$, as the dyadic odometer on the dyadic expansion $(x_1,x_2,\dots) \in X_2=\{0,1\}^\N$ of $x = \sum x_i 2^{-i}$. As such, the map induced on a transversal $T_2$ running from the bottom to the top of the square making up $(S_2,\alpha_2)$ such as the one depicted in Figure \ref{fig:skew}
is isomorphic to the dyadic odometer. See \cite[\S 2]{LT:models} for a thorough description of the dyadic odometer and its relationship to Chamanara's surface.

More generally, for integers $n \geq 2$, one can construct a homeomorphic  translation surface $(S_n,\alpha_n)$ by letting the identified sides in Figure \ref{fig:Chamanara} be of length $\frac{n-1}{n^k}$ for $k\in\mathbb{N}$. The right side of Figure \ref{fig:Chamanara}
illustrates the case of $n=3$. The surface $(S_n, \alpha_n)$ admits
an affine automorphisms $\phi_n$ with diagonal derivative and eigenvalues of $n$ and $\frac{1}{n}$. As in the case of the dyadic odometer, this surface admits a section which is the $n$-adic odometer.

\begin{remark}[Veech groups]
The Veech group $V(S_n,\alpha_n)$ is known to be generated by two parabolics. See \cite{HR:chamanara}.
\end{remark}

\compat{Moved this from later in the section.}
We will now introduce some more notation which will be useful for the Proof of Theorem \ref{thm:skew} and for work later in this section.
Let $G$ be a subgroup of the symmetric group $\Pi_d$ with $d \geq 2$ which acts simply transitively on $\{1,\ldots, d\}$. 
For a general surface $(S,\alpha)$ the space $\Cov_G(S,\alpha)$ of covers with monodromy in $G$ is a quotient of the space of topological covers $\Covt_G(S)$ by the translation equivalence relation; see \eqref{eq:top cover to cover}. 
However it is not hard to see that by Remark \ref{rem:primitivity} that this equivalence relation is trivial in the case of $(S_n, \alpha_n)$. (The surface $(S_n,\alpha_n)$ is formed by identifying the boundary edges of the square along an IET. The deck group of the universal cover of $(S_n,\alpha_n)$ acts transitively on lifts of these squares, and the collection of these lifts must be preserved by translation automorphisms because of the singularities in their boundary.) \compat{Detail added above to address referee's comment.}
Thus by recalling \eqref{eq:space of covers} we see
that the space of covers of $(S_n,\alpha_n)$ with monodromy in $G$ can be thought of as
\begin{equation}
\label{eq:cov chamanara}
\Cov_G(S_n,\alpha_n)=\Pi_d \bs \Hom\big(\pi_1(S_n,s_0),G\big),
\end{equation}
where $s_0$ is a basepoint of $S_n$.
If $h$ is a homomorphism from $\pi_1(S_n,s_0)$ to $G$, we use $[h]$ to denote its equivalence class in $\Cov_G(S_n,\alpha_n)$.

Recall that formally, the affine automorphism $\phi^{-1}_n$ does not act on the fundamental group of $S_n$. 
We choose a basepoint $s_0$ in the interior of the square near the southwest corner of the square in Figure \ref{fig:Chamanara}.
To get an action on the fundamental group, we need to select
a curve joining $s_0$ to $\phi_n^{-1}(s_0)$ as described by equation \ref{eq:curve needed}. Because we chose $s_0$ in the interior of the square near the southwest corner, its image $\phi_n^{-1}(s_0)$ will also lie near the southwest corner and in the interior of the square. We specify $\beta$ to be a curve joining $\phi^{-1}(s_0)$ to $s_0$ while not leaving the interior of the square. Then as in \eqref{eq:curve needed} we define the group automorphism
\begin{equation}
\label{eq:phi beta}
\phi^{-1}_\beta:\pi_1(S_n,s_0) \to \pi_1(S_n,s_0); \quad [\gamma] \mapsto [\beta^{-1} \bullet (\phi_n^{-1} \circ \gamma) \bullet \beta].
\end{equation}

We now see that Theorem \ref{thm:skew} is a consequence of Corollary \ref{cor:2}.

\begin{proof}[Proof of Theorem \ref{thm:skew}]
Recall that the $n$-adic odometer is uniquely ergodic, because it can be understood as a minimal rotation of a compact abelian group.
Therefore, the translation flow on the surface $(S_n, \alpha_n)$ is also uniquely ergodic. 
Observe that the Teichm\"uller flow is periodic, because the affine automorphism $\phi_n$ has diagonal derivative. 

We begin with some definitions which are illustrated in Figure \ref{fig:skew}. Since this figure illustrates $n=2$, we will just discuss this case for this paragraph. The dotted line $T_2$ on $S_2$, isometric to $[0,1]$, is a transversal to the translation flow. We selected this transversal to be vertical and to pass through the basepoint $s_0$ of $S_2$. The return map to the transversal is isomorphic to the dyadic odometer. 
We create an infinite set of closed loops $\gamma_i$ on $S_2$ indexed by $\mathbb{N}$. For each $i\in\mathbb{N}$, the loop $\gamma_i$ 
moves from the basepoint to a point $x\in T_2$ within $T_2$, follows the translation flow until it first returns to $T_2$,
and then travels back to the basepoint within $T_2$ (shown in Figure \ref{fig:skew} as three dashed line segments). 
In defining $\gamma_i$, we insist that the dyadic expansion $(x_1,x_2,x_3,\dots)$ of $x$ has the property that $x_j = 1$ for all $j<i$ and $x_i = 0$. Doing this for every $i\in \mathbb{N}$ we create the countable set $\{\gamma_i\}$. Let $\Gamma^+ = \langle\gamma_1,\gamma_2,\dots\rangle$ be the free group generated by the $\gamma_i$ which is a subgroup of $\Gamma = \pi_1(S_2,s_0)$.

\begin{figure}
  \includegraphics[width = 6in]{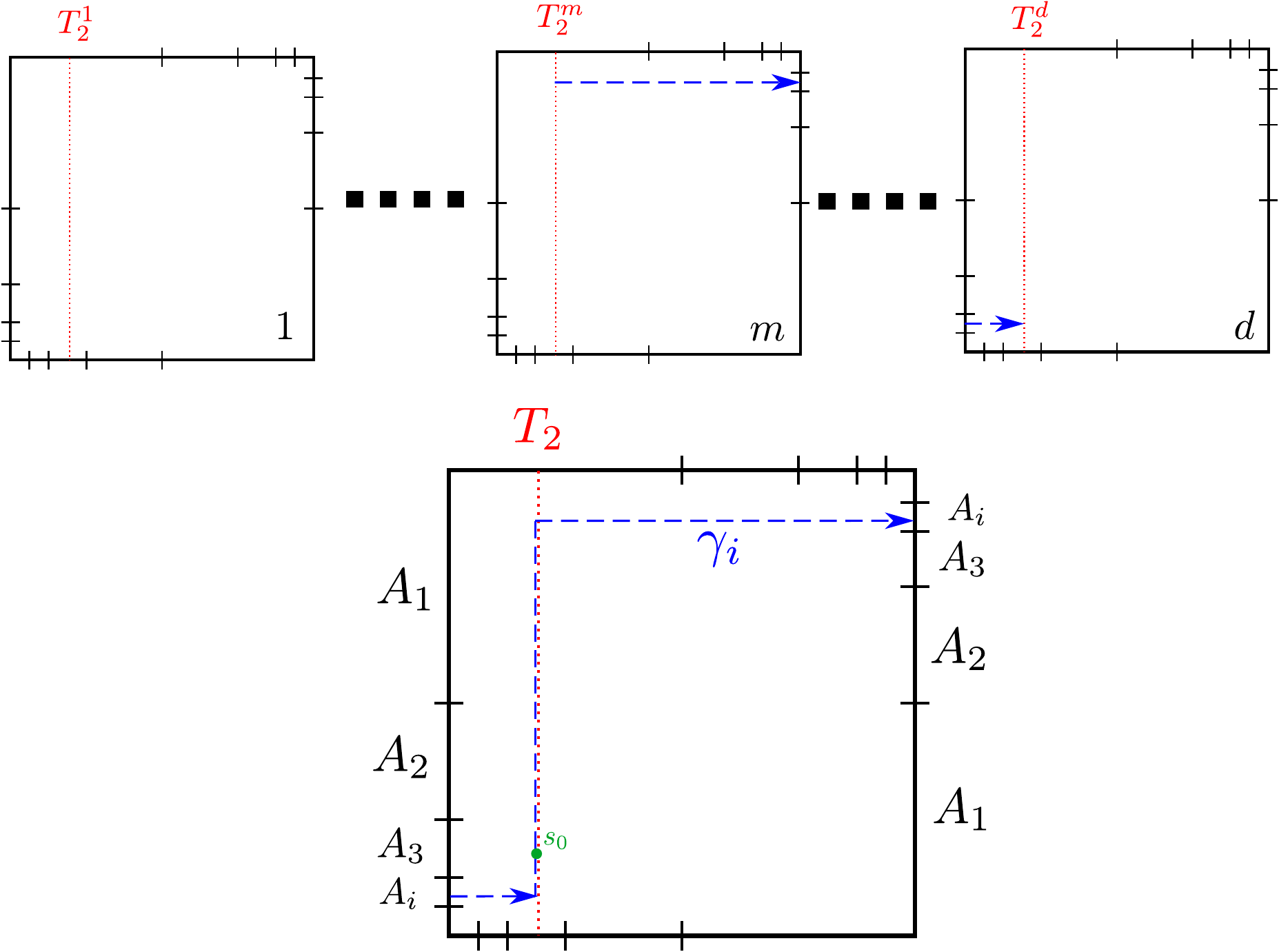}
\caption{{\em Bottom:} The surface $(S_2, \alpha_2)$ with the transversal $T_2$ depicted with a dotted line, and a curve $\gamma_i$ shown as a dashed line. {\em Top:} A $d$-fold cover $(\tilde S_{\psi},\tilde \alpha_{\psi})$ is built out of copies of the square labeled $\{1,\ldots, d\}$ with edges glued according to $\psi_+$. The dotted lines are the lifts of the transversals and the dashed is a lift of the horizontal part of the path of $\gamma_i$ depicted under the assumption that $\psi(\gamma_i)(m)=d$.}
\label{fig:skew}
\end{figure}

The general case is not qualitatively different. Following the same ideas in the previous paragraph, we define $\gamma_i$, $\Gamma^+$
and $\Gamma$ for all integers $n \geq 2$.

Fix a $d \geq 2$ and a subgroup $G \subset \Pi_d$ that acts transitively on $\{1,\ldots, d\}$. 
Corollaries \ref{cor:lifting} and \ref{cor:2} imply that $m_G$-almost every cover has a uniquely ergodic translation flow,
where $m_G$ is the probability measure on $\Cov_G(S_n,\alpha_n)$  from \S \ref{sect:measures}. 
Recall that $m_G$ is the measure induced by the quotient \eqref{eq:cov chamanara} from the product measure $\mu$ on 
$\Hom(\Gamma,G)$; see Definition \ref{def:random cover} of in \S \ref{sect:measures}.
Thus, for $\mu$-a.e. $h \in \Hom(\Gamma,G)$, the corresponding cover
$(\tilde S_h, \tilde \alpha_h)$ has uniquely ergodic translation flow.

Consider the automorphism 
$\phi^{-1}_\beta$ of $\pi_1(S_n, \alpha_n)$ given in \eqref{eq:phi beta}. We observe by inspecting the action
of $\phi^{-1}$ that
$$\phi^{-1}_\beta(\gamma_i)=\gamma_1^{n-1} \gamma_{i+1} \quad \text{for all $i \geq 1$},$$
where calculations are done in the fundamental group.
This implies that $\phi^{-1}_\beta(\Gamma^+) \subset \Gamma^+$.
The inclusion $i:\Gamma^+\rightarrow\Gamma$ then induces a surjective map $i^*: \Hom(\Gamma,G)\rightarrow \Hom(\Gamma^+,G)$ which commutes with the actions of $\phi^{-1}_\beta$. 
The set $\Hom(\Gamma^+,G)$ is a Cantor set with a product measure $\mu_{+}$ as defined in Definition \ref{def:random cover}. It is straight forward to see from the definitions of these measures that 
$\mu_{+} = (i^*)_*\mu$.

Let $i^*(\psi) = \psi_+\in\Hom(\Gamma^+,G)$ for some $\psi \in\Hom(\Gamma,G)$ and let 
$(\tilde S_{\psi},\tilde \alpha_{\psi})$ be the $d$-cover of $(S_n,\alpha_n)$ determined by $\psi$.
The transversal $T_2$ lifts to $\tilde{T}_2$, a transversal on $\tilde S_{\psi}$
for the translation flow. 
It is the union of the $d$ copies $T_2^1,\dots, T_2^d$ of $T_2$, which are also illustrated using dotted lines in the figure. As such, the translation flow on a cover $(\tilde S_{\psi}, \tilde \alpha_{\psi})$ is canonically a suspension flow of the skew product over the $n$-adic odometer. 
One can observe that this skew product is precisely $E_{\psi_+}$ given in \eqref{eqn:skew} of the introduction.
Whenever the translation flow on $(\tilde S_{\psi}, \tilde \alpha_{\psi})$ is uniquely ergodic,
the skew product $E_{\psi_+}$ must be as well. Since this holds for $\mu$ a.e. $\psi$ and $\mu_{+} = (i^*)_*\mu$,
we see that for $\mu_+$ a.e. $\psi_+$, $E_{\psi_+}$ is uniquely ergodic.
\end{proof}

\subsection*{Pathological covers}
For the remainder of the section, we will concentrate on the simplest of Chamanara's surfaces,
$(S_2,\alpha_2)$. \compat{Note that previously this section called $(S_2,\alpha_2)$ simply $(S,\alpha)$. The referee objected so hopefully I have eliminated all $(S,\alpha)$ from this section.}
We will simplify the notation for the hyperbolic automorphism
$\phi_2$ by denoting it by $\phi$.

Our goal with the remainder of this section is to investigate what happens when we have a connected
cover $(\tilde S_2, \tilde \alpha_2)$ which when iterated by application of $\phi$
only accumulates on disconnected covers. The disconnected covers in $\Cov_G(S_n,\alpha_n)$ are given by 
$$\bigcup_{H \in \sH} \Pi_d \bs \Hom(\Gamma,H),$$
where $\Gamma=\pi_1(S_2,s_0)$ and $\sH$ denotes the collection of all subgroups of $G$ which fail to act transitively on $\{1,\ldots, d\}$. 

It can be observed that the connectivity of $(\tilde S, \tilde \alpha)$
does not guarantee the ergodicity of the (horizontal) translation flow, since the connectivity can be arranged with only the gluings of horizontal edges when building the cover as $d$ copies of the unit square with edge identifications. It is not surprising then that such connected devious covers are dense
inside of the space of covers:

\begin{theorem}[Non-ergodic covers]
\label{thm:Reza non-ergodic}
Let $h \in \Hom(\Gamma,G)$. The translation flow on the cover $(\tilde S_h,\tilde \alpha_h)$ 
of $(S_2,\alpha_2)$ associated to $[h]$ is non-ergodic whenever there is an $H \in \sH$ so that every accumulation point of $h \circ \phi_\beta^{-n}$ (as $n \to +\infty$) lies in $\Hom(\Gamma,H)$. For any $H \in \sH$, there exists a
 collection of connected covers dense in $\Cov_G(S_2,\alpha_2)$
 so that every accumulation point lies in $\Pi_d \bs \Hom(\Gamma,H)$
 but no accumulation point lies in $\Pi_d \bs \Hom(\Gamma,H')$ for any proper subgroup $H' \subset H$.
\end{theorem}

It is interesting to consider whether there are connected covers which accumulate only on disconnected covers under the Teichm\"uller deformation but whose translation flow is nonetheless uniquely ergodic. 
We show such covers exist.
This is analogous to sufficiently slow divergence of the  Teichm\"uller deformation giving rise to unique ergodicity in the classical setting of closed translation surfaces as in \cite{CE07} \cite{rodrigo:erg}.

\begin{theorem}
\label{thm:Reza slow}
Suppose $G$ is a subgroup of $\Pi_d$, and $H_1, H_2 \subset G$ are subgroups which do not act transitively
on $\{1,\ldots,d\}$, but the group generated by the elements of $H_1 \cup H_2$ does act transitively.
Then, there are finite covers of $(S_2,\alpha_2)$ with monodromy in $G$ 
whose translation flow is uniquely ergodic, but whose orbit under $\phi$ accumulates only on surfaces in the collection of disconnected covers,
$$\Pi_d \bs \big(\Hom(\Gamma,H_1) \cup \Hom(\Gamma,H_2)\big).$$
\end{theorem}

\begin{remark}
It seems likely that there are also non-ergodic covers whose orbits under $\phi$ accumulate as in Theorem \ref{thm:Reza slow}. We do not investigate this question.
\end{remark}

The key to proving these results is an understanding of the action on $\Hom(\Gamma,G)$ given by
$$h \mapsto h \circ \phi^{-1}_\beta,$$
where $\phi^{-1}_\beta: \Gamma \to \Gamma$ is the automorphism from \eqref{eq:phi beta}.

In order to describe this action, we select a generating set for $\Gamma$. For each integer
$n$, we will let $\gamma_n \in \Gamma$ be a homotopy class of curves which start and end at the basepoint.
If $n \leq 0$, we define $\gamma_n$ to contain the curves which move downward from the basepoint passing through the horizontal edge labeled $n$ and returning to the basepoint without passing through any other labeled edges. We similarly define $\gamma_n$ for $n>0$ to contain the curves which move rightward over the vertical edge labeled $n$. 
(This is compatible with the definition of $\gamma_n$ for $n \geq 1$ in the Proof of Theorem \ref{thm:skew},
and depicted in Figure \ref{fig:skew}.)
Observe:
\begin{equation}
\label{eq:generated}
\Gamma=\pi_1(S_2,s_0) \quad \text{is freely generated by} \quad \{\gamma_n~:~n \in \Z\}.
\end{equation}
(The curves $\gamma_n$ can be taken to be pairwise disjoint except at the basepoint and so the union of these curves is a bouquet of countably many circles. There is a deformation retraction of the surface to this bouquet.)

The dynamics of $\phi_\beta$ action on $\Hom(\Gamma,G)$ turn out to be conjugate to the shift on $G^\Z$:

\begin{lemma}
\label{lem:conjugacy}
The map $\gmap: \Hom\big(\Gamma,G\big) \to G^\Z$ defined by
$$\gmap(h)_m=h \circ \phi_\beta^{-m}(\gamma_1)$$
is a homeomorphism. Let $\sigma:G^\Z \to G^\Z$ be the shift map $\sigma(\g)_i=\g_{i+1}$. Then 
$$\gmap(h \circ \phi_\beta^{-1})=\sigma \circ \gmap(h) \quad \text{for all $h \in \Hom\big(\Gamma,G\big)$.}$$
\end{lemma}

We call $\gmap(h)$ the {\em $G$-sequence} of $h$. As a first step to proving this theorem, we work out the action
of $\phi_\beta^{-1}$ and its inverse $\phi_\beta$ on $\Gamma$:

\begin{proposition}
For each $n \in \N$, we have 
$$\phi^{-1}_\beta(\gamma_n)=\begin{cases}
\gamma_{n+1} \gamma_1^{-1} & \text{if $n < 0$} \\
\gamma_1 & \text{if $n=0$} \\
\gamma_1 \gamma_{n+1} & \text{if $n>0$,}
\end{cases}
\and 
\phi_\beta(\gamma_n)=\begin{cases}
\gamma_{n-1} \gamma_{0} & \text{if $n \leq 0$}\\
\gamma_0 & \text{if $n=1$}\\
\gamma_0^{-1} \gamma_{n-1} & \text{if $n>1$.}
\end{cases}
$$
\end{proposition}
This proposition may be proved by inspecting the action of $\phi^{-1}_\beta$ as defined in 
\eqref{eq:phi beta}. We leave the details to the reader.

The following lemma describes how to recover information about the $\phi_\beta$-orbit of $h$ from its $G$-sequence.

\begin{lemma}
\label{lem:G-sequence}
Let $\g=\gmap(h)$ be the $G$-sequence of a homomorphisms $h:\Gamma \to G$. Then, for each $k, n \in \Z$,
$$h \circ \phi_\beta^{-k}(\gamma_n)=\begin{cases}
\g_{k+n-1}\g_{k+n}^{-1} \g_{k+n+1}^{-1} \ldots \g_{k-1}^{-1} & \text{if $n<0$}\\
\g_{k-1} & \text{if $n=0$}\\
\g_k & \text{if $n=1$} \\
\g_k^{-1} \g_{k+1}^{-1} \ldots \g_{k+n-2}^{-1} \g_{k+n-1} & \text{if $n>1$.}
\end{cases}$$
\end{lemma}
\begin{proof}
Observe that by definition $h \circ \phi_\beta^{-k}(\gamma_1)=\g_k$ for all $k \in \Z$. This case of $n=1$ will serve as a base case for proving the statement holds when $n \geq 1$. Note that the formula given in the case of $n>1$ can be extended to hold for $n=1$ if one allows the (empty) product of inverses $\g_k^{-1} \g_{k+1}^{-1} \ldots \g_{k+n-2}^{-1}$
to be the identity when $n=1$. So, suppose our formula holds for some $n \geq 1$ and all $k$, 
we will show it holds for $n+1$ and all $k$. Using the proposition, we observe that for $n \geq 1$, 
$$h \circ \phi_\beta^{-k-1}(\gamma_n)=h\circ \phi_\beta^{-k}(\gamma_1 \gamma_{n+1})=\g_k \cdot h\circ \phi_\beta^{-k}(\gamma_{n+1}).$$
By our inductive hypothesis applied to the left side, we see that 
$$\g_{k+1}^{-1} \g_{k+2}^{-1} \ldots \g_{k+n-1}^{-1} \g_{k+n}=\g_k \cdot h\circ \phi_\beta^{-k}(\gamma_{n+1})$$
which gives us that $\phi_\beta^{-k}(\gamma_{n+1})=\g_k^{-1} \ldots \g_{k+n-1}^{-1} \g_{k+n}$. This proves
our formula for $n+1$ and and all $k$. So, by induction, the statement holds for $n \geq 1$.

Now we consider the base case of $n=0$. Observe that 
$$h \circ \phi_\beta^{-k}(\gamma_0)=h \circ \phi_\beta^{-k+1} \circ \phi^{-1}_\beta(\gamma_0)=h \circ \phi_\beta^{-k+1}(\gamma_1)=\g_{k-1}.$$
We will proceed by induction to cover cases with $n<0$. The case of $n=0$ can serve as our base case since
when $n=0$ we have $\g_{k+n-1}\g_{k+n}^{-1} \g_{k+n+1}^{-1} \ldots \g_{k-1}^{-1}=\g_{k-1}$ by convention as above since the product of negations
$\g_{k+n}^{-1} \g_{k+n+1}^{-1}\ldots \g_{k-1}^{-1}$ is an empty product (as $k+n>k-1$ here).
Now suppose the formula holds for some $n\leq 0$ and all $k$. Then, by the proposition and the base case
$$h \circ \phi_\beta^{-k+1}(\gamma_n)=h \circ \phi_\beta^{-k}(\gamma_{n-1}\gamma_0)=h \circ \phi_\beta^{-k}(\gamma_{n-1}) \cdot \g_{k-1}.$$
By inductive hypothesis, we see
$$h \circ \phi_\beta^{-k}(\gamma_{n-1})=h \circ \phi_\beta^{-k+1}(\gamma_n) \cdot \g_{k-1}^{-1}=
(\g_{k+n-2}\g_{k+n-1}^{-1} \ldots \g_{k}^{-1}) \g_{k-1}^{-1}.$$
This completes the inductive step, proving the statement for all $n \leq 0$. 
\end{proof}

This Lemma allow us to prove that $\gmap$ is a topological conjugacy:
\begin{proof}[Proof of Lemma \ref{lem:conjugacy}]
The map $\gmap$ is clearly continuous. The identity provided in Lemma \ref{lem:G-sequence} considered in the special case of $k=0$ allows us to evaluate $h$ on the generators $\{\gamma_n\}$ of $\Gamma$ in terms of $\gmap(h)$. This therefore gives an inverse $\gmap^{-1}$, and we observe that it is continuous. The conjugacy equation is easily verified:
$$\gmap(h \circ \phi_\beta^{-1})_m=h \circ \phi_\beta^{-1} \circ \phi_\beta^{-m}(\gamma_1)=h \circ \phi_\beta^{-m-1}(\gamma_1)=\gmap(h)_{m+1}=\big(\sigma \circ \gmap(h)\big)_m$$
for all $m \in \Z$.
\end{proof}

Now we prove Theorem \ref{thm:Reza non-ergodic} on the non-ergodicity of covers.
\begin{proof}[Proof of Theorem \ref{thm:Reza non-ergodic}]
Let $h:\Gamma \to G$ be a homomorphism, and suppose that all accumulation points of
$h \circ \phi_\beta^{-k}$ as $k \to \infty$ lie in $\Hom(\Gamma,H)$ for some subgroup $H \subset G$,
where $H$ does not act transitively on $\{1,\ldots, d\}$. 
We will show that the translation flow on $(\tilde S_h, \tilde \alpha_h)$ is not ergodic. By compactness of $\Hom(\Gamma,H)$, there is a $K$ so that
for each $k \geq K$, $h \circ \phi_\beta^{-k}(\gamma_1) \in H$. In terms of the $G$-sequence $\g=\gmap(h)$, we see that $\g_k \in H$ for $k \geq K$. It can be observed from the formula in Lemma \ref{lem:G-sequence}
that it follows that $h \circ \phi_\beta^{-K}(\gamma_n) \in H$ for all $n \geq 1$. Observe that the horizontal straight-line flow on the surface $(S_2,\alpha_2)$ only crosses
the intervals with positive label in Figure \ref{fig:Chamanara}. Now consider the cover associated to 
$h \circ \phi_\beta^{-K}$. The cover can be built from copies of the square indexed by $\{1,\ldots, d\}$ with edges identified according to 
$h \circ \phi_\beta^{-K}$. In particular, the intervals with positive label are glued only according to elements of $H$. Thus, points in copy $i \in \{1,\dots, d\}$ only can reach the copies of the square indexed by elements
of the orbit $H(i)$, and by assumption $H(i) \neq \{1,\ldots, d\}$. In particular the union of the squares
indexed by $H(i)$ gives an invariant set with measure strictly between zero and full measure. Note that $\phi_\beta^K$ induces an affine homeomorphism with diagonal derivative from the cover associated to $h$ to the cover associated to $h \circ \phi_\beta^{-K}$, so pulling back this invariant set gives an invariant subset of the straight line flow for the cover associated to $h$ with intermediate measure as desired.

We will now construct a dense set of devious covers as described in second sentence of the Theorem.
Consider the set $X_H \subset G^\Z$ consisting of those $\g \in G^\Z$ which satisfy the following statements:
\begin{enumerate}
\item The subgroup of $G$ generated by $\{\g_k:~k \in \Z\}$ acts transitively on $\{1,\ldots, d\}$.
\item There is a $K \in \Z$ so that $\g_k \in H$ for $k>K$.
\item There is a $B>0$ so that for any $k>K$, the subgroup $H \subset G$ is generated by 
$\{\g_k, \g_{k+1}, \ldots, \g_{k+B-1}\}.$
\end{enumerate}
It should be clear that $X_H$ is dense in $G^\Z$. Since $\gmap$ is a homeomorphism by Lemma \ref{lem:conjugacy}, we know $\gmap^{-1}(X_H)$ is dense in $\Hom(\Gamma,G)$. Now let $h \in \gmap^{-1}(X_H)$ and consider the cover $(\tilde S_h, \tilde \alpha_h)$ of $(S_2, \alpha_2)$. By Lemma \ref{lem:G-sequence},
the group generated by $\{\g_k\}$ coincides with the image $h(\Gamma)$ and so statement (1) implies that $\tilde S_h$ is connected. Now let $h'$ be any accumulation point of $h \circ \phi_\beta^{-n}$ taken as $n \to +\infty$.
Let $\g'=\gmap(h')$. Then by Lemma \ref{lem:conjugacy}, $\g'$ is an accumulation point of $\sigma^n(\g)$. 
Statement (2) then implies that $\g' \in H^\Z$ and so $h' \in \Hom(\Gamma,H)$ by Lemma \ref{lem:G-sequence}.
Statement (3) implies that $\{\g'_1, \g'_2, \ldots, \g'_B\}$ generates $H$ and Lemma \ref{lem:G-sequence}
then guarantees that $h'(\Gamma)=H$ so in particular $h' \not \in \Hom(\Gamma,H')$ for any proper subgroup $H' \subset H$.
\end{proof}

Now we will move toward proving Theorem \ref{thm:Reza non-ergodic} about the existence of covers with ergodic translation flow whose $\phi_\beta$-orbits accumulate only on disconnected covers. Recall that this theorem dealt with the situation when we have two subgroups $H_1$ and $H_2$ of $G$ which do not act transitively on $\{1,\ldots, d\}$ but which together act transitively.

We now introduce a combinatorial tool that will allow us
to find examples of $\g=\gmap(h)$ where the sequence of covers associated to $h \circ \phi_\beta^{-k}$ only has disconnected limits.
Let $a \in \Z$ and $A=\{n \in \Z~:~n \geq a\}.$ Let $p:A \to \{0, 1, 2\}$ be an arbitrary function
with the property that $p(n)+p(n+1) \neq 3$ for each $n \in A$ (or equivalently $p(n)=1$ implies $p(n+1) \neq 2$ and
$p(n)=2$ implies $p(n+1)\neq 1$.) For such a map $p$ and an $n \in A$, we define the {\em preimage
interval} containing $n$, $I(p,n) \subset \Z$, to be the maximal collection of consecutive elements of $A$ so that $n \in I(p,n)$ and $p$ is constant on $I(p,n)$. We'll say that $\g \in G^\Z$ is
{\em $p$-ready} \label{def:p-ready} if the following two statements hold for each $n \in A$:
\begin{itemize}
\item If $p(n)=0$, then $\g_n \in H_1 \cap H_2$.
\item If $p(n) \in \{1,2\}$, then $\g_n \in H_{p(n)}$ and $\{\g_k~:~k \in I(p,n)\}$ generates $H_{p(n)}$. 
\end{itemize}
Let $|I(p,n)|$ denote the number of integers in $I(p,n)$. 

\begin{proposition}[Criterion for disconnected accumulation points]
\label{prop:disconnected}
Let $h \in \Hom(\Gamma,G)$ and suppose $\g=\gmap(h)$ is $p$-ready. Suppose that 
\begin{equation}
\label{eq:growth of zeros}
\lim_{n \to \infty; ~p(n)=0} |I(p,n)|=\infty,
\end{equation}
(with the limit taken only over those $n$ so that $p(n)=0$).
Then every accumulation point of $h \circ \phi_\beta^{-k}$ as $k \to \infty$ corresponds to a disconnected cover. Moreover, these accumulation points all lie in 
$\Hom(\Gamma,H_1) \cup \Hom(\Gamma,H_2).$
\end{proposition}
\begin{proof}
Let $h_\infty$ be an accumulation point of $h \circ \phi_\beta^{-k}$ as $k \to \infty$. 
Then there is an increasing sequence of integers $k_j$ so that $h_\infty=\lim_{j \to \infty} h \circ \phi_\beta^{-k_j}$. Let $\g^\infty=\gmap(h_\infty)$. By Lemma \ref{lem:conjugacy}, we have
$$\g^\infty = \lim_{j \to \infty} \sigma^{k_j}(\g).$$

We claim there is an $i \in \{1,2\}$ so that
$\g^\infty_m \in H_i$ for all $m \in \Z$. If not, there are two indices $n_1$ and $n_2$ so that $\g^\infty_{n_1} \in H_1 \smallsetminus H_2$ and $\g^\infty_{n_2} \in H_2 \smallsetminus H_1$. Then for $j$ large enough,
$\sigma^{k_j}(\g)_{n_1}=\g_{k_j+n_1}$ and $\sigma^{k_j}(\g)_{n_2}=\g_{k_j+n_2}$ have the same property. 
From the definition of $p$-ready we have $p(k_j+n_1)=1$ and $p(k_j+n_2)=2$ so between $k_j+n_1$ and $k_j+n_2$ there is a $m_j \in \Z$ so that $p(m_j)=0$ and we must have $|I(p,m_j)| < |n_1-n_2|$. This violates \eqref{eq:growth of zeros}.

We have shown $\g^\infty \in H_i^\Z$. Lemma \ref{lem:G-sequence} implies that
$h_\infty \in \Hom(\Gamma, H_i)$.
\end{proof}

Our proof of Theorem \ref{thm:Reza slow} will combine the above with the following lemma, which will allow us to apply the ergodicity
criterion given in Theorem \ref{thm:integrability}.

\begin{lemma}
\label{lem:contribution}
Let $G \subset \Pi_d$ be a subgroup which acts transitively on $\{1,\ldots, d\}$
for some $d \geq 2$.
For each finite subset $J \subset \Z$ and for each $\eta>0$, there is a constant $c=c(\eta, J)>0$ so that 
for each $h \in \Hom(\Gamma,G)$ that satisfies the condition that the subgroup generated by 
$\{h(\gamma_j)~:~j \in J\}$ acts transitively on $\{1, \ldots, d\}$, there is a connected subsurface $\tilde U_h$ of the cover $(\tilde S_h, \tilde \alpha_h)$ associated to $h$, so that $\text{Area}(\tilde S_h \smallsetminus \tilde U_h)<\eta \text{Area}(\tilde S_h)$ and
$$\epsilon(t)^4 \sD_t^{-2} > c \quad \text{for each
$t \in \R$ with $|t| \leq \frac{\ln(2)}{2}$,}$$
where $\epsilon(t)$ represents a lower bound for the distance from points in $\tilde U_t$ to points in $\Sigma$
and $\sD_t$ is an upper bound for the diameter of $\tilde U_h$ both measured using the pullback metric $\dist_t$
as in Theorem \ref{thm:integrability}.
\end{lemma}

We note that the expression $\epsilon(t)^4 \sD_t^{-2}$ is the function inside the integral in Theorem \ref{thm:integrability} in the special case of considering a single connected subsurface (i.e., $C_t \equiv 1$).

\begin{proof}
We will explicitly describe how to build $\tilde U_h \subset (\tilde S_h, \tilde \alpha_h)$. We will define
$U$ to be a subsurface of $(S_2,\alpha_2)$, which only depends on $J$ and $\eta$. Recall that $(S_2,\alpha_2)$ was built from a single unit area square with edge identifications. To build $U$, start with a smaller concentric square with area equal to $1-\eta$. Then for each $j \in J$, consider the edge of the surface labeled $j$. Attach to the concentric square a ``handle'' passing through the edge, which stays a bounded distance away from $\Sigma$ (the points added in the metric completion). See Figure \ref{fig:reza u} for an example. Then, we define $\tilde U_h$ to be the preimage of $U$
under the covering map $\tilde S_h \to S$. We note that $\tilde U_h$ is connected because of our condition that $\{h(\gamma_j)~:~j \in J\}$ acts transitively.

\begin{figure}
\includegraphics[scale=0.75]{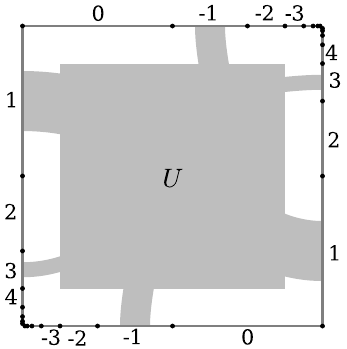}
\caption{A subsurface $U$ when $J=\{-1,1,3\}$.}
\label{fig:reza u}
\end{figure}

The $\dist_t$ distances $\epsilon(t)$ and ${\mathcal D}_t$ are the same as the corresponding distances measured
using $g^t(\tilde U)$ in $g^t(\tilde S_h, \tilde \alpha_h)$. Note that the minimal distance from the boundary of $g^t(\tilde U_h)$ to the metric completion in $g^t(\tilde S_h, \tilde \alpha_h)$ is the same as the minimal distance from $g^t(U)$ to the metric completion in $g^t(S_2,\alpha_2)$. This
distance is always positive and varies continuously in $t$, so we can get a uniform lower bound
which holds when $|t| \leq \frac{1}{2} \ln 2$. Similarly,
the diameter of $g^t(\tilde U_h)$ varies continuously in $t$, so we can get an upper bound on the diameter which is uniform in $t$ with $|t| \leq \frac{1}{2} \ln 2$. Finally observe that up to translation equivalence there are only finitely many $\tilde U_h$ as we vary $h$. (The geometry of $\tilde U_h$ only depends on the restriction of $h$ to $J$.)
Thus, we can get an upper bound on the diameter of which is uniform in both $h$ satisfying our condition
and $t$ satisfying $|t| \leq \frac{1}{2} \ln 2$.
\end{proof}

We need a mechanism to ensure that $h$
satisfies the condition of the lemma (i.e., for some fixed $J \subset \Z$ the subgroup generated by $\{ h(\gamma_j)~:~j \in J\}$ acts transitively on $\{1,\ldots, d\}$),
given conditions on the $G$-sequence $\gmap(h)$. This mechanism is a corollary of Lemma \ref{lem:G-sequence}.

\begin{corollary}
\label{cor:conversion}
Suppose $h$ has $G$-sequence $\g=\gmap(h)$. Let $m, k \in \Z$ with $m>0$. Then, if 
the subgroup generated by $\{\g_j~:~k-m\leq j \leq k+m\}$ acts transitively on $\{1,\ldots,d\}$, then so does
the subgroup generated by $\{h \circ \phi_\beta^{-k} (\gamma_j)~:~-m+1\leq j \leq m+1\}$.
\end{corollary}
\begin{proof}
Using Lemma \ref{lem:G-sequence}, we can find an expression for each $\g_j$ with $k-m\leq j \leq k+m$ 
in terms of a $\{h \circ \phi_\beta^{-k} (\gamma_j)~:~-m+1\leq j \leq m+1\}$.
\end{proof}

\begin{proof}[Proof of Theorem \ref{thm:Reza slow}]
We will exhibit an ergodic cover $(\tilde S_h, \tilde \alpha_h)$ by verifying the conditions of Theorem \ref{thm:integrability}. As stated, this theorem involves verifying that for each $\eta>0$, a certain integral is infinite. However, it really only involves proving it for arbitrarily small $\eta$, since if the statement is true for $\eta>0$ then it is also true for any $\eta'>\eta$ with the same subsurfaces and geometric functions chosen.
Hence, it suffices to verify that this integral is infinite for each of $\eta_i>0$ in a sequence tending to zero.

The main tool to get infinite integrals is Lemma \ref{lem:contribution}. We will only apply this Lemma to sets $J$ of the form
$$J_m=\{n \in \Z~:~|n| \leq m \}.$$
Let 
$\Hom_m \subset \Hom(\Gamma,G)$ be the collection of all $h$ so that
the subgroup generated by $\{h(\gamma_n):~n \in J_m\}$ acts transitively on $\{1,\ldots, d\}$.
Then the conditions needed for applying Lemma \ref{lem:contribution} with $J=J_m$ can be succinctly stated as $h \in \Hom_m$. Let $c_{i,m}=c(\eta_i, J_m)$ be the constants given by Lemma \ref{lem:contribution}.
Observe that $D(\phi)=g^{\ln 2}$. This means that
$$g^{k \ln 2}(\tilde S_h, \tilde \alpha_h)=(\tilde S_{h \circ \phi^{-k}}, \tilde \alpha_{h \circ \phi^{-k}}).$$
So, assuming that $h \circ \phi^{-k} \in \Hom_m$, the contribution to the integral in Theorem \ref{thm:integrability}
taken with $\eta=\eta_i$ and $t \in \big[(k-\half) \ln 2, (k+\half) \ln 2\big]$ is at least $c_{i,m} \ln 2$. 

We need the total contribution to the integrals for each $\eta_i$ to be infinite. 
For each $i$, choose a sequence $N_{i,m}$ so that
$$\sum_m^\infty N_{i,m} c_{i,m} \ln 2=+\infty.$$
(The number $N_{i,m}$ representing a number of times we contribute $c_{i,m} \ln 2$ to the integral using the method of the previous paragraph.) Now define the sequence
\begin{equation}
\label{eq:integral as sum}
N_m = \max \{ N_{1,m}, N_{2,m}, \ldots, N_{m,m}\}
\quad \text{and observe}\quad 
\sum_m N_{m} c_{i,m} \ln 2=+\infty \quad \text{for all $i$.}
\end{equation}
Observe that we have the following criterion for ergodicity:

{\bf Claim.} Fix $h$. If there is a sequence of pairwise disjoint subsets $K_m \subset \N$ defined for $m \geq m_0$
for some $m_0>0$ so that $|K_m| \geq N_m$ and 
$h \circ \phi^{-k} \in \Hom_m$ for each $k \in K_m$, then the translation flow on $(\tilde S_h, \tilde \alpha_h)$ is ergodic.

The proof of the claim is that by using Lemma \ref{lem:contribution} with $J=J_m$ on $h \circ \phi^{-k}$ at times
$t \in \big[(k-\frac{1}{2}) \ln 2, (k+\frac{1}{2}) \ln 2\big]$ where $k \in K_m$ we get a contribution of $c_{i,m} \ln 2$ to the integral associated to $\eta_i$. These contributions occur at least $N_m$ times when $m \geq m_0$,
so each of the integrals is infinite by \eqref{eq:integral as sum}.

Now we will explain how to meet the hypotheses of this claim. First we prefer to work with the $G$-sequence
$\g=\gmap(h)$ than with $h$ directly. By Corollary \ref{cor:conversion}, we can guarantee that $h \circ \phi^{-k} \in \Hom_{m}$ if $\sigma^k(\g) \in \mathit{Gen(m)}$ where 
$\mathit{Gen(m)}$ denotes the set of $\g' \in G^\Z$ so that the subgroup 
generated by $\{\g'_j:~|j|<m\}$ acts transitively on $\{1,\ldots, d\}$. Hence we can replace
the condition $h \circ \phi^{-k} \in \Hom_m$ in our Claim with the condition that $\sigma^k(\g) \in \mathit{Gen(m)}$.

We will now produce a $\g$ together with subsets $K_m$ as above so that $\sigma^k(\g) \in \mathit{Gen(m)}$ when $k \in K_m$. To do this we only need to specify a tail for $\g$; i.e., the values for $\g_k$ for $k$ large.
Say that a {\em word} is an element $w \in G^l$ for some $l \in \N$. We use $|w|$ to denote the length of $w$.
We will express this tail for $\g$ as an infinite concatenation of words. For $i \in \{1,2\}$ choose words
$w_i$ so that the group elements appearing in the words generate the subgroup $H_i$ as in the statement of the Theorem. Let $e$ denote the identity element of $G$, and let $e^j$ denote word formed by repeating the identity $j$ times. Observe that if $\ell+|w_1|+|w_2|=2m-1$ and the word 
\begin{equation}
\label{eq:word form}
\g'_{-m+1} \g'_{-m+2} \ldots \g'_{m-1} \quad \text{equals} \quad w_1 e^\ell w_2 \quad \text{or} \quad w_2 e^\ell w_1
\end{equation}
then $\g' \in \mathit{Gen(m)}$ since by hypothesis the group generated by $H_1 \cup H_2$ acts transitively on $\{1,\ldots, d\}$. Let $m_0$ be such that $2 m_0-1 \geq 1+|w_1|+|w_2|$. Let $(m_j)$ be any sequence in the set $\{m \in \N:~m \geq m_0\}$ so that each $m \geq m_0$ appears $N_m$ times in the sequence. 
For each $j$ define $\ell_j$ so that $\ell_j+|w_1|+|w_2|=2m_j-1$.
Now assume $\g$ has a tail of the form 
\begin{equation}
\label{eq:tail}
w_1 e^{\ell_1} w_2 e^{\ell_2} w_1 e^{\ell_3} w_2 e^{\ell_4} w_1 \ldots.
\end{equation}
Then for each $j$ there is a $k_j$ so that when $\g'=\sigma^{k_j}(\g)$ the word of \eqref{eq:word form}
taken with $m=m_j$ has the form of either $w_1 e^{\ell_j} w_2$ or $w_2 e^{\ell_j} w_1$ depending on the parity of $j$. In particular, $\sigma^{k_j}(\g) \in \mathit{Gen(m_j)}$ for each $j$. By construction, $m$ appears $N_m$ times in the sequence $(m_j)$, so by application of the Claim we see that if $\gmap(h)=\g$ then 
the translation flow on $(\tilde S_h, \tilde \alpha_h)$ is ergodic.

Recall that we can improve ergodicity to unique ergodicity using Corollary \ref{cor:lifting} since the $2$-adic odometer is uniquely ergodic.

It remains to explain that the orbit of $(\tilde S_h, \tilde \alpha_h)$ under $\phi$ only accumulates on disconnected covers as stated in the Theorem. Observe that the $\g$ we constructed is $p$-ready, where $p:\{n \in \Z:~n \geq a\} \to \{0,1,2\}$ is defined so that $a$ represents the position of the first letter in the tail \eqref{eq:tail},
and $p(a+n)$ is determined by the $n+1$-st position in this tail: If the $n+1$-st position belongs to an $e^{\ell_\ast}$ word we define $p(a+n)=0$ and if it belongs to a $w_i$ word we define $p(a+n)=i$. Since each $m$
appears only finitely many times in the sequence $m_j$, we see that the length of the intervals $I(p,n)$ where $p(n)=0$ tends to infinity. Thus as an application of Proposition \ref{prop:disconnected}, $(\tilde S_h, \tilde \alpha_h)$ only accumulates on disconnected covers as desired.
\end{proof}

\subsection{The ladder surface}
\label{sect:ladder surface} 
For this subsection, we let $(S_L,\alpha_L)$ denote the infinite genus translation surface of Figure \ref{fig:ladder_surface}, which we call the {\em ladder surface}. It can be constructed from a region in the plane bounded by countably many horizontal and vertical edges. We have labeled the edges using the set $\Z^\ast=\Z \smallsetminus \{0\}$. Edges with the same label are glued by translation to form the surface, which does not include endpoints of these edges or the limiting point in the top right of the region. Let $\varphi$ denote the golden ratio, $\frac{\sqrt{5}+1}{2}$. An edge with label $j \in \Z^\ast$ has length given by 
$\varphi^{-|j|}.$
This information specifies the geometry of the surface. We have also selected a basepoint $s_0$ for our surface in the figure.

The surface is built using Thurston's construction \cite[\S 6]{T88} from the graph $\sG$ depicted in Figure \ref{fig:ladder_surface}. We will briefly review this construction. The surface $(S_L, \alpha_L)$ has horizontal and vertical cylinder decompositions with cylinders bounded by dotted lines in Figure \ref{fig:ladder_surface}.
From these cylinders one can build a bipartite graph: the white vertices consist of the horizontal cylinders, the gray vertices consist of the vertical cylinders and an edge is drawn between two vertices for each intersection between the corresponding two cylinders. 
The vertices of $\sG$ have been labeled in the figure to match the label of the edge passed through by the corresponding cylinder.
(Two adjacent cylinders pass through edges labeled $-1$ and $1$.)
For such a connected graph, for any $A>0$ there is at most one way to assign widths to the cylinders so that the area of the resulting surface is $A$ and the moduli (ratio of widths to circumferences) of all these cylinders are equal: this assignment viewed as a real valued function defined on the vertices of the graph must be an eigenfunction for the adjacency operator. (This eigenfunction always exists for finite graphs by the Perron-Frobenius theorem, while uniqueness is guaranteed in the infinite case by an analogous result \cite[Theorem 2.2]{V68}.) 
This forces the lengths of edges of the region defining our surface to be as mentioned above (up to uniform scaling). 

\begin{figure}[t]
\includegraphics[width=3in]{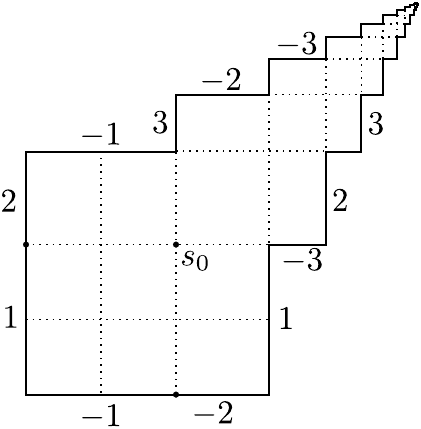}
\includegraphics[width=3in]{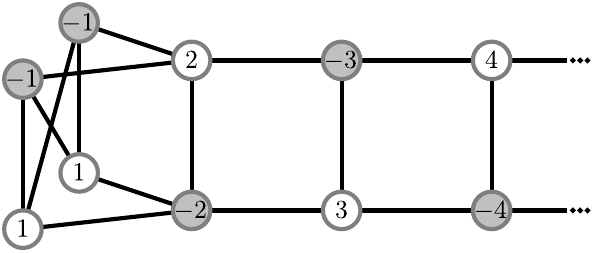}
\caption{{\em Left:} The ladder surface $(S_L,\alpha_L)$. {\em Right:} The graph used in Thurston's construction to build the ladder surface.}
\label{fig:ladder_surface}
\end{figure}

Once the moduli of the cylinders are all equal, an observation of Thurston
guarantees existence of two affine automorphisms of $(S_L, \alpha_L)$ with parabolic derivative and horizontal and vertical eigenvectors respectively preserving the horizontal and vertical cylinders in the decompositions. In the case of our surface, there is an additional reflective symmetry in a line of slope $1$. To generate these three affine automorphisms it is sufficient to use $\psi$ which performs a single left Dehn twist in each of the countably many vertical cylinders, and the reflective symmetry $\rho$. The derivatives are given by
$$D(\psi)=\left[\begin{array}{rr}
1 & 0 \\
2 \varphi & 1
\end{array}\right] \and
D(\rho)=\left[\begin{array}{rr}
0 & 1 \\
1 & 0
\end{array}\right].
$$
The composition $\rho \circ \psi \circ \rho$ performs a right Dehn twist in each of the horizontal cylinders. All these affine automorphisms fix the basepoint $s_0$
depicted in Figure \ref{fig:ladder_surface}.
We define $\Aff'$ to be the subgroup of $\Aff(S_L,\alpha_L)$ generated by $\psi$ and $\rho$. The group $\Aff'$ is isomorphic to the free product $\Z \ast (\Z/2\Z)$ and contains no translation automorphisms.

\begin{remark}
\label{rem:unknown affine group}
It is unknown to the authors if $\Aff'=\Aff(S_L,\alpha_L)$.
\end{remark}

Recall that $\sO(S,\alpha)$ is naturally identified with $\SL_\pm(2,\R)/V(S,\alpha)$. See the discussion below \eqref{eq:sO}.
Let $V'=D(\Aff')$. Since $\Aff'$ contains no translation automorphisms, the derivative map $D:\Aff' \to V'$ is a group isomorphism.
Unfortunately we do not know if $V'=V(S,\alpha)$, so all we get is a covering map:
\begin{equation}
\label{eq:covering O 1}
\SL_\pm(2,\R)/V' \to \sO(S,\alpha).
\end{equation}
Fortunately, if a geodesic recurs in $\SL_\pm(2,\R)/V'$, then its image in $\sO(S,\alpha)$ also recurs.
So non-divergence of $g^t A V'$ implies ergodicity of the translation flow on $A(S,\alpha)$ by Corollary \ref{cor:rodrigo ergodicity from Veech group}.
We will be working throughout this section without knowing precise information about the Veech group and this will create 
difficulties which will be dealt with. It turns out we will be able to get by only knowing the following:

\begin{proposition}
\label{prop:discrete}
The Veech group $V(S_L,\alpha_L)$ is discrete.
\end{proposition}
\begin{proof}
Consider the maximal vertical cylinder $C_{-2}$ in $S_L$ which passes through the edge labeled $-2$ in Figure \ref{fig:ladder_surface}. We chose this because it has the longest circumference of all vertical cylinders.
Suppose $A \in \SL(2,\R)$ is close enough to the identity so that a closed vertical segment of length equal to the circumference of $C_{-2}$ immerses inside of $A(C_{-2})$. Now suppose that $A \in V(S_L,\alpha_L)$. Then there is a translation isomorphism $e:A(S_L,\alpha_L) \to (S_L,\alpha_L)$.
The restriction $e|_{\ell}$ must not given an embedding of $\ell$ because all vertical trajectories on $(S_L,\alpha_L)$ close up at or before this length. Since $e|_{C_{-2}}$ is an embedding we see that $\ell$ must close up on $A(C_{-2})$. Thus $A(C_{-2})$ must be a vertical cylinder. 
Because $C_{-2}$ has singularities in its boundary so must $A(C_{-2})$, so that $A(C_{-2})$ is a maximal cylinder.
Area considerations show that $A(C_{-2})=C_{-2}$. We conclude that $A$ has a vertical eigenvector with eigenvalue $1$. A symmetric argument implies that
if $A$ is taken from a small enough open neighborhood of the identity then
$A$ has a horizontal eigenvector with eigenvalue $1$. By choosing our neighborhood small enough, both arguments hold and we see that the only element of $V(S_L,\alpha_L)$ in this neighborhood is the identity.
\end{proof}

We will need to understand $\SL_\pm(2,\R)/V'$ for later arguments.
The group $\SL_\pm(2,\R)$ can be identified with the flag bundle of the hyperbolic plane (where a flag is a choice of a unit tangent vector
and an orthogonal unit tangent vector taken from the same tangent plane). Here the hyperbolic plane is $O(2) \bs \SL_\pm(2,\R)$.
The group $\SL_\pm(2,\R)$ then acts on the right by isometry. The unit speed geodesics have the form $t \mapsto O(2) g^t A$
for some $A \in \SL_\pm(2,\R)$. 
The boundary of the hyperbolic plane can be identified with $\R \P^1=\R^2/(\R \smallsetminus \{0\})$.
The geodesic $t \mapsto O(2) g^t A$ converges as $t \to +\infty$ to the boundary point representing the vectors which are contracted
to zero as $t \to \infty$, the projective equivalence class of
$$\left(\begin{array}{r} x \\ y \end{array}\right)=A^{-1}\left(\begin{array}{r} 1 \\ 0 \end{array}\right).$$
We identify the hyperbolic plane $O(2) \bs \SL_\pm(2,\R)$ with the upper half plane in such a way so that
the projective class of $(x,y)$ viewed as a boundary point is identified with the slope $y/x \in \R \cup \{\infty\}$.

\begin{figure}
\includegraphics[width=3in]{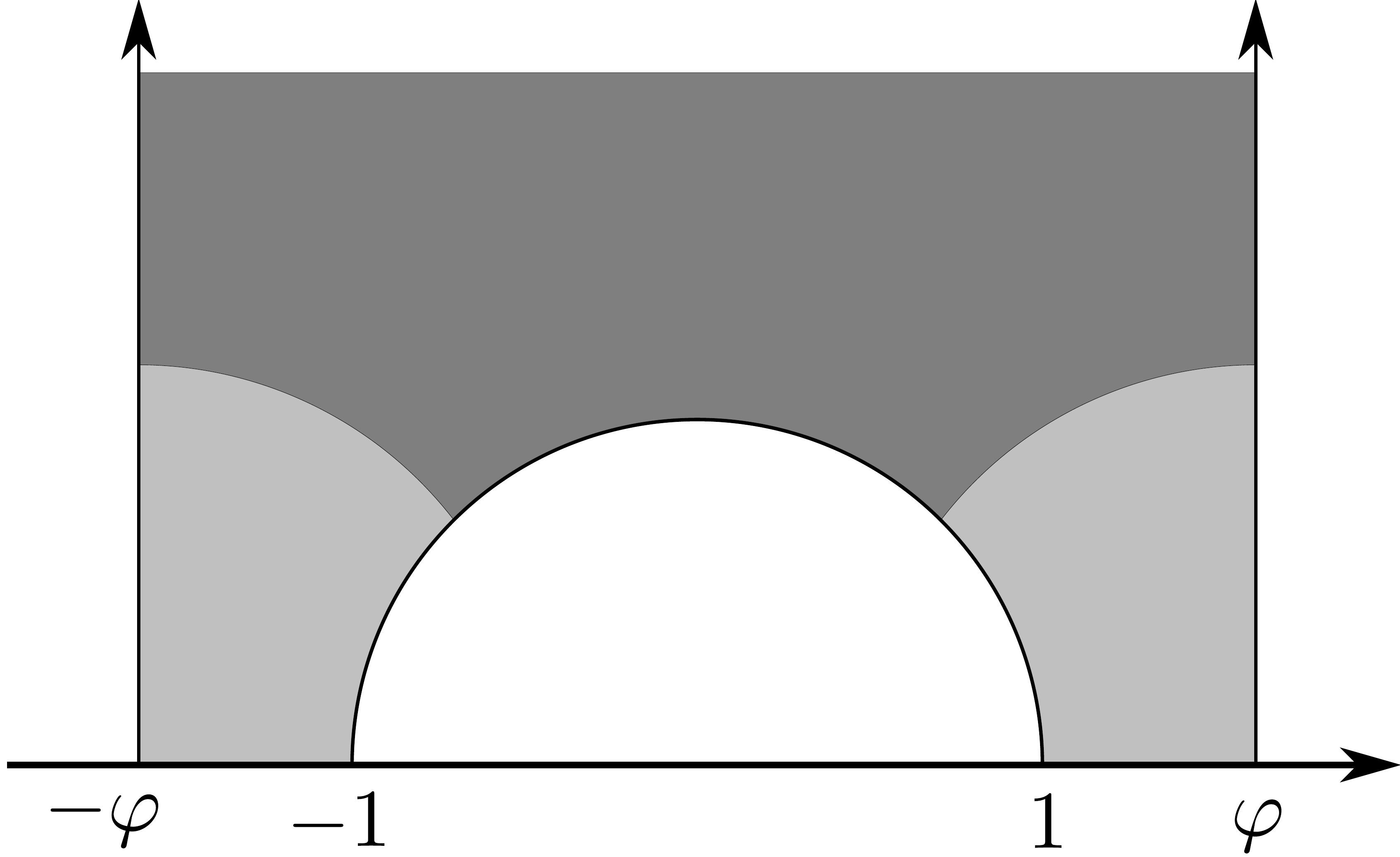}
\caption{A fundamental domain for the action of $V'$ on the upper half plane is shown in gray (it extends vertically to $\infty$).
The geodesic between $-1$ and $1$ is preserved by the reflection $\rho$, and
$\psi$ acts by translation by $2 \varphi$. The convex core is in dark gray and extends upward.}
\label{fig:fundamental domain}
\end{figure}

We let $V'$ denote the group of derivatives of $\Aff'$. 
The group $V'$ acts discretely on the hyperbolic plane, and Figure \ref{fig:fundamental domain} shows a fundamental domain for the action.
A geodesic ray $t \mapsto O(2) g^t A V'$ in $\SL_\pm(2,\R)/V'$ for $t>0$
is non-divergent if and only if the endpoint of the geodesic in the hyperbolic plane lies
in the horospherical limit set, which in this case is the limit set of $V'$ with fixed points of parabolics removed
\cite[Theorem 2]{BM74}. The limit set in this case is a Cantor set of Hausdorff dimension larger than $\frac{1}{2}$
\cite[Remark 4.1]{Hinf}. The {\em convex core} is the convex hull of the limit set in the hyperbolic plane projected into the surface $O(2) \bs \SL_\pm(2,\R)/V'$.
A geodesic ray in this surface is non-divergent if and only if it has an $\omega$-limit point in the convex core.
The convex core of this surface is depicted in Figure \ref{fig:fundamental domain}.

\begin{remark}[Boundary of convex core]
The geodesic in the surface $O(2) \bs \SL_\pm(2,\R) / V'$ with monodromy given by the derivative of the commutator $D([\rho, \psi])$ is the boundary of the convex core of this surface. 
\end{remark}

\begin{theorem}
\label{thm:prior work}
Let $A \in \SL_\pm(2,\R)$, and $(S_L,\alpha_L)$ be the ladder surface. Consider the geodesic ray
$t \mapsto O(2) g^t A V'$ in the surface $O(2) \bs \SL_\pm(2,\R) / V'$. 
\begin{itemize}
\item If the ray has an $\omega$-limit point (which must be in the convex core) then 
the translation flow on $A (S_L,\alpha_L)$ is defined for all time almost everywhere and is ergodic.
\item If the ray has an $\omega$-limit point in the interior of the convex core then 
the translation flow on $A(S_L,\alpha_L)$ is uniquely ergodic.
\end{itemize}
\end{theorem}
\begin{proof}
To see the first statement observe that if the geodesic has an $\omega$-limit point then so does the trajectory 
$t \mapsto g^t A (S_L,\alpha_L)$ in $\sO(S_L,\alpha_L)$ because of the covering map \eqref{eq:covering O 1}. The conclusion is then given by 
Corollary \ref{cor:rodrigo ergodicity from Veech group}.

The second statement is an application of Theorem H.5 of \cite{Hinf}. To use the theorem we need to verify several hypotheses.
First $(S_L,\alpha_L)$ arises from Thurston's construction using a graph with no vertices of valence one
as with the graph here; see Figure \ref{fig:ladder_surface}. Second in the language of that article we need the trajectory
$g^t A$ not be asymptotic to an endpoint of a maximal interval of the compliment of the limit set of the subgroup of $V(S_L,\alpha_L)$ generated by the horizontal and vertical parabolics. The subgroup generated by these parabolics is the same as $V'$ up to finite index,
so the two limit sets are identical. This condition that the trajectory is not asymptotic to an endpoint of a complimentary interval 
is equivalent to $O(2) g^t A V'$ not being asymptotic to the boundary of the convex core in $O(2) \bs \SL_\pm(2,\R) / V'$.
It is impossible for a geodesic ray with an $\omega$-limit point in the interior of the convex core to be asymptotic to the convex core boundary.
\end{proof}

We will be considering double covers of the ladder surface.
For $i \in \Z^\ast$, we let $\gamma_i$ be one of the two homotopy class of loops which start at the basepoint, cross only the edge labeled $i$, and return to the basepoint. These two homotopy classes are inverses in $\Gamma=\pi_1(S,s_0)$. We make the choice of $\gamma_i$ so $\gamma_i$ moves rightward across the vertical edge labeled $i$ if $i>0$,
and moves upward over the horizontal edge labeled $i$ if $i<0$. These curves freely generate the fundamental group,
which we denote by $\Gamma=\pi_1(S,s_0)$. (As with the the Chamanara surface in the previous section, the curves $\gamma_i$ can be chosen so that $S_L$ retracts onto the union of the curves which is homeomorphic to a countable bouquet of circles.)

Let $\Z_2=\Z/2\Z$. Observe that $\Z_2$ acts trivially on $\Hom(\Gamma, \Z_2)$ by conjugation so that
$\Covt_{\Z_2}(S_L)=\Hom(\Gamma, \Z_2)$; see \eqref{eq:space of covers}. It can be seen from Remark \ref{rem:primitivity}
that $\Cov_{\Z_2}(S_L,\alpha_L)=\Covt_{\Z_2}(S_L)$, i.e., \eqref{eq:top cover to cover} is a bijection in this case.
So in summary,
\begin{equation}
\label{eq:covers equals hom}
\Cov_{\Z_2}(S_L,\alpha_L)=\Hom(\Gamma, \Z_2).
\end{equation}

Since our basepoint is fixed by the affine automorphisms in $\Aff'$, we have a canonical action of each
element in $\Aff'$ on the space of covers. For example, if $h \in \Hom(\Gamma, \Z_2)$, then 
$\psi_\ast(h)=h \circ \psi^{-1}$ and $\rho_\ast(h)=h \circ \rho^{-1}$. We will work out this action below.

\begin{notation}
Let $h \in \Hom(\Gamma, \Z_2)$. 
We will write $h(i)$ to abbreviate $h(\gamma_i)$ for $i \in \Z^\ast$. 
(We have a bijection between elements of $\Hom(\Gamma, \Z_2)$ 
and the set of all functions $\Z^\ast \to \Z_2$.)
\end{notation}
It can be observed that the generators of
$\Aff'$ act on $\Hom(\Gamma, \Z_2)$ as follows:
\begin{equation}
\label{eq:psi ladder}
\big(\psi_\ast(h)\big)(i)=\begin{cases}
h(i) & \text{if $i <0$,} \\
h(i)+h({-2}) & \text{if $i=1$,} \\
h(i)+h({-2})+h({-3}) & \text{if $i=2$,} \\
h(i)+h({-i-1})+h({-i})+h({-i+1}) & \text{if $i>2$.}
\end{cases}
\end{equation}
\begin{equation}
\label{eq:rho ladder}
\big(\rho_\ast(h)\big)(i)=h(-i).
\end{equation}

A key observation is the following:
\begin{proposition}
\label{prop:psi squared trivial}
The square $\psi^2_\ast$ acts trivially on 
$\Hom(\Gamma, \Z_2)$.
\end{proposition}
\begin{proof}
The action of $\psi^2_\ast$ preserves the value of $h(i)$ for $i<0$, and adds a sum of values of $h$ evaluated
at negative integers to the values of $h(i)$ for $i>0$.
Since adding the same number twice is the same as adding zero
in $\Z_2$, $\psi_\ast^2$ acts trivially.
\end{proof}

We have the following trivial consequence:

\begin{corollary}
\label{cor:veech group of cover}
The subgroup of $V'$ given by
\begin{equation}
\label{eq:tilde V'}
\tilde V'=\langle R\, D(\psi^2) R^{-1} ~|~R \in V' \rangle
\end{equation}
is a subgroup of the Veech group of any $(\tilde S_h, \tilde \alpha_h) \in \Cov_{\Z_2}(S_L,\alpha_L)$.
\end{corollary}

From this we get a covering map (similar to \eqref{eq:covering O 1}):
\begin{equation}
\label{eq:covering O 2}
\SL_\pm(2,\R)/ \tilde V'  \to \sO(\tilde S_h, \tilde \alpha_h).
\end{equation}
Observe that a non-divergent geodesic in $\SL_\pm(2,\R)/ \tilde V'$ will descend to a non-divergent geodesic in $\sO(\tilde S_h, \tilde \alpha_h).$ We have recovered statement (L1) of the introduction:

\begin{proposition}
\label{prop:non-divergence covers}
Suppose that $g^t A \tilde V'$ is non-divergent in $\SL_\pm(2,\R)/ \tilde V'$. Then 
for any connected $(\tilde S_h, \tilde \alpha_h) \in \Cov_{\Z_2}(S_L,\alpha_L)$,
the double cover $A(\tilde S_h, \tilde \alpha_h)$ of $A(S_L,\alpha_L)$ is not devious
and every $A(\tilde S_h, \tilde \alpha_h)$ has uniquely ergodic translation flow.
\end{proposition}

\begin{proof}
Fix $A$ satisfying the first statement and any cover $(\tilde S_h, \tilde \alpha_h)$.
From the covering map \eqref{eq:covering O 2}, we know $g^t A(\tilde S_h, \tilde \alpha_h)$ is non-divergent
in $\sO(\tilde S_h, \tilde \alpha_h)$. Thus this cover is not devious and ergodicity follows from Corollary \ref{cor:rodrigo ergodicity from Veech group}.
To prove unique ergodicity, it suffices in light of Corollary \ref{cor:lifting} to show that the translation flow on the surface $A(S_L,\alpha_L)$ is uniquely ergodic. This holds unless $O(2) g^t A V'$ is forward asymptotic to the convex core boundary of $O(2)\bs \SL_\pm(2,\R) /V'$ by Theorem \ref{thm:prior work}. But this convex core boundary lifts to a pair of divergent geodesics in $O(2)\bs \SL_\pm(2,\R) /\tilde V'$. (This will be justified more below; the wavy lines between light and dark gray regions in Figure \ref{fig:periodic disk} denote the lifted boundary of the convex core.)
\end{proof}

Our main goal of the section is to find some devious covers, so from the above proposition we will need to consider
elements $A \in \SL_\pm(2,\R)$ so that $g^t A \tilde V'$ is divergent in $\SL_\pm(2,\R)/ \tilde V'$.
We will show that the existence of such covers is related to a combinatorial rate of divergence of this trajectory. To measure this we need to further develop our understanding of $\tilde V'$ and the surface $O(2) \bs \SL_\pm(2,\R)/ \tilde V'$.

From \eqref{eq:tilde V'}, the group $\tilde V'$ is evidently normal in $V'$. Let $\Delta=V' / \tilde V'$ be the quotient. This group has a right action as the Deck group of the covering maps
\begin{equation}
\label{eq:cover}
\SL_\pm(2,\R)/ \tilde V' \to \SL_\pm(2,\R)/ V'
\quad \text{and} \quad
O(2) \bs \SL_\pm(2,\R)/ \tilde V' \to O(2) \bs \SL_\pm(2,\R)/ V'.
\end{equation}
Since as a group $V'=\langle D(\psi), D(\rho)\rangle$ is isomorphic to the free product $\Z \ast \Z_2$, and 
$\psi^2$ is in $\tilde V'$, the quotient $\Delta$ is isomorphic to the infinite dihedral group which we think of as $\Isom(\Z)$, where $\Z$ is given the standard metric. A homomorphism $\delta:V' \to \Isom(\Z)$ which factors through $\Delta$ is given by
\begin{equation}
\label{eq:delta}\delta \circ D(\psi):n \mapsto -n \and \delta \circ D(\rho):n \mapsto 1-n.
\end{equation}
The action of the Veech group on $\Cov_{\Z_2}(S_L, \alpha_L)=\Hom(\Gamma, \Z_2)$ induces an action of $\Isom(\Z)$, where
\begin{equation}
\label{eq:ast action}
\big(\delta \circ D(\psi)\big)_\ast=\psi_\ast \and \big(\delta \circ D(\rho)\big)_\ast=\rho_\ast.
\end{equation}

We will now give a cartoon description of the surface $O(2) \bs \SL_\pm(2,\R)/ \tilde V'$.
We can deform our fundamental domain for the $V'$ action so that it looks like the left side of Figure \ref{fig:periodic disk}. The advantage here is that two copies of this domain can be joined together by a Euclidean rotation by $180^\circ$ about the point $\infty$. This rotation is then representing the order two action on $\SL_\pm(2,\R)/ \tilde V'$ given by the right action of $D(\psi)$. The surface $O(2) \bs \SL_\pm(2,\R)/ \tilde V'$ can be tiled by copies of this fundamental domain as shown on the right side of the figure. We have drawn this in such a way so that the action of the Deck group $\Isom(\Z)$ appears natural. In particular, for each $n \in \Z$, we can join together two copies of the fundamental domain together about the point labeled $\infty_n$. We call this region $F_n$; these regions are depicted in the figure.
The isometry 
$\delta \circ D(\psi):n \mapsto -n$ acts by a rotation by $180^\circ$ about the point $\infty_0$.
The isometry $\delta \circ D(\rho):n \mapsto 1-n$ acts by a Euclidean reflection in the line forming the boundary between $F_0$ and $F_1$. We observe that $\Isom(\Z)$ acts in a natural way on these regions as labeled by $\Z$. Namely, we can think of these regions as subsets of $O(2) \bs \SL_\pm(2,\R) / \tilde V'$ and 
\begin{equation}
\label{eq:action on regions}
F_n R^{-1} = F_{\delta(R)(n)} \quad \text{for all $R \in V'$ and $n \in \Z$}.
\end{equation}
The actual action is somewhat subtle: the translation 
\begin{equation}
\label{eq:translation action}
\tau^m=\delta \circ D\big((\rho \circ \psi)^m\big):n \mapsto n+m
\end{equation} 
acts by translation by $m$ in the figure
when $m$ is even, but when $m$ is odd $\tau^m$ acts as a glide reflection carrying each $F_n$ to $F_{n+m}$.

\begin{figure}
\includegraphics[width=6in]{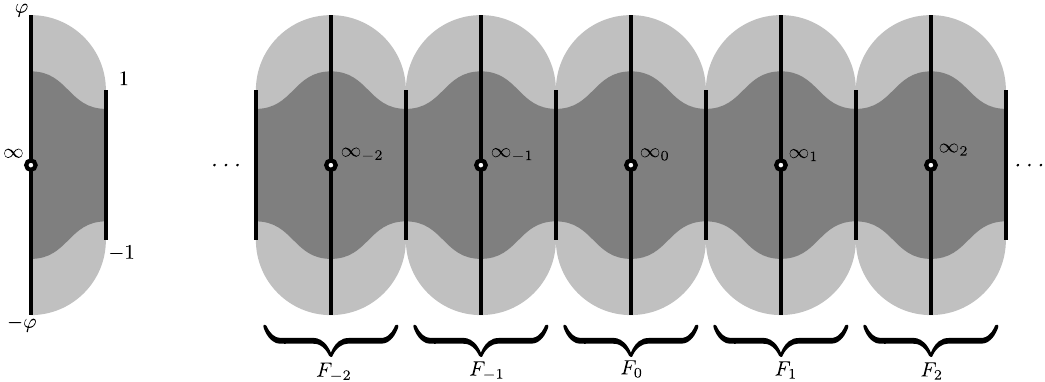}
\caption{{\em Left:} A caricature of the fundamental domain shown in Figure \ref{fig:fundamental domain}. Geodesic boundaries are drawn in black. {\em Right:}
The quotient surface $O(2) \bs \SL_\pm(2,\R)/ \tilde V'$ tiled by copies of this fundamental domain.}
\label{fig:periodic disk}
\end{figure}

We use the regions $F_n$ in $\SL_\pm(2,\R)/ \tilde V'$ to code $g^t$-trajectories in $\SL_\pm(2,\R)/ \tilde V'$ arising from non-divergent trajectories in $\SL_\pm(2,\R)/ V'$:
\begin{proposition}[Coding walk]
\label{prop:walk}
Let $A \in \SL(2,\R)$. 
If the trajectory $g^t A V'$ is non-divergent in forward time on $\SL_\pm(2,\R)/ V'$, then there is a sequence of countably many times
$$0=t_0 < t_1<t_2< \ldots$$
with $\lim_{k \to \infty} t_k=+\infty$ and a sequence of integers $\{n_k~:~k=0,1,2, \ldots\}$
so that 
$$O(2) g^t A \tilde V' \in F_{n_k} \quad \text{whenever $t_k<t<t_{k+1}$}.$$ 
Furthermore, $n_{k+1} \in \{n_k +1, n_k-1\}$ for each $k \geq 0$. 
\end{proposition}
In other words, the lift of a non-divergent geodesic on $O(2) \bs \SL_\pm(2,\R)/ V'$ to a geodesic on  $O(2) \bs \SL_\pm(2,\R)/ \tilde V'$ gives rise to a walk on the integers. We'll call this walk the {\em coding walk} of the geodesic $g^t A \tilde V'$.
\begin{proof}
If this is not true, then the geodesic change regions only finitely many times.
Then the geodesic eventually stays in one region, say $F_n$. But, we can see that the only geodesics that stay within $F_n$ forever are geodesics which exit the cusp $\infty_n$ or which exit the preimage of the convex core of 
$O(2) \bs \SL_\pm(2,\R)/ V'$ in $F_n$. (The region $F_n$ is an annulus with parabolic monodromy.) Both possibilities contradict the non-divergence of $g^t A V'$ in $\SL_\pm(2,\R)/V'$.
\end{proof}

The following is a tool which will be useful to obtain convergence in $\tilde O_{\Z_2}(S_L,\alpha_L)$
using the coding walk.

\begin{lemma}
\label{lem:convergence}
Let $U \subset O(2) \bs \SL_\pm(2,\R) / V'$ be a cusp neighborhood given by the points in the fundamental domain
of Figure \ref{fig:fundamental domain} with imaginary part larger than $2$. Let $C_d$  be a metric $d$-neighborhood of the convex core of $\subset O(2) \bs \SL_\pm(2,\R) / V'$. Let $A \in \SL(2,\R)$ and assume $g^t A V'$ is non-divergent in forward time on $\SL_\pm(2,\R)/V'$. Let $\{n_k\}$ be the associated coding walk. Let $(\tilde S_h,\tilde \alpha_h) \in \Cov_{\Z_2}(S_L,\alpha_L)$. There is a compact
set $K \subset \SL_\pm(2,\R)$ so that if $t$ lies in the interval $[t_k, t_{k+1}]$ associated to $n_k$ (as in Proposition \ref{prop:walk})
and $O(2) g^t A V' \in C_d \smallsetminus U$ then there is a $M \in K$ so that
$$g^t A (\tilde S_h,\tilde \alpha_h) = M \left(\tilde S_{\tau^{-n_k}_\ast(h)},\tilde \alpha_{\tau^{-n_k}_\ast(h)}\right)
\quad \text{in $\tilde O_{\Z_2}(S_L,\alpha_L)$.}
$$
Moreover, for any $k$ sufficiently large, there exists such a $t \in [t_k, t_{k+1}]$.
\end{lemma} 
\begin{proof}
Recall there is a covering map $O(2) \bs \SL_\pm(2,\R) / \tilde V' \to O(2) \bs \SL_\pm(2,\R) / V'$. 
Consider the region $F_0$ in the domain. Let $F_0' \subset O(2) \bs \SL_\pm(2,\R) / \tilde V'$ be the $d$-neighborhood of preimage of the convex core of
$O(2) \bs \SL_\pm(2,\R) / V'$ in $F_0$ with the preimage of $U$ removed. Then $F_0'$ is compact and this covering map sends $F_0'$ onto $C_d \smallsetminus U$. Let
$\tilde F_0' \subset \SL_\pm(2,\R) / \tilde V'$ be the preimage of $F_0'$ under projection to $O(2) \bs \SL_\pm(2,\R) / \tilde V'$. Then $\tilde F_0'$ is also compact and so there is a compact set $K \subset \SL_\pm(2,\R)$ which projects to $\tilde F_0'$ under the quotient map $\SL_\pm(2,\R) \to \SL_\pm(2,\R) / \tilde V'$.

Now consider a $t \in [t_k, t_{k+1}]$ so that $O(2) g^t A V' \in C_d \smallsetminus U$. Then by definition of the coding sequence, $O(2) g^t A \tilde V' \in F_{n_k}$. Then 
\begin{equation}
\label{eq:tau image}
\tau^{-n_k} \big(O(2) g^t A \tilde V'\big) \in F_0;
\end{equation}
see \eqref{eq:action on regions} and \eqref{eq:translation action}. Since $O(2) g^t A V'$ is asymptotic to the convex core in $O(2) \bs \SL_\pm(2,\R) / V$,
for $k$ sufficiently large the point in \eqref{eq:tau image} actually lies in $F_0'$.
Observe also that
$$\tau^{-n_k} \big(O(2) g^t A \tilde V'\big)=O(2) g^t A D\big((\rho \circ \psi)^{n_k}\big)\tilde V',$$
so that from the above there is an $R \in \tilde V'$ and an $M \in K$ so that
$$g^t A D\big((\rho \circ \psi)^{n_k}\big) R = M.$$

Now observe that
$$g^t A (\tilde S_h,\tilde \alpha_h)=M\cdot \Big(R^{-1} \circ D\big((\rho \circ \psi)^{-n_k}\big)(\tilde S_h,\tilde \alpha_h) \quad \text{in $\tilde O_{\Z_2}(S_L,\alpha_L)$,}
$$
where we think of the expression to the right of $M\cdot$ as the action of the Veech group of $(S_L,\alpha_L)$
on $\Cov_{\Z_2}(S_L,\alpha_L)$. Here $\Cov_{\Z_2}(S_L,\alpha_L)=\Hom(\Gamma,\Z_2)$; see \eqref{eq:covers equals hom}.
Thus, the action of $V'$ on $\Hom(\Gamma,\Z_2)$ factors through $\Delta=V'/\tilde V'$ and in particular,
$$\Big(R^{-1} \circ D\big((\rho \circ \psi)^{-n_k}\big)(\tilde S_h,\tilde \alpha_h)=(\tilde S_{\tau^{-n_k}_\ast(h)},\tilde \alpha_{\tau^{-n_k}_\ast(h)}) \quad \text{in $\Cov_{\Z_2}(S_L,\alpha_L)$.}$$

Finally, to obtain such a $t \in [t_k, t_{k+1}]$ for $k$ sufficiently large, observe that because $g^t A V'$ is asymptotic to the convex core in $\SL_\pm(2,\R)/V'$, it eventually lies in $C_d$. Values of $t=t_k$ and $t=t_{k+1}$ satisfy the statement because at these times
$O(2) g^{t_k} A \tilde V'$ lies in the boundary of a region, so the corresponding point $O(2) g^{t_k} A V'$ in $O(2) \bs \SL_\pm(2,\R) / V'$ lies on the geodesic joining $-1$ to $1$ in 
Figure \ref{fig:fundamental domain} and in particular this point does not lie in $U$.
\end{proof}

We already understand the ergodic properties of the translation flow of covers when $g^t A \tilde V'$ is non-divergent; see Proposition \ref{prop:non-divergence covers}. This non-divergence can be captured by the coding walk:

\begin{proposition}
\label{prop:non-divergence in tilde V'}
Assume $g^t A V'$ is non-divergent in forward time on $\SL_\pm(2,\R)/ V'$.
The trajectory $g^t A \tilde V'$ is non-divergent in $\SL_\pm(2,\R) / \tilde V'$ if and only if the coding walk $\langle n_k\rangle$ is recurrent, i.e., there is an $n \in \Z$ so that $n_k=n$ for infinitely many $k$.
\end{proposition}
\begin{proof}
First suppose $g^t A \tilde V'$ is non-divergent. Then $g^t A \tilde V'$ has an $\omega$-limit point in $\SL_\pm(2,\R) / \tilde V'$. This limit point must lie in some region $F_n$. Proposition \ref{prop:walk} tells us we must visit $F_n$ infinitely often, so infinitely many $n_k=n$. Now suppose the coding walk recurs to $n$. Then infinitely many times the trajectory $g^t A \tilde V'$  moves from $F_{n-1}$ to $F_n$ or from $F_{n+1}$ to $F_n$. This means that the trajectory $g^t A \tilde V'$ infinitely often visits a compact neighborhood of the common boundary between two adjacent regions. But then the trajectory must have an $\omega$-limit point in this compact neighborhood,
and so the trajectory is non-divergent.
\end{proof}

The proposition above tells us that the only way devious covers could exist
is if the coding walk does not recur. In this case $\lim n_k=+\infty$ or $\lim n_k=-\infty$. Theorem \ref{thm:1} tells us that covers which are not devious
have ergodic translation flow. So, it remains to understand the devious covers
when $\lim n_k= \pm \infty$. We now state our main result of the section.

\begin{theorem}[Main result on double covers of the ladder surface]
\label{thm:main ladder}
Let $A \in \SL(2,\R)$ and suppose that the geodesic $g^t A V'$ is non-divergent in $\SL_\pm(2,\R) / V'$. Let $\langle n_k \rangle$ be the coding walk of the geodesic $g^t A \tilde V'$, and assume this walk does not recur.
Then $\lim_{k \to \infty} n_k=s \infty$ where $s\in\{\pm 1\}$ is a sign.
In this case, there exist covers $(\tilde S_h, \tilde \alpha_h) \in \Cov_{\Z_2}(S_L,\alpha_L)$ so that $A(\tilde S_h, \tilde \alpha_h)$ is a devious double cover
of $A(S_L,\alpha_L)$. We define the {\em growth exponent of the visit count} to be
$$v=\limsup_{N \to s \infty} \big(\# \{k:~n_k=N\}\big)^\frac{1}{|N|}.$$
Then the following hold:
\begin{enumerate}
\item[(L2)] If $v>\varphi^2$, then  all devious double covers $A(\tilde S_h, \tilde \alpha_h)$ have ergodic
translation flow. 
\item[(L3)] If $v< \varphi^2$, then all devious double covers $A(\tilde S_h, \tilde \alpha_h)$ have non-ergodic
translation flow. 
\end{enumerate}
\end{theorem}

Our first order of business in proving this theorem is to characterize the devious double covers $A(\tilde S_h, \tilde \alpha_h)$ of $A(S_L,\alpha_L)$.
We will prove the following:

\begin{lemma}[Criterion for deviousness]
\label{lem:deviousness}
In the context of the theorem above, suppose that $\langle n_k \rangle$ limits
to $s\infty$ with $s \in \{\pm 1\}$. 
Let $A \in \SL_\pm(2,\R)$ and $h \in \Hom(\Gamma,\Z_2)$ be chosen so that
$h \neq \0$, where  $\0 \in \Hom(\Gamma,\Z_2)$ denotes the trivial homomorphism.
Then, the cover $A(\tilde S_h, \tilde \alpha_h)$ of $A(S_L,\alpha_L)$ is devious if and only if
$$\lim_{m \to +\infty} \tau^{-sm}_\ast(h)= \0.$$
\end{lemma}

\begin{remark}[Reduction to case $s=-1$]
\label{rem:reflection simplification}
It suffices to prove Lemma \ref{lem:deviousness} in the case of $s=-1$. 
To see this fix $A$ and suppose $\lim n_k=+\infty$.
Observe that $A(\tilde S_h, \tilde \alpha_h)$ is a devious double cover of $A(S_L,\alpha_L)$ if and only if $A'(\tilde S_{\psi_\ast(h)}, \tilde \alpha_{\psi_\ast(h)})$ is a devious double cover of $A'(S_L,\alpha_L)$, since the surfaces are translation equivalent. Let $A'=D(\psi)A$ and $h'=\psi_\ast(h)$.
Use $A'$ to define the coding walk $\langle n_k' \rangle$. Because of the action of $D(\psi)$ on regions (see \eqref{eq:delta} and \eqref{eq:action on regions}),
we have $n_k'=-n_k$ so that $\lim n_k=-\infty$. Finally observe that
$\lim_{m \to \infty} \tau^{-m}_\ast(h)=\0$ if and only if 
$\lim_{n \to \infty} \tau^{-n}_\ast(h')=\0$ since
$$\tau^{-n}_\ast(h')=\tau^{-n} \circ \psi_\ast(h)=
\psi_\ast \circ \tau_\ast^{n}(h) \quad \text{for all $n$}$$
as the action factors through the $\Isom(\Z)$; see \eqref{eq:ast action}.
Similar observations demonstrate that we can assume $s=-1$ in the proof of Theorem \ref{thm:main ladder} as well.
\end{remark}

One direction of Lemma \ref{lem:deviousness} is fairly easy:

\begin{proof}[Proof of the ``only if'' part of Lemma \ref{lem:deviousness}]
From the remark above, we can assume $s=-1$. Observe that a cover $(\tilde S_{h'},\tilde \alpha_{h'})$ is connected if and only if $h' \neq \0$. 
So, assume $\lim n_k=-\infty$
and there is an $h' \neq \0$ which is an accumulation point of $\tau^m_\ast(h)$ as $m \to +\infty$. Let $m(j)$ be an increasing sequence of integers so that $\tau^{m(j)}_\ast(h) \to h'$.
We will show that $g^t A(\tilde S_h, \tilde \alpha_h)$ has a connected accumulation point in  $\tilde \sO_{\Z_2}(S_L,\alpha_L)$.

Since the walk $\langle n_k \rangle$ tends to $-\infty$, for sufficiently large integers $j$, there is a $k(j)$ so that $n_{k(j)}=-m(j)$. From Lemma \ref{lem:convergence}, there is a sequence of times $t^j$ taken from the intervals associated to $n_{k(j)}$, a compact subset $K \subset \SL_\pm(2,\R)$, and $M_j \in K$ so that
$$g^{t^j} A (\tilde S_h, \tilde \alpha_h) = M_j \left(\tilde S_{\tau^{-n_{k(j)}}_\ast(h)},\tilde \alpha_{\tau^{-n_{k(j)}}_\ast(h)}\right).$$
Since $n_{k(j)}=-m(j)$ we know $\lim_{k \to \infty} \tau^{-n_{k(j)}}_\ast(h)=h'$. Since $K$ is compact, we find a subsequence $M_{j(i)}$ converging to some $M \in K$. Then 
$$\lim_{i \to \infty} M_{j(i)} \left(\tilde S_{\tau^{-n_{k \circ j(i)}}_\ast(h)},\tilde \alpha_{\tau^{-n_{k \circ j(i)}}_\ast(h)}\right)=M (\tilde S_{h'}, \tilde \alpha_{h'}) \quad \text{in $\tilde \sO(S_L,\alpha_L)$.}$$
This is our desired connected accumulation point.
\end{proof}

The other direction is more difficult in part because we do not know if $V'=V(S_L,\alpha_L)$. We need to show that if $\lim_{m \to +\infty} \tau^{-sm}_\ast(h)= \0$ then every accumulation point of $g^t A(\tilde S_h, \tilde \alpha_h)$ in $\tilde \sO_{\Z_2}(S_L,\alpha_L)$ is disconnected. But $\tilde \sO_{\Z_2}(S_L,\alpha_L)$ is a quotient of $\SL(2,\R) \times \Cov_{\Z_2}(S_L,\alpha_L)$ by $V(S_L,\alpha)$ which we do not know. 

Fortunately, we are restricting attention to $A \in \SL(2,\R)$ so that $g^t A V'$ is non-divergent in $\SL_\pm(2,\R)/V'$. Then the geodesic $g^t A V'$ is asymptotic to the convex core.
We will use this and a compactness argument to argue that if $g^{t_i} A(\tilde S_h, \tilde \alpha_h)$ converges then there is a subsequence converging to a disconnected cover. Assuming that the space $\tilde \sO_{\Z_2}(S_L,\alpha_L)$ is Hausdorff, the two limit points must be the same showing that every accumulation point is disconnected. So we need:

\begin{proposition}
\label{prop:Hausdorff}
The space $\tilde \sO_{\Z_2}(S_L,\alpha_L)$ is Hausdorff.
\end{proposition}
\begin{proof}
Let $A(\tilde S_h, \tilde \alpha_h)$ and $A'(\tilde S_{h'}, \tilde \alpha_{h'})$
be distinct points in $\tilde \sO_{\Z_2}(S_L,\alpha_L)$. We will find an open sets that isolate them from each other. 

Observe that there is a natural continuous map $\tilde \sO_{\Z_2}(S_L,\alpha_L) \to \sO(S_L,\alpha_L)$ which sends $A(\tilde S_h, \tilde \alpha_h)$
to the quotient of $A(\tilde S_h, \tilde \alpha_h)$ by its translation automorphisms. Since all double covers are regular and $(S_L,\alpha_L)$ has no translation automorphisms, there is always a $\Z_2$ action on $A(\tilde S_h, \tilde \alpha_h)$ by its translation automorphisms. If under this map
$A(\tilde S_h, \tilde \alpha_h)$ and $A'(\tilde S_{h'}, \tilde \alpha_{h'})$ project to different points, then we can build disjoint open sets by lifting disjoint open sets from $\sO(S_L,\alpha_L)$. Recall that $\sO(S_L,\alpha_L)$ is naturally identified with $\SL_\pm(2,\R)/V(S_L,\alpha_L)$ which is Hausdorff because $V(S_L,\alpha_L)$ is discrete by Proposition \ref{prop:discrete}.

Now suppose that  $A(\tilde S_h, \tilde \alpha_h)$ and $A'(\tilde S_{h'}, \tilde \alpha_{h'})$ both cover $A(S_L,\alpha_L)$. By left multiplying both surfaces by $A^{-1}$ we can assume without loss of generality that both surfaces cover
$(S_L,\alpha_L)$. Then they both lie in the $\Cov_{\Z_2}(S_L,\alpha_L)$. 
Recall $\Cov_{\Z_2}(S_L,\alpha_L)=\Hom(\Gamma,\Z_2)$; see \eqref{eq:covers equals hom}. This space is a Cantor set and so is Hausdorff. So, we can find open sets $U_1$ and $U_2$ that isolate the two surfaces within $\Cov_{\Z_2}(S_L,\alpha_L)$. Since the Veech groups are discrete, we can find a neighborhood $N \subset \SL_\pm(2,\R)$ so that no non-trivial element of $V(S_L,\alpha_L)$ has the form $B_1^{-1} B_2$ where $B_1,B_2 \in N$. We claim that the images of
$N \times U_1$ and $N \times U_2$ are disjoint in $\tilde \sO_{\Z_2}(S_L,\alpha_L)$. Let $B_1 (S_{h_1},\alpha_{h_1})$ and $B_2 (S_{h_2},\alpha_{h_2})$ be two surfaces taken from these neighborhoods. Then they are translation equivalent if and only if  $(S_{h_1},\alpha_{h_1})$ is equivalent to $B_1^{-1} B_2 (S_{h_2},\alpha_{h_2})$. If they are translation equivalent, we must have
$B_1^{-1} B_2= I $ from arguments in the previous paragraph together with the definition of $N$. But $(S_{h_1},\alpha_{h_1}) \in U_1$ and $(S_{h_2},\alpha_{h_2}) \in U_2$, so they can not be translation equivalent.
\end{proof}

The following will finish our proof of Lemma \ref{lem:deviousness}.

\begin{proof}[Proof of the ``if'' part of Lemma \ref{lem:deviousness}]
Suppose that $\lim n_k=-\infty$ and $\lim_{m \to +\infty} \tau^{m}_\ast(h)= \0$. We must show that the cover $A(\tilde S_{h}, \tilde \alpha_h)$ of $A(S_L, \alpha_L)$ is devious.

We will show every $\omega$-limit point of $g^t A (\tilde S_{h}, \tilde \alpha_{h})$ in $\tilde \sO_{\Z_2}(S_L,\alpha_L)$ is disconnected.
Suppose that $t_i$ is a sequence of times tending to $+\infty$ and 
 $$\lim_{i \to \infty} g^{t_i} A (\tilde S_{h}, \tilde \alpha_{h})=B(\tilde S_{h'}, \tilde \alpha_{h'}) \quad \text{in $\tilde O_{\Z_2}(S_L,\alpha_L)$}.$$

Consider the sequence of points $p_i=O(2) g^{t_i} A V'$ in $O(2) \bs \SL_\pm(2,\R) / V'$. We claim that there is a cusp neighborhood $U$ of $O(2) \bs \SL_\pm(2,\R) / V'$ (given by $\mathrm{Im}(z)>c$ for some $c \geq 2$ in the fundamental domain of Figure \ref{fig:fundamental domain}) so that $p_i \not \in U$ for all $i$. Otherwise there is a subsequence where $p_{i_j}$ exists the cusp. However, in this case the injectivity radius of the surface $g^{t_{i_j}} A (S_L,\alpha_L)$ would tend to zero, because $g^{t_{i_j}} A$ would contract the circumferences of some cylinder decomposition (which up the action of $V'$ is vertical) more an more as we move up the cusp, and thus $g^{t_{i_j}} A (S_L,\alpha_L)$ would have injectivity radius tending to zero as $j \to \infty$. For the same reason,
the injectivity radius of the surfaces $g^{t_{i_j}} A (\tilde S_h \tilde \alpha_h)$ would have to tend to zero. But this contradicts that this sequence limits to $B(\tilde S_{h'}, \tilde \alpha_{h'})$.

Let $C_d$ be a $d$-neighborhood of the compact core of $O(2) \bs \SL_\pm(2,\R) / V'$. Since by hypothesis $g^t A V'$ is non-divergent in $\SL_\pm(2,\R)/V'$, we know that this trajectory is asymptotic to the convex core. Combining this with this the previous paragraph, we see that
$$O(2) g^{t_i} A V' \in C_d \smallsetminus U \quad \text{for $i$ sufficiently large.}$$
By compactness of the fiber over $C_d \smallsetminus U$, we can assume that after passing to a subsequence that 
$g^{t_i} A V'$ converges in $\SL_\pm(2,\R)/V'$.

Now consider the sequence $g^{t_i} A \tilde V'$ in $\SL_\pm(2,\R) / \tilde V'$. Let $k(i)$ be such that $t_i \in [t_{k(i)}, t_{k(i)+1}]$. Defining
$$\tilde p_i=O(2) g^{t_i} A \tilde V' \quad \text{we see} \quad
\tilde p_i \in F_{n_{k(i)}} \quad  \text{for each $i$,}$$ 
and so $\tau^{-n_{k(i)}}(\tilde p_i) \in F_{0}$ for each $i$ where
$\tau^{-n_{k(i)}}$ is acting as an element of the deck group of the covering \eqref{eq:cover} as in \eqref{eq:translation action}. Since each point of 
$O(2) \bs \SL_\pm(2,\R) / V'$ has only two lifts to $F_0$, we see that
by passing to a subsequence we can assume
$\tau^{-n_{k(i)}}(g^{t_i} A \tilde V')$ converges in $\SL_\pm(2,\R) / \tilde V'$
above $F_0$. Recalling how $\tau$ acts (see \eqref{eq:ast action} and \eqref{eq:translation action}), this means there is a sequence $R_i \in \tilde V'$ so that
\begin{equation}
\label{eq:convergence of matrices}
g^{t_i} A D(\rho \circ \psi)^{n_{k(i)}} R_i \quad \text{converges to some limit $L \in \SL_\pm(2,\R)$}.
\end{equation}
Observe that in $\tilde \sO_{\Z_2}(S_L,\alpha_L)$ we have
$$\begin{array}{rcl}
\displaystyle g^{t_i} A (\tilde S_{h}, \tilde \alpha_{h}) & = &  
\displaystyle g^{t_i} A D(\rho \circ \psi)^{n_{k(i)}} R_i \cdot \big(R_i^{-1} D(\rho \circ \psi)^{-n_{k(i)}} (\tilde S_{h}, \tilde \alpha_{h})\big) \\
& = & \displaystyle g^{t_i} A D(\rho \circ \psi)^{n_{k(i)}} R_i
\big(\tilde S_{\tau^{-n_{k(i)}}_\ast(h)}, \tilde \alpha_{\tau^{-n_{k(i)}}_\ast(h)}\big).
\end{array}$$
The $\SL_\pm(2,\R)$ part of this last expression converges to $L$. The surface part converges to $(\tilde S_{\0}, \tilde \alpha_\0)$ because $\lim n_k=-\infty$ and $\lim_{m \to +\infty} \tau^{m}_\ast(h)= \0$. So, because $\tilde O_{\Z_2}(S_L,\alpha_L)$ is Hausdorff by Proposition \ref{prop:Hausdorff}, we must have
$$B(\tilde S_{h'}, \tilde \alpha_{h'})=\lim_{i \to \infty} \displaystyle g^{t_i} A (\tilde S_{h}, \tilde \alpha_{h})=L (\tilde S_{\0}, \tilde \alpha_\0).$$
In particular,
$B(\tilde S_{h'}, \tilde \alpha_{h'})$ is disconnected.
\end{proof}

Now we will study devious covers by studying those $h$ so that $\tau^m_\ast(h)$ decays to $\0$. A formula for $\tau_\ast$ is given below.

\begin{proposition}
\label{prop:tau 1}
The action of translation by one on $\Hom(\Gamma, \Z_2)$ is given by 
\begin{equation*}
\big(\tau_\ast(h)\big)(i)=\begin{cases}
h(-i)+h({i-1})+h({i})+h({i+1}) & \text{if $i<-2$,} \\
h(-i)+h({-2})+h({-3}) & \text{if $i=-2$,} \\
h(-i)+h({-2}) & \text{if $i=-1$,} \\
h(-i) & \text{if $i>0$.} 
\end{cases}
\end{equation*}
The action of translation by negative one is given by:
\begin{equation*}
\big(\tau^{-1}_\ast(h)\big)(i)=\begin{cases}
h(-i) & \text{if $i <0$,} \\
h(-i)+h({2}) & \text{if $i=1$,} \\
h(-i)+h({2})+h({3}) & \text{if $i=2$,} \\
h(-i)+h({i-1})+h({i})+h({i+1}) & \text{if $i>2$.}
\end{cases}
\end{equation*}

\end{proposition}
\begin{proof}
This follows from the fact that $\tau=\delta \circ D(\rho \circ \psi)$ and
$\tau^{-1}=\delta \circ D(\psi \circ \rho)$; see equation \ref{eq:delta}. The actions of $\psi$ and $\rho$ on $\Hom(\Gamma, \Z_2)$
are given in equations \ref{eq:psi ladder} and \ref{eq:rho ladder}, respectively.
\end{proof}

To analyze the case when $\lim_{m \to +\infty} \tau^m_\ast(h)=\0$,
we introduce the {\em proximity function} which measures how close a homomorphism is to the trivial homomorphism $\0$:
\begin{equation}
\label{eq:proximity}
P:\Hom(\Gamma, \Z_2) \to \{1,2,3, \ldots, \infty\};
\quad 
P(h)= \begin{cases}
\infty & \text{if $h=\0$,}\\
\min \{|i|~:~h(i) \neq 0\} & \text{otherwise}.
\end{cases}
\end{equation}
Observe a sequence $h_n$ tends to $\0$ if and only if $P(h_n) \to \infty$. 

For each integer $k \geq 0$, we define a cylinder set in $\Hom(\Gamma, \Z_2)$ by
$$C_k=\{h~:~\text{$h(k)=h(-k-1)=1$  and $h(i)=0$ when $-k-1<i<k$}\}.$$ 
Note that $h \in C_k$ implies $P(h)=k$.
These cylinder sets turn out to be very important, for understanding which elements of $\Hom(\Gamma, \Z_2)$ are limit to the trivial homomorphism $\0$
under translation.

\begin{lemma}
\label{lem:limits to zero}
The collection of $h \in \Hom(\Gamma, \Z_2)$ so that $\lim_{n \to +\infty} \tau^n_\ast(h)=\0$ is given by
$$
\{\0\} \cup \bigcup_{j \geq 0} \tau^{-j}_\ast(Z)
\quad \text{where}
\quad Z=\bigcup_{k \geq 2} \bigcap_{m \geq 0} \tau^{-m}_\ast(C_{k+m}).$$
Furthermore, for any function $f:\{i~:~i \geq 2\} \to \Z_2$ which is not identically zero,
there is an $h \in Z$ so that $h(i)=f(i)$ for all $i \geq 2$. 
\end{lemma}
The second statement says that there are a number of elements of $\Hom(\Gamma, \Z_2)$ which limit on $\0$. (Informally, $Z$ has half the information entropy of $\Hom(\Gamma, \Z_2)$, or with an appropriate natural metric, $Z$ has half the Hausdorff dimension of $\Hom(\Gamma, \Z_2)$.)
This together with Lemma \ref{lem:deviousness} prove the existence of devious covers as stated in Theorem \ref{thm:main ladder}.b

The following proposition is the main ingredient in the proof of the lemma above.

\begin{proposition}
\label{prop:cylinder}
\begin{enumerate}
\item If $k \geq 2$ and $h \in C_k$, then $P\big(\tau_\ast(h)\big)=k+1$. 
\item If $k \geq 2$ is an integer, then $\tau^{-1}_\ast(C_{k+1}) \subset C_k$.
\item If $k \geq 2$, $P(h)=k$ and $h \not \in C_k$, then either 
$P\big(\tau_\ast(h)\big)=k-1$ and $\tau_\ast(h) \not \in C_{k-1}$
or $P\big(\tau^2_\ast(h)\big)=k-1$ and $\tau^2_\ast(h) \not \in C_{k-1}$.
\end{enumerate}
\end{proposition}
\begin{proof}
We prove each of the statements below.

{\bf (1) } Suppose that $h \in C_k$ and $k \geq 2$. Then $h(k)=h(-k-1)=1$ while $h(i)=0$ for $-k \leq i \leq k-1$. Then $P\big(\tau_\ast(h)\big)\leq k+1$, because
$\tau_\ast(h)(k+1)=h(-k-1)=1$ 
by Proposition \ref{prop:tau 1}. For $0<i \leq k$, we have
$\tau_\ast(h)(i)=h(-i)=0$. Also observe that
$$\tau_\ast(h)(-k)=h(k)+h(-k-1)=1+1=0,$$
while the other terms from Proposition \ref{prop:tau 1} vanish. Finally, all terms in the expression for $\tau_\ast(h)(i)$ vanish when $-k<i<0$. Therefore,
$P\big(\tau_\ast(h)\big)= k+1$ as claimed.

{\bf (2) } Suppose $h \in C_{k+1}$ with $k \geq 2$. Then $h(k+1)=h(-k-2)=1$ while $h(i)=0$ for $-k-1 \leq i \leq k$. By Proposition \ref{prop:tau 1}, we see that
$\tau^{-1}_\ast(h)(-k-1)=h(k+1)=1.$ Also,
$$\tau^{-1}_\ast(h)(k)=\begin{cases}
h(-k)+h(2)+h(3)=0+0+1=1 & \text{if $k=2$,}\\
h(-k)+h(k-1)+h(k)+h(k+1)=0+0+0+1=1 & \text{if $k>2$}.
\end{cases}
$$
Now suppose $-k \leq i <0$. Here we have
$\tau^{-1}_\ast(h)(i)=h(-i)=0.$
Finally consider the case when $0 < i < k$. All terms vanish from
the expression for the expression $\tau^{-1}_\ast(h)(i)$ given in Proposition \ref{prop:tau 1}, so $\tau^{-1}_\ast(h)(i)=0$.
Taken together, we see $\tau^{-1}_\ast(h) \in C_k$.

{\bf (3) } Suppose $k \geq 2$, $P(h)=k$ and $h \not \in C_k$.
Since $P(h)=k$, we know that $h(i)=0$ when $|i|<k$.  Observe that the set of $h$ with $P(h)=k$ and
$h \not \in C_k$ is the union of the two pieces:
$$A=\{h~:~\text{$P(h)=k$ and $h(-k)=1$}\},$$
$$B=\{h~:~\text{$P(h)=k$, $h(-k)=0$, $h(k)=1$, and $h(-k-1)=0$}\}.$$
First assume that $h \in A$ so that $h(-k)=1$ but $h(i)=0$ when $|i|<k$. Then, 
$\tau^{1}_\ast(h)(-k+1)$ is given by one of the expressions
$$
\begin{cases}
h(k-1)+h(-k)+h(-k+1)+h(-k+2)=0+1+0+0=1 & \text{if $k>3$},\\
h(k-1)+h(-2)+h(-3)=0+0+1=1 & \text{if $k=3$},\\
h(k-1)+h(-2)=0+1=1 & \text{if $k=2$}.
\end{cases}$$
Since proximity can decrease by at most one when applying $\tau^{1}_\ast$,
we see $P\big(\tau^{1}_\ast(h)\big)=k-1$. Also, we have by definition of $C_{k-1}$
that $\tau^{1}_\ast(h) \not \in C_{k-1}$.
Now consider the case when $h \in B$. Then $h(i)=0$ when $-k-1 \leq i<k$
and $h(k)=1$. In this case $\tau^{1}_\ast(h)(i)=0$ when $|i|<k-1$ since proximity can decrease by at most one. We observe $\tau^{1}_\ast(h)(k-1)=h(-k+1)=0.$
All terms vanish in the expression for $\tau^{1}_\ast(h)(-k+1)$ given by 
Proposition \ref{prop:tau 1}, so $\tau^{1}_\ast(h)(-k+1)=0.$
On the other hand,
$$\tau^{1}_\ast(h)(-k)=\begin{cases}
h(k)+h({-k-1})+h({-k})+h({-k+1})=1+0+0+0=1 & \text{if $k>2$,} \\
h(k)+h({-2})+h({-3})=1+0+0=1 & \text{if $k=2$.} 
\end{cases}
$$
Since $\tau^{1}_\ast(h)(-k)=1$, we see that $\tau^{1}_\ast(h) \in A$.
Because of our discussion of what happens for elements of $A$, we
see that $P\big(\tau^2_\ast(h)\big)=k-1$ and $\tau^2_\ast(h) \not \in C_{k-1}$.
\end{proof}

In the proof of Lemma \ref{lem:limits to zero}, it is useful to note the following Corollary to the Proposition above.
\begin{corollary}
Let $k \geq 1$ be an integer.
If $P(h)=k$ but $h \not \in C_k$, then there is an $n \geq 0$ so that $P\big(\tau^n_\ast(h)\big)=1$.
\end{corollary}
\begin{proof}
We prove this by induction in $k$.
This holds with $n=0$ when $k=1$. When $k=2$, this follows from statement (3) of Proposition \ref{prop:cylinder}. Now let $k \geq 3$ and suppose the conclusion holds  when $P(h)=k-1$ and $h \not \in C_{k-1}$. Suppose $P(h)=k$ and $h \not \in C_k$. Again using statement (3) of Proposition \ref{prop:cylinder}, we see that 
either $P\big(\tau_\ast(h)\big)=k-1$ and $\tau_\ast(h) \not \in C_{k-1}$
or $P\big(\tau^2_\ast(h)\big)=k-1$ and $\tau^2_\ast(h) \not \in C_{k-1}$. Using our induction hypothesis applied to these cases, we see the conclusion holds, and the whole statement holds by induction.
\end{proof}

\begin{proof}[Proof of Lemma \ref{lem:limits to zero}]
Let $Z=\bigcup_{k \geq 2} \bigcap_{m \geq 0} \tau^{-m}_\ast(C_{k+m})$ as in the theorem. Let $h \in Z$. Then there is a $k \geq 2$ so that $\tau^m_\ast(h) \in C_{k+m}$ for all $m \geq 0$. Then $P\big(\tau^m_\ast(h)\big)=k+m$ tends to $\infty$ as $m \to \infty$. Thus, $\lim_{m \to \infty} \tau^m_\ast(h)=\0$. From this it follows that everything in the set listed in the theorem, 
$\{\0\} \cup \bigcup_{j \geq 0} \tau^{-j}_\ast(Z)$, limits to $\0$.

Conversely, suppose $h \in \Hom(\Gamma, \Z_2)$ satisfies $\lim_{n \to \infty} \tau^n_\ast(h)=\0$ and $h \neq \0$. We know that $P \big(\tau^n_\ast(h)\big) \to \infty$ as $n \to \infty$. In particular, we see that for any $k \geq 1$, there are only finitely many $n$ so that 
\begin{equation}
\label{eq:proximity2}
P \big(\tau^n_\ast(h)\big)=k.
\end{equation}
Let $K \geq 2$ be the smallest value of $k$ larger than one so that the above equation has a solution for $n$, and let $N$ be the maximal $n$ satisfying the equation when $k=K$. We claim that
\begin{equation}
\label{eq:h}
\tau^N_\ast(h) \in \bigcap_{n \geq 0} \tau^{-n}_\ast(C_{K+n}).
\end{equation}
Otherwise, there is a smallest $n \geq 0$ so that $\tau^{N+n}_\ast(h) \not \in C_{K+n}$. Then by the corollary above, we see that there is an $m \geq N+n$
so that $P\big(\tau^m_\ast(h)\big)=1$. But since $P \big(\tau^n_\ast(h)\big)$ tends to $\infty$ and as $n$ increases the proximity can increase by at most one,
we see that there is an $m'>m$ so that \eqref{eq:proximity2} is satisfied for $n=m'$. But this contradicts the definition of $N$.  Therefore, \eqref{eq:h} is true after all. Observe that the equation implies that
$\tau^N_\ast(h) \in Z$, and thus $h \in \bigcup_{j \geq 0} \tau^{-j}_\ast(Z)$
as desired.

It remains to prove the last sentence of the theorem which guarantees there are a lot of $h$ so that $\lim_{n \to \infty} \tau^n_\ast(h)=\0$.
A main point here is that for each $k \geq 2$, the set 
$A_k=\bigcap_{m \geq 0}  \tau^{-m}_\ast(C_{k+m})$ is non-empty.
Indeed, as continuous images of cylinder sets, each $\tau^{-m}_\ast(C_{k+m})$ is compact. The sets being intersected are nested in the sense that
$$\tau^{-m-1}_\ast(C_{k+m+1}) \subset \tau^{-m}_\ast(C_{k+m})$$
by statement (2) of the lemma. Therefore, we can make a choice of an $\tilde h_k \in A_k$ for every $k \geq 2$.

Now suppose that $f: \{i \in \Z~:~i \geq 2\} \to \Z_2$ is defined and not identically zero. Let $k=\min \{i~:~f(i)\neq 0\}$. 
We will inductively define a sequence of functions $h_j \in \Hom(\Gamma, \Z_2)$
for $j \geq k$ so that the sequence converges to an extension of $f$. 
Our functions will all lie in $A_k$. Since $A_k$ is closed, this will suffice
to prove that the limiting function lies in $A_k$. We will also ensure that
\begin{equation}
\label{eq:ast}
h_j(i)=f(i) \quad \text{when $2 \leq i \leq j$}.
\end{equation}
To ensure convergence of the sequence $h_j$, we will also have that
\begin{equation}
\label{eq:prior}
h_j(i)=h_{j-1}(i) \quad \text{when $|i|<j$.}
\end{equation}
Observe that $f(k)=1$, thus we can define $h_k=\tilde h_k \in A_k$. This serves as our base case. Now assume that $j > k$ and $h_{j-1}$ is defined and satisfies the hypotheses above. We will define $h_{j}$. If $h_{j-1}(j)=f(j)$,
then we can take $h_j=h_{j-1}$. Otherwise, we define
$h_{j}=h_{j-1}+\tilde h_j$. Since $\tilde h_j \in C_{j}$, we know that $\tilde h_j(j)=1$. Therefore, we must have $h_{j}(j)=f(j)$. When $|i|<j$, we have
$\tilde h_j(i)=0$ because $h_j \in C_j$, therefore 
$$h_j(i)=h_{j-1}(i) \quad \text{when $|i|<j$}$$
by inductive hypothesis. This simultaneously ensures both equation (\ref{eq:prior})
holds and verifies equation (\ref{eq:ast}) for $i<j$ since by hypothesis
$h_{j-1}(i)=f(i)$ when $2 \leq i \leq j-1$. 
From the inductive hypothesis and definition of $\tilde h_j$, we have
$$\tau^m_\ast(h_{j-1}) \in C_{k+m} \and 
\tau^m_\ast(\tilde h_{j}) \in C_{j+m}$$
for each integer $m \geq 0$. Observe that when $j>k$, the sum of an element
in $C_{k+m}$ and $C_{j+m}$ lies in $C_{k+m}$. Therefore, we see by linearity of $\tau^m_\ast$ that
$$\tau^m_\ast(h_{j})=\tau^m_\ast(h_{j-1})+\tau^m_\ast(\tilde h_{j}) \in C_{k+m}$$
for every integer $m \geq 0$. Thus, $h_j \in A_k$ as desired. This completes the inductive step.
\end{proof}

\compat{Change (a) and (b) to (L2) and (L3)}

\begin{proof}[Proof of statement {\em (L2)} of Theorem \ref{thm:main ladder}]
We will consider the translation flow on the devious double cover $A(\tilde S_h, \tilde \alpha_h)$ of $A(S_L,\alpha_L)$ under  several assumptions.
We assume the trajectory $g^t A V'$ is non-divergent in $\SL_\pm(2,\R)/V'$. 
From the geodesic $g^t A \tilde V'$, we define the coding walk $\{n_k\}$.
We assume that the number of visits to any integer is finite, but that the 
growth exponent of the visit count satisfies $v>\varphi^2$.
As allowed by Remark \ref{rem:reflection simplification},
we assume that $\lim_{k \to \infty} n_k = -\infty$.  
By Lemma \ref{lem:deviousness}, we know that $\lim_{m \to +\infty} \tau^m_\ast(h)=\0.$
In particular, $\lim_{k \to \infty} \tau^{-n_k}_\ast(h)=\0.$

We will prove ergodicity holds for the translation flow on $A(\tilde S_h, \tilde \alpha_h)$
by appealing to Theorem \ref{thm:integrability}. 
(We upgrade to unique ergodicity at the end of the proof.)
Let $\eta>0$.
We will find subsurfaces of $X_t=g^t A(\tilde S_h, \tilde \alpha_h)$ and related geometric quantities so that the integral \eqref{eqn:integrability2} is infinite.

Now recall the definitions of the coding walk; see Proposition \ref{prop:walk}. 
There is a sequence of times $t_k$ so that 
$$g^t A \tilde V' \in F_{n_k} \quad \text{when $t_k<t<t_{k+1}$}.$$
Let $U$ and $C_d$ be subsets of $O(2) \bs \SL_\pm(2,\R) / V'$ as in Lemma \ref{lem:convergence}. Let
$$J_k=\{t \in (t_k,t_{k+1}):~O(2) g^t A V' \in C_d \smallsetminus U\}.$$
Since the geodesic $O(2) g^t A V'$ is non-divergent it is asymptotic to the convex core and so 
there is a ${\text \j}>0$ so that 
the constant
\begin{equation}
\label{eq:j}
{\text \j} \geq \textit{length}(J_k) \quad \text{for $k$ sufficiently large}.
\end{equation}
(In fact, $\textit{length}(J_k)$ has a uniform lower bound once the geodesic is within distance $d$
from the convex core.) By Lemma \ref{lem:convergence}, there is a compact set $K \subset \SL_\pm(2,\R)$ so that for any $t \in J_k$ there is a matrix $M_t \in K$ so that
\begin{equation}
\label{eq:X t}
X_t=M_t (\tilde S_{\tau^{-n_k}_\ast(h)}, \tilde \alpha_{\tau^{-n_k}_\ast(h)}\big) \quad \text{as elements of $\tilde \sO_{\Z_2}(S_L,\alpha_L)$}.
\end{equation}

We will now explain how to find subsurfaces. The surfaces will always be
obtained by lifting two copies of a disk in $(S_L,\alpha_L)$; so we will always
have $C_t=2$ in the language of the Theorem \ref{thm:integrability}.
We think of $(S_L,\alpha_L)$ as depicted by Figure \ref{fig:ladder_surface}: the surface is a topological disk in the plane with edge identifications. For small $\kappa>0$, consider the subset $U_0(\kappa)$ of this disk consisting of points whose distance from the boundary of the disk is greater than $\kappa$. Then let $U(\kappa) \subset U_0(\kappa)$
be the subsurface of the largest area. Since these regions exhaust the disk as $\kappa \to 0$, we can
choose a $\kappa$ so that $U=U(\kappa)$ contains more than a factor of $1-\eta$ of the disk's area.
See Figure \ref{fig:ladder_surface_subsurface} for an example of $U(\kappa)$.
We can think of $U$ as lying in the surface $(S_L,\alpha_L)$. 

\begin{figure}
\includegraphics[width=2in]{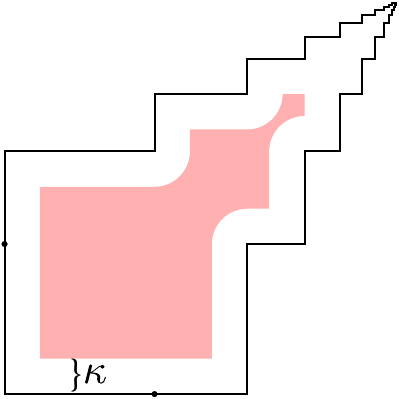}
\caption{An example $U(\kappa)$ in the disk making up $(S_L,\alpha_L)$.}
\label{fig:ladder_surface_subsurface}
\end{figure}

In order to define our surfaces $S_t \subset A(\tilde S_h, \tilde \alpha_t)$, 
observe that $A(\tilde S_h, \tilde \alpha_t)=g^{-t}(X_t)$ 
when $t \in \bigcup_k J_k$
and consider $X_t$ as in \eqref{eq:X t}. (We will not bother to estimate the value being integrated when $t \not \in \bigcup_k J_k$.)
Let $\tilde U_t^1$ and $\tilde U_t^2$ be the two lifts of $U$ to
$(\tilde S_{\tau^{-n_k}_\ast(h)}, \tilde \alpha_{\tau^{-n_k}_\ast(h)}\big)$. Then we define 
$S_t=g^{-t} M_t(\tilde U_t^1 \cup \tilde U_t^2).$ Evaluating geometric quantities of $S_t \subset A(\tilde S_h, \tilde \alpha_t)$ with $\dist_t$ is the same
as evaluating the same quantities for the subsurface $M_t(\tilde U_t^1 \cup \tilde U_t^2)$ of $X_t$. Observe that for $t \in \bigcup_k J_k$:
\begin{itemize}
\item $C_t=2$.
\item The distance from $M_t(\tilde U_t^1 \cup \tilde U_t^2)$ to the singular set has a uniform lower bound, i.e., $\epsilon(t)$ can be taken to be a positive constant, since $\epsilon(t)$ is bounded by $\kappa$ divided by the operator norm of $M_t^{-1}$ which has a uniform upper bound since $M_t \in K$ and $K$ is compact.
\item The diameters of the components have a uniform upper bound given by the diameter of $U$ times the maximal operator norm of $M \in K$. So ${\mathcal D}_t^i$ can be taken to be a constant independent of $t$ and $i$.
\end{itemize}

Finally, we need to consider the maximum over all curves joining
our two subsurfaces of the minimum $\dist_t$ distance from a point on the curve to a singularity. We will choose a canonical curve which depends mostly on the proximity of $\tau^{-n_k}_\ast(h)$ to zero, $P \circ \tau^{-n_k}_\ast(h)$, which was defined in equation \eqref{eq:proximity}. We will first define a curve joining
$\tilde U_t^1$ and $\tilde U_t^2$ in $(\tilde S_{\tau^{-n_k}_\ast(h)}, \tilde \alpha_{\tau^{-n_k}_\ast(h)}\big)$, then we will push this curve under $M_t$ into $X_t$.
(Again it suffices to measure things in $X_t$.)

Consider the cover $(\tilde S_{\tau^{-n_k}_\ast(h)}, \tilde \alpha_{\tau^{-n_k}_\ast(h)}\big)$ as a pair of infinite polygons with edges labeled as in Figure \ref{fig:ladder_surface} and identified in some way. The edge labels with the smallest absolute values which are glued so as to join up the two polygons
are of the form $p=\pm P \circ \tau^{-n_k}_\ast(h)$ in Figure \ref{fig:ladder_surface}.
We choose a curve to leave the first subsurface and move upward along the slope $1$ line of symmetry until it reaches the height of the edge $p$, then it changes trajectory by $\pm 45^\circ$ and travels through the midpoint of the edge connecting to the other disk. The curve continues until it hits the line of symmetry in the second disk, and it returns to the second subsurface along the line of symmetry of the disk. See Figure \ref{fig:ladder_surface_path}. It should be observed that there are positive constants 
$a$ and $b$ so that
the distance from this curve to metric completion are of the form
$$\min~\{a, b \varphi^{-|P \circ \tau^{-n_k}_\ast(h)|}\}.$$
(This is a consequence of the apparent self-similarity of the surface.)
Then using the fact that $M_t$ is taken from a compact set, we see we can take
\begin{itemize}
\item $\delta_t=c \min~\{a, b \varphi^{-|P \circ \tau^{-n_k}_\ast(h)|}\}$ for some $c>0$.
\end{itemize}

\begin{figure}
\includegraphics[width=5in]{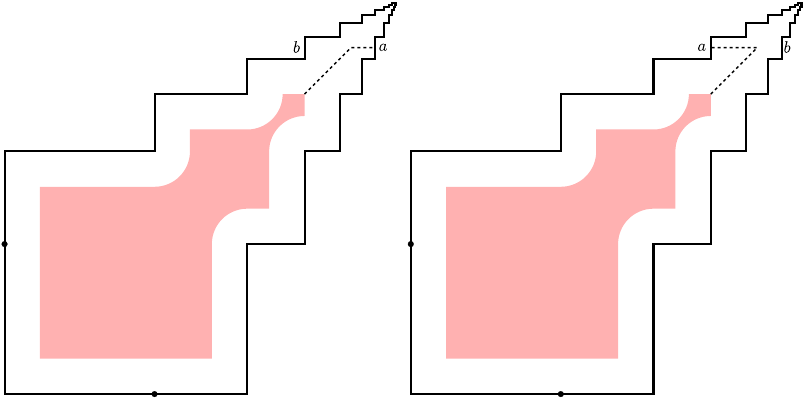}
\caption{The path joining the two subsurfaces when $P(h_\ast)=5$ and $h_\ast(5)=1$.
The letters $a$ and $b$ denote edges joining the two disks defining the cover in this case.}
\label{fig:ladder_surface_path}
\end{figure}

We now have control of all the geometric quantities and can bound the integral.
From the above definitions $d=\epsilon(t)^{-2} \sum_{i=1}^2 {\mathcal D}_t^i$ is a positive constant. For $t \in J_k$, the quantity being integrated is of the form 
$$\Big(d+\frac{1}{c \min~\{a, b \varphi^{-|P \circ \tau^{-n_k}_\ast(h)|}\}}\Big)^{-2}.$$

Now we invoke the hypothesis that $\lim_{N \to \infty} \tau^N_\ast(h)=\0$. By Lemma \ref{lem:limits to zero},
we know that there is an $M>0$ and a $L \in \Z$ so that
$\tau^N_\ast(h)$ lies in the cylinder set $C_{L+N}$ for all integers $N>M$. In particular,
$P\big(\tau^N_\ast(h)\big)=L+N$. Then for sufficiently large values of $N$, say $N \geq N_0$, we
can arrange that when $n_k=-N$,
$$\min~\{a, b \varphi^{-|P \circ \tau^{-n_k}_\ast(h)|}\}=b \varphi^{-N-L}.$$
Let $V_N=\#\{k:~-n_k=N\}$. Recall \eqref{eq:j}, the length of each $J_j$ is bounded from below by ${\text \j}>0$. This allows us to write a lower bounds for the integral as
$${\text \j} \sum_{N=N_0}^\infty V_N \big(d+\frac{1}{b c} \varphi^{N+L}\big)^{-2}.$$
An application of the root test tells us that this series diverges if 
$$v=\limsup_{N \to \infty} V_N> \varphi^2.$$
This was precisely our hypothesis, and Theorem \ref{thm:integrability} gives us ergodicity.

We will now show that ergodicity implies unique ergodicity in statement (L2).
By Corollary \ref{cor:lifting}, we just need to know that the translation flow on $A(S_L,\alpha_L)$ is uniquely ergodic. By Theorem \ref{thm:prior work}, we have unique ergodicity unless $O(2)g^t A V'$ is asymptotic to the convex core boundary. But then the lifted geodesic $O(2)g^t A \tilde V'$ on $O(2) \bs \SL_\pm(2,\R) / \tilde V'$ is asymptotic to the lifted convex core boundary, which is depicted as the boundary between the light and dark gray regions on the right side of Figure \ref{fig:periodic disk}. But then the coding walk $\langle n_k \rangle$ must grow asymptotically linearly, i.e., there is a $K$ so that for all $k>K$ we have $n_{k+1}=n_k+1$  or for all $k>K$ we have $n_{k+1}=n_k+1$. But then the growth exponent is $v=1$ which is not allowed in case (L2).
\end{proof}

Now we will consider how to obtain non-ergodic covers. We will make use of ideas
of Masur and Smillie which first appeared in \cite[Theorem 2.1]{MS91}. 
The criterion developed there for non-ergodicity carries over from the closed surface case to the infinite type case. 
We will state the (only slightly different) version from \cite[Theorem 3.3]{MT} in our setting. For the following theorem recall that the {\em vertical holonomy} of a curve $\gamma$ in a translation surface is the imaginary part of $\int_\gamma \alpha$. 

\begin{theorem}[Masur-Smillie \cite{MS91}]
\label{thm:Masur-Smillie}
Let $(S,\alpha)$ be a unit area translation surface of possibly infinite topological type,
and assume the translation flow is defined for all time almost everywhere.
Suppose there is a sequence of directions $\theta_n$ tending to the horizontal
and a sequence of partitions of the surface into two pieces, $S=A_n \sqcup B_n$,
so that the common boundary consists of a countable union
of line segments in direction $\theta_n$. 
Assume further that the absolute values of the vertical holonomies of the segments sum to $h_n<\infty$. Suppose also that:
\begin{enumerate}
\item[(i)] $\lim_{n \to \infty} h_n=0$.
\item[(ii)] There are constants $c$ and $c'$, so that $0<c<\mu(A_n)<c'<1$
for each $n$, where $\mu$ is Lebesgue measure on $(S_L,\alpha_L)$.
\item[(iii)] $\sum_{n=1}^\infty \mu(A_n \Delta A_{n+1})<\infty$, where $\Delta$ denotes symmetric difference.
\end{enumerate}
Then, the translation flow (horizontal straight-line flow) on $(S_L,\alpha_L)$ is not ergodic.
\end{theorem}
The proof as provided in \cite[Theorem 3.3]{MT} works in our setting, so we only explain the main ideas of the proof.
Through a measure theoretic argument, one can see that the set $A_\infty=\liminf A_n$ must be almost flow invariant in the sense that for any $t$ the symmetric difference of $A_\infty$ with its image under translation flow for time $t$ has Lebesgue measure zero. 
The set $A_\infty$ has measure bounded away from zero and one. An argument involving Fubini's theorem then produces an invariant set which is the same up to sets of Lebesgue measure zero which proves non-ergodicity. The arguments make no use of compactness or any other properties which do not hold in our setting.

\begin{proof}[Proof of statement {\em (L3)} of Theorem \ref{thm:main ladder}]
Let $A \in \SL(2,\R)$ and $h \in \Hom(\Gamma,\Z_2)$ determine a devious cover $A(\tilde S_h,\tilde \alpha_h)$ of $A(S_L,\alpha_L)$. We also make several other hypotheses. We consider the coding walk of the geodesic $g^t A \tilde V'$ in $O(2) \bs \SL_\pm(2,\R)/ \tilde V'$,
and consider its coding walk $\{n_k\}$. Statement (L3) concerns the case when the coding walk diverges quickly
in the sense that the growth exponent $v$ is smaller than $\varphi^2$. In light of Remark \ref{rem:reflection simplification}, we can assume that $\lim_{k \to \infty} n_k=-\infty$. By Lemma \ref{lem:deviousness}, we know that $\lim_{m \to +\infty} \tau_\ast^{m}(h)=\0$. 

We will prove that the translation flow on $A(\tilde S_h, \tilde \alpha_h)$ is not ergodic. To do this, we will find a sequence of partitions satisfying the criteria set out by Masur and Smillie (Theorem \ref{thm:Masur-Smillie}). Our partitions will be obtained by pulling back partitions
of the deformed surface $g^t A(\tilde S_h, \tilde \alpha_h)$ along a sequence of times tending to $+\infty$. 

It will be important for us to label the lifts of the basepoint to our double covers. We label the lifts
by the elements of $\Z_2$. Then when we deform our surfaces or apply linear maps, we will respect the labels of the basepoint. (Note that all double covers are normal and so admit a translation automorphism swapping the lifts of the basepoint.)

We introduce the following setup. Consider the region $F_0$ in $O(2) \bs \SL_\pm(2,\R) / \tilde V'$. 
Since the geodesic $g_t A V'$ is asymptotic to the convex core of $O(2) \bs \SL_\pm(2,\R) / V'$, there is a $d>0$ so that
this geodesic is contained entirely in the $d$-neighborhood of this convex core. 
Let $F_0' \subset F_0$ be the closed $d$-neighborhood of the lift of the convex core.
Choose an $M_0 \in \SL_\pm(2,\R)$ so that $M_0 \tilde V'$ lies in $F_0'$. Now consider another point $B \tilde V'$ of $F_0'$. Select a path
$\gamma_B$ in $F_0'$ joining $M_0 \tilde V'$  to $B \tilde V'$. We can lift this path in $\SL_\pm(2,\R)/ \tilde V'$
to a path in $\SL(2,\R)$ beginning at $M_0$. We let $E(B) \in \SL(2,\R)$ denote the endpoint of this path. 
Observe that $E(B)$ lies in the coset $B \tilde V'$. The selection of $E(B)$ is not quite canonical; it 
does not depend on the choice of $B$ from the class $B \tilde V'$, but it depends on the homotopy class of the curve $\gamma_B$. Since $F_0'$ is topologically an annulus, a different choice of path might give us a different $E(B)$, and difference is explained by monodromy around the cusp. Let $\Psi=\langle D(\psi)^2 \rangle \subset \SL(2,\R)$ be the monodromy group of $F_0$. We see that while $E(B)$ is not canonical, the coset $E(B) \Psi$ is. We will make choices only depending on $E(B)\Psi$. 

Now consider an affine image of our cover, 
$g^t A(\tilde S_{h}, \tilde \alpha_{h})$, and suppose that $g^t A \tilde V'$ lies in the region $F_n$. Select
an element $R \in V'$ which projects to $\tau^{-n} \in V'/\tilde V'$
so that $g^t A \tilde V' R^{-1}=g^t A R^{-1} \tilde V'$ lies
in the region $F_0'$. The action of $R$ on $\Hom(\Gamma, \Z_2)$ is given by $\delta(R)=\tau^{-n}_\ast$. 
Therefore,
$$g^t A(\tilde S_{h}, \tilde \alpha_{h})=g^t A R^{-1} (\tilde S_{\tau^{-n}_\ast(h)}, \tilde \alpha_{\tau^{-n}_\ast(h)}),$$
where equality denotes translation equivalence respecting marked points.
Now consider $E(g^t A R^{-1})$ as in the prior paragraph. We have $E(g^t A R^{-1})=g^t A R^{-1} \tilde R^{-1}$
for some $\tilde R \in \tilde V'$. Since $\tilde V'$ acts trivially on $\Hom(\Gamma,\Z_2)$, we see that
\begin{equation}
\label{eq:identification}
g^t A(\tilde S_{h}, \tilde \alpha_{h})=E(g^t A R^{-1}) (\tilde S_{\tau^{-n}_\ast(h)}, \tilde \alpha_{\tau^{-n}_\ast(h)})
\quad \text{in $\tilde \sO_{\Z_2}(S_L,\alpha_L)$.} 
\end{equation}
A partition of the cover $(\tilde S_{\tau^{-n}_\ast(h)}, \tilde \alpha_{\tau^{-n}_\ast(h)})$
will then pull back under this translation isomorphism followed by $g^{-t}$ to a partition of our original surface $A(\tilde S_{h}, \tilde \alpha_{h})$.
The cover $(\tilde S_{\tau^{-n}_\ast(h)}, \tilde \alpha_{\tau^{-n}_\ast(h)})$ can be thought of two copies of the infinite polygon used to define $(S_L,\alpha_L)$ labeled $\sR_0$ and $\sR_1$ and glued together according to $\tau^{-n}_\ast(h) \in \Hom(\Gamma,\Z_2)$. 
The index of the regions is determined by the index of the lift of the basepoint the region contains.
We will partition the cover into two subsurfaces $\sA$ and $\sB$. From our point of view, there are three types of vertical cylinders
on the cover $(\tilde S_{\tau^{-n}_\ast(h)}, \tilde \alpha_{\tau^{-n}_\ast(h)})$:
\begin{enumerate}
\item vertical cylinders that stay entirely in region $\sR_0$.
\item vertical cylinders that stay entirely in region $\sR_1$.
\item vertical cylinders that pass through both $\sR_0$ and $\sR_1$.
\end{enumerate}
Figure \ref{fig:ladder_surface_cylinders} illustrates the two regions and the cylinder types.
For any integer $j<0$, we have either $\tau^{-n}_\ast(h)(j)=0$ or $\tau^{-n}_\ast(h)(j)=1$. In the first
case, we get vertical cylinders of types (1) and (2) passing through lifts of the edge labeled $j$
of the base surface (as depicted in Figure \ref{fig:ladder_surface}),
and in the second case, there is a single vertical cylinder of type (3) which passes through both the lifts of edges labeled $j$. We define the surface $\sA_n \subset (\tilde S_{\tau^{-n}_\ast(h)}, \tilde \alpha_{\tau^{-n}_\ast(h)})$ to 
be the union of all vertical cylinders of type (1), and we define $\sB_n$ to be the union of cylinders of the remaining
two types. We can push this partition onto the surface $g^t A (\tilde S_{h}, \tilde \alpha_{h})$
by applying the affine map $E(g^t A R^{-1})$ and using the identification given in equation \ref{eq:identification}.
This gives us a partition $(\sA^t, \sB^t)$ of $g^t A (\tilde S_{h}, \tilde \alpha_{h})$.
Note that while $E(g^t A R^{-1})$ is not quite canonical as noted in the prior paragraph, it is well defined
up to a power of $D(\psi)^2$. Note that the partition of $g^t A (\tilde S_{h}, \tilde \alpha_{h})$
obtained is the same no mater which element of $E(g^t A R^{-1})\Psi$ we choose, because the action of $D(\psi^2) \in \tilde V'$ preserves each vertical cylinder.

\begin{figure}
\includegraphics[width=5in]{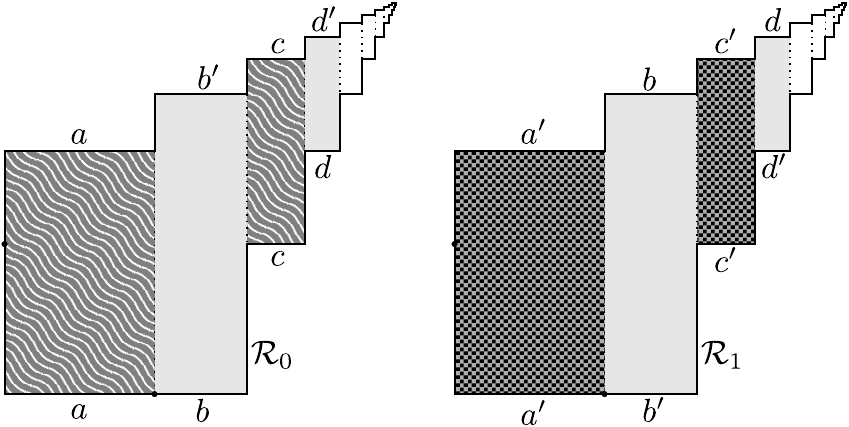}
\caption{A double cover of $(S_L,\alpha_L)$ with some edge gluings of horizontal edges labeled. There are two vertical cylinders of type (1) shown in gray with white waves, two of type (2) with a checkerboard pattern, and two of type (3) shown in light gray.}
\label{fig:ladder_surface_cylinders}
\end{figure}

For each $t$, we can pullback the partition $(\sA^t, \sB^t)$ of $g^t A (\tilde S_{h}, \tilde \alpha_{h})$
to a partition $g^{-t}(\sA^t, \sB^t)$ of $A (\tilde S_{h}, \tilde \alpha_{h})$.
This is actually a sequence of partitions, as we now explain.
When $g^t A \tilde V'$ lies in the region $F_n$, we have that $t$ lies in some interval
$(t_k, t_{k+1})$ where $n_k=n$. We claim the partition $g^{-t}(\sA^t, \sB^t)$ is independent of the choice
of $t$ from this interval. This is because the coset $g^{-t} E(g^t A R^{-1})\Psi$
is constant on the interval, since given a path to $g^t A R^{-1}$ we can get a path to 
$g^{t_\ast} A R^{-1}$ by traveling along the geodesic, and this change is canceled by the application of $g^{-t}$. 
Thus the partition only depends on $k$, and we define 
$\sA'_k=g^{-t}(\sA^t)$ and $\sB'_k=g^{-t}(\sA^t)$ for any $t$ satisfying $t_k<t<t_{k+1}$.
These subsurfaces partition $A(\tilde S_{h}, \tilde \alpha_{h})$.

It remains to check that our sequence of partitions $(\sA'_k, \sB'_k)$ satisfy the criterion of Masur and Smillie.

Consider statement (i), we need to show the total vertical holonomy of $\partial \sA'_k$ tends to zero as $k \to \infty$. This can be based on the observation that the total length of the boundary of the partitions $(\sA_n, \sB_n)$ of
the surface $(\tilde S_{\tau^{-n}_\ast(h)}, \tilde \alpha_{\tau^{-n}_\ast(h)})$ can be bounded from above independent of
$n$ and $\tau^{-n}(h)$. Indeed, for any double cover, the total length of all boundaries of all vertical cylinders is a finite constant independent of the cover. It is twice the corresponding constant for the base surface $(S_L,\alpha_L)$,
and the constant there is finite because the cylinders decay in size exponentially. We can choose $R \in V'$
as above so that $g^t A R^{-1} \tilde V'$ lies in the region $F_0'$ when $t_k<t<t_{k+1}$.
Select $t$ so that $g^t A R^{-1} \tilde V'$ lies in a fixed compact set of $F_0'$, which can be taken to be a neighborhood
of the convex core with a cusp neighborhood removed. Note then that by definition of $E$, the quantity
$E(g^t A R^{-1})\Psi$ is taken from a compact subset of $\SL(2,\R)/ \Psi$. So, in particular we get a uniform upper bound $L<\infty$ on the length of the boundary of the partition $(\sA^t, \sB^t)$ of the surface $g^t A(\tilde S_{h}, \tilde \alpha_{h})$. (This uses the identification of equation \ref{eq:identification}.) In particular,
the vertical component of this length is bounded from above by $L$. Pulling back via $g^{-t}$, we see 
that the vertical component of the length of the common boundary of $\sA'_k$ and $\sB'_k$ is bounded from above by $e^{-t} L$, which tends to zero as $k \to \infty$ because there is a uniform lower bound on $t_{k+1}-t_k$. This verifies statement (i).

To prove statement (ii), we need to make use of the assumptions that $\lim_{k \to \infty} n_k=+\infty$
and $\lim_{m \to \infty} \tau^m_\ast(h)=\0$. By Lemma \ref{lem:limits to zero},
there is an $M \geq 0$ and an integer $j$ so that $m>M$ implies that $P\big(\tau^m_\ast(h)\big)=j+m,$
where $P$ denotes proximity; see equation \ref{eq:proximity}.
Since $\{n_k\}$ tends to $+\infty$, there is a $K$ so that $k>K$ implies $P\big(\tau^{-n_k}_\ast(h)\big)>1$.
Since the proximity is larger than one, $\tau^{-n_k}_\ast(h)(-1)=0$. This means that there are two vertical cylinders in the surface $(\tilde S_{\tau^{-n_k}_\ast(h)}, \tilde \alpha_{\tau^{-n_k}_\ast(h)})$ passing through lifts of the edge labeled $-1$. One of these cylinders lies in $\sA_{n_k}$ and the other lies in $\sB_{n_k}$. Thus when $k>K$,
we obtain upper and lower bounds on the area of $\sA_{n_k}$. Since the partition $(\sA'_k, \sB'_k)$
was obtained by pulling back along an area preserving map, the same bounds hold here. (We only need to verify statement (ii) holds only for all but finitely many $k$.)

In order to understand (iii), we need to investigate how our partition changes
when the geodesic $g^t A \tilde V'$ passes from a region $F_n$ to $F_{n+1}$. Consider a time
$t$ so that $g^t A \tilde V'$ lies in the common boundary between $F_n$ and $F_{n+1}$. This is the point at which the partition changes. First consider this point $g^t A \tilde V'$ as part of $F_n$. Then,
we build a partition $(\sA_{n}, \sB_{n})$ of the cover $(\tilde S_{\tau^{-n}_\ast(h)}, \tilde \alpha_{\tau^{-n}_\ast(h)})$ as described above,
and apply $g^{-t} E(g^t A R_n^{-1})$, where $R_n \in V'$ is any Veech group element carrying $F_0$ to $F_n$,
to partition our original surface.
Here $E(g^t AR_n^{-1}) \in \SL(2,\R)$ is determined by lifting a path joining the $M_0 \tilde V'$ to $g^t AR_n^{-1} \tilde V'$ within $F_0'$. Now we consider what happens when we consider  $g^t A \tilde V'$ as part of $F_{n+1}$.
We construct a partition $(\sA_{n+1}, \sB_{n+1})$ of the cover $(\tilde S_{\tau^{-n-1}_\ast(h)}, \tilde \alpha_{\tau^{-n-1}_\ast(h)})$. We apply $g^{-t} E(g^t A R_{n+1}^{-1})$, where $R_{n+1}$ is some Veech group element carrying $F_0$ to $F_{n+1}$, to partition our original surface.
We need to understand the area of the symmetric difference of the pulled back partitions. It is equivalent to find the area of the symmetric difference between the partition $(\sA_{n}, \sB_{n})$ and the other partition of $(\tilde S_{\tau^{-n}_\ast(h)}, \tilde \alpha_{\tau^{-n}_\ast(h)})$ obtained as the image of $(\sA_{n+1}, \sB_{n+1})$ under
$$R_\ast=\big(g^{-t} E(g^t A R_n^{-1})\big)^{-1}g^{-t} E(g^t A R_{n+1}^{-1})=E(g^t A R_n^{-1})^{-1} E(g^t A R_{n+1}^{-1}).$$
The elements $E(g^t A R_n^{-1})$ and $E(g^t A R_{n+1}^{-1})$ are determined based on lifting paths $\gamma_{n}$ and $\gamma_{n+1}$
in $\SL_\pm(2,\R) /\tilde V'$ to paths $\tilde \gamma_n$ and $\tilde \gamma_{n+1}$ respectively. 
The value of $R_\ast \in \SL(2,\R)$ can then be determined by lifting the path obtained by first following $\gamma_1$
and then following the translated path $E(g^t A R_n^{-1})(\gamma_2)$ backward. The result is a path which passes once through the common boundary between $F_0$ and $F_1$, and joins the equivalence class of the $M_0$ to $D(\rho \circ \psi) \tilde V'$. The value of $R_\ast$ is the endpoint of this path lifted to $\SL(2,\R)$, which we see lies in the double coset 
$$\Psi D(\rho \circ \psi)\Psi \in \Psi \bs \SL(2,\R) / \Psi.$$
As the action of an element $\Psi$ does not change our partitions, we can choose to work with the simplest element
from our point of view, $R_\ast=D(\rho \circ \psi)$. 

We need to estimate the area of the symmetric difference
$\sA_n \Delta R_\ast(\sA_{n+1})$. 
Here, we interpret $R_\ast$ as an affine homeomorphism
$$R_\ast:(\tilde S_{\tau^{-n-1}_\ast(h)}, \tilde \alpha_{\tau^{-n-1}_\ast(h)}) \to (\tilde S_{\tau^{-n}_\ast(h)}, \tilde \alpha_{\tau^{-n}_\ast(h)})$$
which is characterized by its derivative and respects the labels of lifts of the basepoint. Observe that
the action of the matrix $R_\ast$ carries vertical cylinders to horizontal cylinders and preserves their widths. 
It also respects the labeling of the lifts of the basepoints.
Suppose that the proximity $P\big(\tau^{-n-1}_\ast(h)\big)=p$. For any integer $e<0$ with $-e<p$, 
we have $\tau^{-n-1}_\ast(h)(e)=0$, and thus the
vertical cylinder in $(\tilde S_{\tau^{-n-1}_\ast(h)}, \tilde \alpha_{\tau^{-n-1}_\ast(h)})$
starting in regions $\sR_0$ and passing through a lift of the edge labeled $e$ in Figure \ref{fig:ladder_surface}
lies in the subsurface $\sA_{n+1}$. Consider the union $U \subset \sA_{n+1}$ of all such vertical cylinders in $\sA_{n+1}$
with $-p<e<0$. We observe that because basepoint labels are preserved, $R_\ast(U) \subset \sA_n$, with each
vertical cylinder through $e$ being sent to a horizontal cylinder in $\sA_n$ passing through a vertical edge labeled $-e$. Let $\nu$ be the Lebesgue probability measure on $(\tilde S_{\tau^{-n}_\ast(h)}, \tilde \alpha_{\tau^{-n}_\ast(h)})$. We get the upper bound on the area of the symmetric difference
$$\nu \big(\sA_n \Delta R_\ast(\sA_{n+1})\big) \leq \nu(\sA_n)+\nu\big(R_\ast(\sA_{n+1})\big)-2 \nu(U) \leq 1 -2 \nu(U).$$
Observe that successively smaller cylinders in $(S_L,\alpha_L)$ decrease in area by a factor of $\varphi^2$.
It follows that there is a constant $\alpha>0$ so that $1-2 \nu(U)< \alpha \varphi^{-2 p}$. So, we see 
$$\nu \big(\sA_n \Delta R_\ast(\sA_{n+1})\big)< \alpha \varphi^{-2 p},$$
where $p=P\big(\tau^{-n-1}_\ast(h)\big)$ as above. 

Let $\mu$ be the normalized Lebesgue measure on the surface $A (\tilde S_h, \tilde \alpha_h)$. 
The prior paragraph gives an upper bound on $\mu(\sA'_k \Delta \sA'_{k+1})$ when $n_{k+1}=n_k+1$
in terms of the proximity $p(n_{k+1})=P\big(\tau^{-n_k-1}_\ast(h)\big)$.
A similar bound holds for the case when $n_{k+1}=n_k-1$; there is a constant $\alpha'>\alpha$ so that
$$\mu(\sA'_{k} \Delta \sA'_{k+1})<\alpha' \varphi^{-2 p(n_{k+1})}$$
regardless if $n_{k+1}$ equals $n_k+1$ or $n_k-1$. 
 
Now we will consider the total sum of the symmetric differences to verify statement (iii). Let
$V_N=\# \{k:~n_k=N\}$ for integers $N$. By assumption, each $V_N$ is finite, and $V_N=0$ for $N<N_0$
for some $N_0 \in \Z$. Then, 
$$\sum_k \mu(\sA'_{k} \Delta \sA'_{k+1}) \leq \sum_{N=N_0}^\infty V_N \varphi^{-2 p(N)}.$$
Now we incorporate the assumption that $\lim_{N \to +\infty} \tau^N_\ast(h)=\0$. 
From Lemma \ref{lem:limits to zero},
we obtain an $M \geq 0$ and an integer $j$ so that $N>M$ implies that $P\big(\tau^N_\ast(h)\big)=j+N.$
Let $X < \infty$ be the sum over the terms with $N \leq M$, then we see that
$$\sum_k \mu(\sA'_{k} \Delta \sA'_{k+1}) \leq X + \sum_{N=M+1}^\infty V_N \varphi^{-2 (j+N)}.$$
An application of the root test tells us this sum converges if $\limsup_{N \to \infty} V_N^{1/N}<\varphi^2$. 
Since this was a hypothesis, we have verified statement (iii). As an application of Theorem \ref{thm:Masur-Smillie},
we see that the translation flow on $A (\tilde S_h, \tilde \alpha_h)$ is not ergodic.
\end{proof}

\bibliographystyle{amsalpha}
\bibliography{bibliography}
\end{document}